\documentclass[12pt]{article}
\usepackage[utf8]{inputenc}

\usepackage{amsmath}
\usepackage{amsthm}
\usepackage{mathtools}
\usepackage{amsfonts}
\usepackage{amssymb, dsfont}
\usepackage{txfonts}
\usepackage{graphicx}
\usepackage[usenames]{color}
\usepackage[dvipsnames]{xcolor}
\usepackage[shortcuts]{extdash}
\usepackage{hyperref}
\usepackage{soul}
\usepackage{imakeidx}
\usepackage{amscd}
\usepackage[all,cmtip]{xy}
\usepackage{tikz-cd}
\usetikzlibrary{decorations.pathmorphing}
\usepackage{pb-diagram}
\usepackage{enumitem}

\newtheorem{theorem}{Theorem}
\newtheorem{lemma}{Lemma}
\newtheorem{proposition}{Proposition}
\newtheorem{fact}{Fact}
\newtheorem{corollary}{Corollary}
\newtheorem{problem}{Problem}
\newtheorem{remark}{Remark}
\newtheorem{notation}{Notation}
\newtheorem{definition}{Definition}

\theoremstyle{definition}
\newtheorem{example}{Example}

\def\N{{\mathbb{N}}}
\def\Z{{\mathbb{Z}}}

\def\MA{{\mathbb{A}}}
\def\MB{{\mathbb{B}}}
\def\MC{{\mathbb{C}}}
\def\MD{{\mathbb{D}}}
\newcommand{\nsA}{\widetilde{\mathbb A}}
\newcommand{\nsuA}{\widetilde{A}}
\newcommand{\nsB}{\widetilde{\mathbb B}}
\newcommand{\nsuB}{\widetilde{B}}

\def\F{{\mathcal{F}}}
\def\G{{\mathcal{G}}}

\def\Int{{\mathbf{Int}}}
\def\Tr{{\mathbf{Tr}}}

\newcommand{\Inj}{\mathrm{inj}}

\def\sk{\mathrm{sk}}

\def\PLS{{\mathcal{PLS}}}
\def\PDS{{\mathcal{PDS}}}
\def\DS{{\mathcal{DS}}}
\def\SC{{\mathcal{C}}}
\def\LS{{\mathcal{LS}}}
\def\AS{{\mathcal{AS}}}
\def\id{{\mathrm{id}}}
\def\term{{\,\mathrm{term}}}
\def\Rad{{\mathrm{Rad}}}
\def\LRad{{\mathrm{LRad}}}
\def\V{{\mathrm{V}}}
\def\group{{\mathrm{group}}}
\def\param{{\mathrm{par}}}

\def\AC{\mathcal {AC}}
\def\Id{{\mathbf{Id}}}

{}
{}
{}
{}

\title{Theory of Interpretations II. Categorical equivalence of projective logical geometries\footnote{The research was supported in accordance with the state task of the IM SB RAS, project FWNF-2026-0033.}}
\date{}

\author{Evelina Danyarova\footnote{Sobolev Institute of Mathematics} and Alexei Myasnikov\footnote{Stevens Institute of Technology}} 

\begin{document}

\maketitle

\begin{flushright}
{\small In memory of B.\,I.\,Plotkin, \\ the founding father of logical geometry}
\end{flushright}

\begin{abstract} 
We introduce projective logical geometry and prove that two algebraic structures are strongly bi-interpretable if and only if their categories of projective logical sets are equivalent relative to the class of interpretation functors, which is also equivalent to their categories of projective definable sets being equivalent relative to the class of translation functors. These constructions generalize two ideas of Boris Plotkin: the concept of geometric equivalence in universal algebraic geometry and the transition from universal algebraic geometry to logical geometry. Furthermore, our categorical approach offers a fresh perspective on the theory of interpretations, enabling us to establish a series of fundamental results using categorical methods.
\end{abstract}

\tableofcontents

\section{Introduction}\label{sec:intro} 

Boris Isaakovich Plotkin is regarded as one of the founders of {\em universal algebraic geometry}, a relatively new branch of mathematics that provides a systematic framework for investigating systems of equations, algebraic sets, coordinate algebras, and other objects familiar from classical algebraic geometry, but now in the setting of arbitrary universal algebras and algebraic structures. In numerous instances within this universal algebraic geometry, the main notions and theorems are carried over from classical algebraic geometry over a field, with the necessary modifications. One of the fundamentally new concepts that arises naturally in universal algebraic geometry is that of geometric equivalence. This notion, together with its generalizations~--- geometric similarity and geometric compatibility~--- was introduced by B.\,I.\,Plotkin as a tool for comparing the algebraic geometries of distinct algebraic structures. Here, two structures $\MA$ and $\MB$ of the same language $L$ are geometrically equivalent if every finite system of equations in $L$ has the same radicals over $\MA$ and $\MB$.  While the geometric similarity and compatibility of $\MA$ and $\MB$ are formulated in terms of their categories of algebraic sets: the former arises when these categories are isomorphic, while the latter is defined by the equivalence of the categories (with respect to appropriate functors). Each of these concepts provides a way to address Plotkin’s central question in universal algebraic geometry: under what conditions do two algebraic structures $\MA$ and $\MB$ share the same algebraic geometry? Numerous intriguing problems and results are known in this field; for a more detailed treatment, we refer the reader to~\cite{Plotkin4}.

Another direction in which B.\,I.\,Plotkin developed the ideas of universal algebraic geometry was the transition from universal algebraic geometry to {\em logical geometry}, or {\em algebraic geometry in first-order logic}. The basic idea of this transition is as follows. In algebraic geometry over $\MA$, the main object of study is an algebraic set, that is, the set of all solutions to a system of equations in the language $L$, where an equation is an atomic formula in $L$. In logical geometry, by contrast, a ``logical equation'' is an arbitrary first-order formula. Thus, by taking the sets of points satisfying arbitrary collections of first-order formulas (in a fixed finite tuple of variables), we obtain {\em algebraic sets in first-order logic}, or {\em logical sets}. Following the established pattern, one then develops the full algebraic-geometric apparatus in this more general, logical setting.

Plotkin’s fundamental question~--- when do two algebraic structures $\MA$ and $\MB$ in the same language $L$ have the same logical geometries?~--- also plays a central role in logical geometry. Once again, Plotkin introduces three levels of relationships between the logical geometries of algebraic structures: logical equivalence, logical similarity, and logical compatibility. Plotkin showed that two structures are logically equivalent if and only if they are isotypic~\cite{Plotkin10}, i.\,e., they realize the same types, which is a strengthening of elementary equivalence. This concept is very interesting in its own right and has inspired research on groups and other algebraic structures~\cite{Bunina1, Gvozdevsky, MR}. Logical similarity and logical compatibility of two algebraic structures are defined via isomorphism and, respectively, equivalence of their categories of logical sets.

In this paper, we develop Plotkin’s categorical approach to logical geometry and place it within a broader model-theoretic framework. To this end, we slightly modify his terminology and, more generally, examine the object of study from a somewhat different perspective. Plotkin’s approach to logical geometry is closely connected with the notion of a Halmos algebra. We will not reproduce here the full range of ideas and concepts from the works of Plotkin and his collaborators. Instead, we restrict attention to the category of algebraic sets in first-order logic and introduce it in a slightly different way, bypassing Halmos algebras.
In fact, we introduce a category that is larger than Plotkin’s. Just as in classical algebraic geometry, where one passes from affine varieties to projective varieties, we pass from the category of logical sets to the category of {\em projective logical sets}. This transition has model-theoretic underpinnings and reflects the need to incorporate imaginary elements. The category of projective logical sets contains Plotkin’s original category of logical sets as a subcategory, along with all other natural categories related to logical geometry. Thus, for a fixed algebraic structure $\MA$, we define 
\begin{itemize}
    \item the category $\PLS(\MA)$ of projective logical sets over $\MA$,
    \item its full subcategory $\LS(\MA)$ of logical sets over $\MA$,
    \item its full subcategory $\PDS(\MA)$ of projective definable sets over $\MA$,
    \item its full subcategory $\DS(\MA)$ of definable sets over $\MA$.
\end{itemize}
Next, we introduce two kinds of functors that connect such logical-geometric categories: {\em interpretation functors} and {\em translation functors}.

The model-theoretic framework in which we place Plotkin’s logical-geometric categories is the theory of interpretations. This paper is the second in a series on that theory; it relies heavily on the first paper~\cite{Th_int1} in its terminology, notation, and principal ideas. We prove two main results here. The first result states that an algebraic structure $\MA$ is interpretable in an algebraic structure $\MB$ if and only if there exists an interpretation functor from the category $\PLS(\MA)$ to the category $\PLS(\MB)$ (Theorem~\ref{th:functor2}). Secondly, we show that algebraic structures $\MA$ and $\MB$ are strongly bi-interpretable if and only if the categories $\PLS(\MA)$ and $\PLS(\MB)$ are equivalent relative to the class of interpretation functors, and also if and only if the categories $\PDS(\MA)$ and $\PDS(\MB)$ are equivalent relative to the class of translation functors (Theorem~\ref{th:bi-inter}). These results and their corollaries give a complete solution to Plotkin’s question concerning the logical compatibility of algebraic structures. Furthermore, this solution is maximally general, as it applies even when the algebraic structures $\MA$ and $\MB$ have different languages (signatures). We also introduce and study a special kind of bi-interpretation between structures $\MA$ and $\MB$, which is equivalent to an isomorphism between their projective logical categories. This is the so-called {\em syntactic isomorphism}, which can likewise occur between algebraic structures of different signatures.

Thus, strong bi-interpretability and syntactic isomorphism answer the fundamental question of under what conditions two algebraic structures $\MA$ and $\MB$ have the same projective logical geometries. At the same time, from the standpoint of the theory of interpretations, the methods of category theory provide powerful tools for studying interpretations and bi-interpretations of algebraic structures. Using these tools, we can establish several results that are crucial for interpretation theory, and that would be difficult to justify without passing to categories. These include the transitivity of homotopies, the associativity of interpretations, the transitivity of bi-interpretations, and others.

We conclude the article with isomorphism theorems for the big categories of all algebraic structures with interpretations, ${\bf AS}$, and all projective logical geometries, ${\bf PLG}$. In fact, these turn out to be isomorphic as $2$-categories (Theorem~\ref{th:2cat}). We also note that the logical nature and regular occurrence of such results were already indicated by N.\,Avni and C.\,Meiri~\cite[Remark~2.1]{AvniMeiri}.

On a philosophical level, our results bear a close relationship to Makkai duality concerning the connection between bi-interpretations of theories and equivalences of suitable categories~\cite{Aratake, DPM, Makkai, MakkaiReyes}. However, although interpretations and bi-interpretations between algebraic structures and between theories share many features, they are not reducible to one another. In the context of interpretations between algebraic structures, we distinguish three different kinds: absolute, regular, and with parameters~\cite{Th_int1, KhMS}, which are mutually nonequivalent. The most general notion is that of interpretation with parameters. If an algebraic structure $\MA$ is interpretable in an algebraic structure $\MB$ with parameters, then, in general, we cannot conclude that there exists an interpretation, for example, between their elementary theories ${\rm Th}(\MA)$ and ${\rm Th}(\MB)$. Conversely, a bi-interpretation of the theories ${\rm Th}(\MA)$ and ${\rm Th}(\MB)$ does not necessarily entail the existence of a bi-interpretation between the structures $\MA$ and $\MB$ themselves. Interpretations of algebraic structures often yield substantial information about the structures themselves, even when little is known about their theories; this is particularly significant in algebraic applications. Conversely, interpretations between theories provide a global perspective on the entire class of models of these theories. It would be interesting to develop a general framework that unifies these two seemingly different approaches.

{\bf Acknowledgments.} The authors gratefully acknowledge the organizers of the Israeli Science Foundation Workshop ``Algebra, Geometry, Dynamics'', held in part in memory of Boris I.\,Plotkin on the centenary of his birth. Discussions with colleagues at this conference~--- especially Eugene Plotkin~--- were invaluable in helping the authors substantially revise and refine this article.

\section{Preliminaries}\label{sec:prelim}

The model-theoretic method of interpretation plays a central role in our analysis. Consequently, this article requires a substantial preliminary section devoted to background material and notation. To avoid unnecessary duplication, we instead refer the reader to the authors' earlier paper~\cite{Th_int1}, of which the present work is a continuation. Throughout, we adopt the concepts and notation from~\cite[Subsections~2.1--2.4, 2.6--2.8, 3.1, 4.1, 4.2, 4.4]{Th_int1}.

Recall that for an algebraic structure $\MA=\langle A; L(\MA)\rangle$ and a subset $P\subseteq A$, we denote by $\MA_P$ the algebraic structure $\langle A; L(\MA)\cup P\rangle$ in the extended language. Note that if $c\in L(\MA)$ is a constant symbol and $c^\MA$ is its interpretation on the underlying set $A$ in $\MA$, then $c$ and $c^\MA$ are two different constant symbols in the language $L(\MA)\cup A$ with the same interpretation in $\MA_A$.

Let $h\colon X\to Y$ be a map. We always use denotation $h$ not only for $h$ itself, but also for maps $h^m\colon X^m\to Y^m$, $m\in \N$, such that $h(x_1,\ldots,x_m)=(h(x_1),\ldots,h(x_m))$.

\begin{definition}\label{def:definable_map}
Let $X\subseteq A^n$, $Y\subseteq A^s$ be sets and $\sim_X$, $\sim_Y$ be equivalence relations on $X$, $Y$, respectively. A map $F\colon (X/{\sim_X})^m\to (Y/{\sim_Y})^l$ is termed {\em definable} in $\MA=\langle A; L(\MA)\rangle$ (or $P$\=/definable, $P\subseteq A$), if the preimage in $\MA$ of the graph of $F$ is definable in $\MA$ ($P$\=/definable), i.\,e., the set 
\begin{multline*}
\{(a^1_1,\ldots,a^1_n,\:\ldots,\: a^m_1,\ldots,a^m_n,\: b^1_1,\ldots, b^1_s,\ldots,b^l_1,\ldots,b^l_s) \in A^{nm+sl}\mid  \\ 
F((a^1_1,\ldots,a^1_n)/{\sim_X},\:\ldots,\: (a^m_1,\ldots,a^m_n)/{\sim_X})=((b^1_1,\ldots,b^1_s)/{
\sim_Y},\ldots,(b^l_1,\ldots,b^l_s)/{
\sim_Y}),\\
(a^1_1,\ldots,a^1_n),\ldots, (a^m_1,\ldots,a^m_n)\in X, \: (b^1_1,\ldots,b^1_s),\ldots,(b^l_1,\ldots,b^l_s)\in Y\}
\end{multline*}
is definable in $\MA$ ($P$\=/definable).  Note that here $F$ may also be a function of arity $0$. Similarly, a relation $Q$ of arity $m$ on the quotient set $X/{\sim_X}$ is {\em definable} ($P$\=/definable) in $\MA$, if the {\em preimage in $\MA$ of its graph}, i.\,e., the full preimage of the set $Q$ in $A^{nm}$, is definable in $\MA$ ($P$\=/definable). Correspondingly, $F$ and $Q$ are $0$\=/{\em definable} in $\MA$, if their preimages  are $0$\=/definable in $\MA$.
\end{definition}

In all other similar situations when defining a notion $C$ that could be absolute, we may also refer to it as $0$\=/$C$, or simply as $C$ without parameters.

\begin{fact}\label{rem0}
Let $\MA=\langle A; L(\MA)\rangle$ be an algebraic structure, $X\subseteq A^m$, $Y\subseteq A^s$ be some sets and $\sim_X$, $\sim_Y$ be some equivalence relations on $X, Y$, and $F\colon (X/{\sim_X})^n\to (Y/{\sim_Y})^l$ be some map, $n,l\in\N$, and $P\subseteq A$. Then
\begin{enumerate}[label=\alph*)]
    \item if $F$ is $P$\=/definable in $\MA$, then $X$ is $P$\=/definable in $\MA$;
    \item if $F$ is $P$\=/definable in $\MA$ and surjective, then $Y$ is $P$\=/definable in $\MA$ too.
\end{enumerate}
\end{fact}

\begin{proof}
Let $\varphi(\bar x_1,\ldots,\bar x_n,\bar y_1,\ldots,\bar y_l)$ be a formula in the language $L(\MA)\cup P$, which defines the map $F$, $|\bar x_i|=m$,  $|\bar y_j|=s$. Then  $X=S_X(\MA_P)$, where $S_X(\bar x)=\exists \, \bar {\bar y}\; \varphi(\bar x,\ldots,\bar x,\bar {\bar y})$, so $X$ is $P$\=/definable.  Further, if $F$ is surjective, then $Y=S_Y(\MA_P)$, where $S_Y(\bar y)=\exists\,\bar {\bar x}\:\varphi(\bar {\bar x}, \bar y,\ldots,\bar y)$.  
\end{proof}

For the basic concepts of category theory, we refer to the books~\cite{BarrWells, MacLane}. 

\section{Translations of formulas and Reductions Theorems}\label{sec:RT}

Let $\MA=\langle A; L(\MA)\rangle$ and $\MB=\langle B; L(\MB)\rangle$ be algebraic structures and  $\Gamma\colon L(\MA)\to L(\MB)$ be an interpretation code:
\begin{equation*} 
\Gamma =  \{U_\Gamma(\bar x,\bar y), E_\Gamma(\bar x, \bar x^\prime,\bar y), c_\Gamma(\bar x, \bar y), f_\Gamma(\bar x_0, \bar x_1, \ldots,\bar x_{n_f},\bar y), R_\Gamma(\bar x_1,\ldots,\bar x_{n_R},\bar y) \mid c,f,R \in L(\MA)\},
\end{equation*}
where $|\bar x|=|\bar x^\prime|=|\bar x_i|=\dim\Gamma$, $|\bar y|=\dim_\param\Gamma$; and $c$ runs all constant symbols from $L(\MA)$, $f$ runs all functional symbols from $L(\MA)$, $R$ runs all predicate symbols from $L(\MA)$. The notation $U_\Gamma$, $E_\Gamma$, and $c_\Gamma,f_\Gamma, R_\Gamma$ is consistently used throughout the entire paper. 

In~\cite[Subsection~2.7]{Th_int1} we define the notion of the $\Gamma$\=/translation $\psi\to\psi_{\Gamma}$ of formulas in the language $L(\MA)$ into formulas in the language $L(\MB)$. This kind of translation depends on the code $\Gamma$ and has no connection to concrete interpretation $\MA\rightsquigarrow\MB$. 

Further when given an interpretation $(\Gamma,\bar p, \mu_\Gamma)\colon \MA\rightsquigarrow\MB$ we define the $(\Gamma,\bar p, \mu_\Gamma)$\=/translation $\psi\to\psi_{\Gamma,\bar p,\mu_\Gamma}$ of formulas in the language $L(\MA)\cup A$ into formulas in the language $L(\MB)\cup B$. Here we use a set of fixed tuples $\{\bar b\in \mu_\Gamma^{-1}(a)\mid a\in A\}$. In other words, when defining the $(\Gamma,\bar p, \mu_\Gamma)$\=/translation we use not exactly the coordinate map $\mu_\Gamma$, but some function, which sends elements $a\in A$ into tuples $\bar b\in \mu_\Gamma^{-1}(a)$.

\begin{remark}
In~\cite[Subsection~2.4]{Th_int1} we made a precise distinction between $(\Gamma,\bar p)$ and $(\Gamma,\bar p,\mu_\Gamma)$, referring to $(\Gamma,\bar p)$ as an \emph{interpretation} and to $(\Gamma,\bar p,\mu_\Gamma)$ as a \emph{coordinatization} of the interpretation $(\Gamma,\bar p)$. Since this paper primarily works with triples of the form $(\Gamma,\bar p,\mu_\Gamma)$, we will, for brevity, refer to such triples simply as interpretations.
\end{remark}

\subsection{Extended codes}\label{subsec:ext_code}

At this point, we need to define the type of $(\Gamma,\bar p,\mu_\Gamma)$\=/translations that an interpretation provides, but without being tied to the interpretation.

\begin{definition}\label{def:code}
 We refer to a triple $(\Gamma,\bar p,\gamma)$ as an {\em extended code}, if  
\begin{enumerate}[label=(T\arabic*)]
    \item\label{T1} $\Gamma\colon L(\MA)\to L(\MB)$ is an interpretation code;
    \item\label{T2} $\bar p$ is a tuple from $\MB$, such that $|\bar p|=\dim_\param\Gamma$;
    \item\label{T3} $\gamma\colon A\to B^n$ is a map, where $n=\dim\Gamma$.
\end{enumerate}  
For brevity, we will write $(\Gamma,\bar p,\gamma)\colon L(\MA)\cup A\to L(\MB)\cup B$ in this case.
\end{definition}

Let us define the $(\Gamma,\bar p, \gamma)$\=/translation $\psi\to \psi_{\Gamma,\bar p,\gamma}$ of formulas in the language $L(\MA)\cup A$ into formulas in the language $L(\MB)\cup B$ (without connection to an interpretation $\MA\rightsquigarrow\MB$). Take any formula $\psi(x_1,\ldots,x_m,a_1,\ldots,a_s)$ in  $L(\MA)\cup A$ with free variables $x_1,\ldots,x_m$ and constants $a_1,\ldots,a_s$ from $A$ and put
$$
\psi_{\Gamma,\bar p, \gamma}(\bar x_1,\ldots,\bar x_m)=\psi_\Gamma(\bar x_1,\ldots,\bar x_m,\gamma(a_1),\ldots,\gamma(a_s), \bar p),
$$
where $|\bar x_i|=n$ and $\psi_\Gamma(\bar x_1,\ldots,\bar x_m,\bar z_1,\ldots,\bar z_m,\bar y)$ is the $\Gamma$\=/translation of $\psi(x_1,\ldots,x_m,z_1,\ldots,z_s)$, $|\bar y|=\dim_{\param}\Gamma$.

\begin{definition}
We refer to a first-order formula $\psi$ as a {\em formula in the normal form}, if it is written in the prenex normal form where the matrix is a disjunction of conjunctions of unnested atomic formulas and their negations.
\end{definition}

Of course, every formula may be written in the normal form. In fact, we define the $\Gamma$\=/translation $\psi\to\psi_\Gamma$ and the $(\Gamma,\bar p,\gamma)$\=/translation $\psi\to\psi_{\Gamma,\bar p,\gamma}$ not for a formula $\psi$, but for the normal form of $\psi$. Since the normal form for a formula $\psi$ is not written uniquely, for correct work with translations, we conveniently consider them to translate normal forms, not formulas. 

\begin{remark}\label{remar:norm_form}
If $\psi$ is a formula in the normal form, then the formulas $(\exists\,x\:\psi)_\Gamma$ and $\exists\,\bar x \:(U_\Gamma(\bar x,\bar y)\wedge \psi_\Gamma)$ are first-order equivalent.
\end{remark}

Given an interpretation $\MA\rightsquigarrow\MB$, all such ambiguities disappear thanks to the Reduction Theorem~\cite[Corollary~3]{Th_int1}. Next, we want to expand to the maximum the sufficient conditions under which the Reduction Theorem holds.

\subsection{Reduction Theorems}\label{sec:IRT}

The Reduction Theorem in formulations~\cite[Theorem~5.3.2]{Hodges} or~\cite[Theorem~1]{Th_int1} is essentially a generalization of the following statement. Suppose that $\MA=\langle A; L(\MA)\rangle$ and $\nsA=\langle \nsuA; L(\MA)\rangle$ are algebraic structures and $\iota\colon \MA\to\nsA$ is an $L(\MA)$\=/isomorphism. Then
\begin{enumerate}[label=(EE)]
\item\label{EE} for any $L(\MA)$\=/formula $\psi(x_1,\ldots,x_s)$ and any elements $a_1,\ldots,a_s\in A$ one has 
$$
\MA\models\psi(a_1,\ldots,a_s) \ \iff \ \nsA\models \psi(\iota(a_1),\ldots,\iota(a_s)).
$$
\end{enumerate}
This statement is about the algebraic structures in a common language $L(\MA)$, while the Reduction Theorem states the same thing, but about algebraic structures in possibly different languages $L(\MA)$ and $L(\MB)$. If we read the above statement in reverse, we get the definition of an elementary embedding. Indeed, a map $\iota\colon A\to\nsuA$ is called an {\em elementary $L(\MA)$\=/embedding} of $\MA$ into $\nsA$, if the condition~\ref{EE} holds. In this section, we will ``read the Reduction Theorem in reverse'' in the same way.

\begin{theorem}[{\cite[The Reduction Theorem]{Th_int1}}]\label{RT} 
Suppose that $(\Gamma,\bar p, \mu_\Gamma)$ is an interpretation of an algebraic structure $\MA=\langle A; L(\MA)\rangle$ into an algebraic structure $\MB=\langle B; L(\MB)\rangle$. 
Then for any formula $\psi(x_1,\ldots,x_s)$ in the normal form in the language $L(\MA)$ and any elements $a_1,\ldots,a_s\in A$, $\bar b_1\in \mu_\Gamma^{-1}(a_1),\ldots,\bar b_s\in \mu_\Gamma(a_s)$ one has
\begin{equation*}
 \MA\models \psi(a_1,\ldots,a_s) \ \Longrightarrow \ \MB\models \psi_\Gamma(\bar b_1,\ldots,\bar b_s,\bar p);
\end{equation*}
and conversely if one has $\MB\models \psi_\Gamma(\bar b_1,\ldots,\bar b_s,\bar p)$ for some tuples $\bar b_1,\ldots,\bar b_s\in B^{\dim\Gamma}$, then $\bar b_1,\ldots,\bar b_s\in U_\Gamma(\MB,\bar p)$ and $\MA\models \psi(a_1,\ldots,a_s)$, where $a_i=\mu_\Gamma(\bar b_i)$, $i=1,\ldots,s$.
\end{theorem}

Intending to shift the emphasis from the coordinate map $\mu_\Gamma$ of the interpretation $(\Gamma,\bar p,\mu_\Gamma)$ to the function $\gamma\colon A\to U_\Gamma(\MB,\bar p)$, we formulate several important facts.

\begin{corollary}\label{cor:gamma}
Let $(\Gamma,\bar p,\mu_\Gamma)\colon \MA\rightsquigarrow\MB$ be an interpretation and $\gamma,\gamma^\prime\colon A\to U_\Gamma(\MB,\bar p)$ be functions, such that $\mu_\Gamma\circ\gamma=\mu_\Gamma\circ\gamma^\prime=\id_A$. Then for any formula $\psi(x_1,\ldots,x_m)$ in the language $L(\MA)\cup A$ one has $\psi_{\Gamma,\bar p,\gamma}(\MB_B)=\psi_{\Gamma,\bar p,\gamma^\prime}(\MB_B)$.
\end{corollary}

\begin{remark}\label{lemma:gamma}
Let $(\Gamma,\bar p)\colon \MA\rightsquigarrow\MB$ be an interpretation and $\mu_\Gamma,\mu_\Gamma^\prime$ be two coordinate maps of the interpretation $(\Gamma,\bar p)$. If $\gamma\colon A\to U_\Gamma(\MB,\bar p)$ is a function, such that $\mu_\Gamma\circ\gamma=\mu^\prime_\Gamma\circ\gamma=\id_A$, then $\mu_\Gamma=\mu_\Gamma^\prime$. Indeed, for any $\bar b\in U_\Gamma(\MB,\bar p)$ and $a=\mu_\Gamma(\bar b)$ one has $\mu_\Gamma(\gamma(a))=\mu_\Gamma(\bar b)$, i.\,e., $\MB\models E_\Gamma(\bar b,\gamma(a),\bar p)$. Therefore, $\mu^\prime_\Gamma(\bar b)=\mu^\prime_\Gamma(\gamma(a))=a$. 
\end{remark}

Below, we give the formulation of the Reduction Theorem in the form in which it is convenient for us to present it for use in this article.

\begin{theorem}[Extended Reduction Theorem]\label{ERT} 
Suppose that $\MA=\langle A; L(\MA)\rangle$, $\nsA=\langle \nsuA; L(\MA)\rangle$, $\MB=\langle B; L(\MB)\rangle$ and $\nsB=\langle \nsuB; L(\MB)\rangle$ are algebraic structures, and
\begin{enumerate}[label=(\arabic*)]
    \item $\iota\colon \MA\to\nsA$ is an elementary $L(\MA)$\=/embedding;
    \item $(\Gamma,\bar p,\mu_\Gamma)\colon\nsA\to\MB$ is an interpretation;
    \item $\gamma\colon \nsuA\to U_\Gamma(\MB,\bar p)$ is a map, such that $\mu_\Gamma\circ\gamma=\id_{\nsuA}$;
    \item $\kappa\colon \MB\to\nsB$ is an elementary $L(\MB)$\=/embedding; i.\,e., one has
\end{enumerate}
$$ 
\begin{tikzcd}
\MA\arrow[r,"\iota"]&\nsA\arrow[r,rightsquigarrow,"{\Gamma,\bar p,\mu_\Gamma}"] &[1cm]\MB \arrow[r,"{\kappa}"] & \nsB.
\end{tikzcd}
$$
Then
\begin{enumerate}[label=(RT)]
\item\label{item:RT} for any formula $\psi(x_1,\ldots,x_s)$ in the normal form in the language $L(\MA)$ and any elements $a_1,\ldots,a_s\in A$ one has
\begin{equation*}
 \MA\models \psi(a_1,\ldots,a_s) \ \iff \ \nsB\models \psi_\Gamma(\kappa(\gamma(\iota(a_1))),\ldots,\kappa(\gamma(\iota(a_s))),\kappa(\bar p))
\end{equation*}
\end{enumerate}
(or, equivalently, for any sentence $\psi$ in the normal form in the language $L(\MA)\cup A$  one has
$ \MA_A\models \psi$ if and only if $\nsB_{\kappa(B)}\models \psi_{\Gamma,\kappa(\bar p), \kappa\circ\gamma\circ\iota}$).
\end{theorem}

In this formulation, we've added preliminary and final elementary embeddings $\MA\preceq \nsA$ and $\MB\preceq\nsB$ to an interpretation $\nsA\rightsquigarrow\MB$. This is a trivial generalization of the standard formulation (note that, strictly speaking, the Extended Reduction Theorem is not a generalization of the Reduction Theorem, since in its formulation the preimages of all elements of $A$ are fixed by a single map $\gamma$, while in the Reduction Theorem this is not the case). We did this to invert the Reduction Theorem's result in the following symmetric way. 

\begin{theorem}[Inverse Reduction Theorem]\label{IRT} 
Let $\MA=\langle A; L(\MA)\rangle$ and $\MB=\langle B; L(\MB)\rangle$ be algebraic structures, and $(\Gamma,\bar p,\gamma)\colon L(\MA)\cup A\to L(\MB)\cup B$ be an extended code. Suppose that
\begin{enumerate}[label=(IRT)]
\item\label{item:IRT} for any formula $\psi(x_1,\ldots,x_s)$ in the normal form in the language $L(\MA)$ and any elements $a_1,\ldots,a_s\in A$ one has
\begin{equation*}
 \MA\models \psi(a_1,\ldots,a_s) \ \iff \ \MB\models \psi_\Gamma(\gamma(a_1),\ldots,\gamma(a_s),\bar p)
\end{equation*}
\end{enumerate}
(or, equivalently, for any sentence $\psi$ in the normal form in the language $L(\MA)\cup A$  one has
$ \MA_A\models \psi$ if and only if $\MB_B\models \psi_{\Gamma,\bar p, \gamma}$).
Then 
\begin{enumerate}[label=(\arabic*)]
    \item the $L(\MA)$\=/structure $\Gamma(\MB,\bar p)$ is well-defined and, in particular, there exists an interpretation $(\Gamma,\bar p)\colon \Gamma(\MB,\bar p)\rightsquigarrow\MB$;
    \item $\gamma(A)\subseteq U_\Gamma(\MB,\bar p)$;
    \item the map $\iota\colon A\to U_\Gamma(\MB,\bar p)/{\sim_\Gamma}$, such that $\iota(a)=\gamma(a)/{\sim_\Gamma}$ for all $a\in A$, is an elementary $L(\MA)$\=/embedding $\iota\colon \MA\to \Gamma(\MB,\bar p)$, i.\,e., one has
    \end{enumerate}
$$ 
\begin{tikzcd}
\MA\arrow[r,"{\iota}"] &\Gamma(\MB,\bar p) \arrow[r,rightsquigarrow,"{\Gamma,\bar p,\mu_\Gamma}"] &[1cm] \MB
\end{tikzcd}
$$
and $\mu_\Gamma\circ\gamma=\iota$.
\end{theorem}

Before proving the Inverse Reduction Theorem, let us recall Tarski\,--\,Vaught Test, writing it in a form that is convenient for us to use in this article.

\begin{fact}[Tarski\,--\,Vaught Test]\label{TV}
A map $\iota\colon A\to\nsuA$ is an elementary $L(\MA)$\=/embedding of an algebraic structure $\MA=\langle A; L(\MA)\rangle$ into an algebraic structure $\nsA=\langle \nsuA; L(\MA)\rangle$ if and only if 
\begin{enumerate}[label=(\arabic*)]
\item $\iota(c^{\MA})=c^{\nsA}$ for every constant symbol $c\in L(\MA)$;
\item $\iota(f^{\MA}(a_1,\ldots,a_{n_f}))=f^{\nsA}(\iota(a_1),\ldots,\iota(a_{n_f}))$ for every functional symbol $f\in L(\MA)$ and elements $a_1,\ldots,a_{n_f}\in A$;
\item for every $L(\MA)$\=/formula $\psi(x,x_1,\ldots,x_s)$ and any elements $a_1,\ldots,a_s\in A$ one has
$$
\nsA\models \exists\,x\:\psi(x,\iota(a_1),\ldots,\iota(a_s)) \ \Longrightarrow \MA\models\exists\, x\:\psi(x,a_1,\ldots,a_s).
$$
\end{enumerate}
\end{fact}

\begin{proof}[Proof of Inverse Reduction Theorem~\ref{IRT}]
Note that the condition~\ref{item:IRT}, in particular, means that 
for any $L(\MA)$\=/sentence $\psi$ if $\MA\models \psi$ then $\MB\models\psi_\Gamma(\bar p)$. So, take the identical interpretation $\Id_{L(\MA)}\colon \MA\rightsquigarrow\MA$. The admissibility conditions of $\Id_{L(\MA)}$ hold on $\MA$, i.\,e., $\MA\models \AC_{\Id_{L(\MA)}}$. Therefore, $\MB\models\AC_{\Gamma}(\bar p)$ (see~\cite[Remark~12]{Th_int1}). Hence, by~\cite[Proposition~1]{Th_int1}, there exists an $L(\MA)$\=/structure $\Gamma(\MB,\bar p)$. 

Further, we will use the condition~\ref{item:IRT} all the time. Take an element $a\in A$. Since $\MA\models (a=a)$, then $\MB\models U_\Gamma(\gamma(a),\bar p)\wedge E_\Gamma(\gamma(a),\gamma(a),\bar p)$, in particular, $\gamma(a)\in U_\Gamma(\MB,\bar p)$ and the map $\iota\colon A\to U_\Gamma(\MB,\bar p)/{\sim_\Gamma}$, such that $\iota(a)=\gamma(a)/{\sim_\Gamma}$ for all $a\in A$, is well-defined. We will skip further the conjuncts $U_\Gamma(\gamma(a),\bar p)$, $a\in A$, in $\Gamma$\=/translations of formulas, since this information is no longer new. Let us do Tarski\,--\,Vaught Test (Fact~\ref{TV}) for the map $\iota\colon A\to U_\Gamma(\MB,\bar p)/{\sim_\Gamma}$.  

Take a constant symbol $c\in L(\MA)$. Since $\MA\models (c=c^\MA)$, where $c^\MA\in A$, then $\MB\models c_\Gamma(\gamma(c^\MA),\bar p)$, i.\,e., $c^{\Gamma(\MB,\bar p)}=\gamma(c^\MA)/{\sim_\Gamma}=\iota(c^\MA)$. For a functional symbol $f\in L(\MA)$ and any elements $a_0,a_1,\ldots,a_{n_f}\in A$ if $\MA\models (f(a_1,\ldots,a_{n_f})=a_0)$, then $\MB\models f_\Gamma(\gamma(a_1),\ldots,\gamma(a_{n_f}),\gamma(a_0),\bar p)$, therefore, $f^{\Gamma(\MB,\bar p)}(\gamma(a_1)/{\sim_\Gamma},\ldots,\gamma(a_{n_f})/{\sim_\Gamma})=\gamma(a_0)/{\sim_\Gamma}$, i.\,e, $f^{\Gamma(\MB,\bar p)}(\iota(a_1),\ldots,\iota(a_{n_f}))=\iota(a_0)$. 

For an $L(\MA)$\=/formula $\psi(x, x_1,\ldots,x_s)$ and elements $a_1,\ldots,a_s\in A$ assume that $\psi$ has the normal form and $\Gamma(\MB,\bar p)\models \exists\,x\: \psi(x,\iota(a_1),\ldots,\iota(a_s))$. Take a tuple $\bar b\in U_\Gamma(\MB,\bar p)$, such that $\Gamma(\MB,\bar p)\models  \psi(\bar b/{\sim_\Gamma},\gamma(a_1)/{\sim_\Gamma},\ldots,\gamma(a_s)/{\sim_\Gamma})$. By~\cite[Corollary~2]{Th_int1}, one has $\MB\models \psi_\Gamma(\bar b, \gamma(a_1),\ldots,\gamma(a_s),\bar p)$. Therefore, $\MB\models \exists \,\bar x\: (U_\Gamma(\bar x,\bar p)\wedge\psi_\Gamma (\bar x,\gamma(a_1),\ldots,\gamma(a_s),\bar p))$, that means $\MB_B\models(\exists \,x\: \psi(x,a_1,\ldots,a_s))_{\Gamma,\bar p,\gamma}$ (see Remark~\ref{remar:norm_form}). Hence, $\MA\models \exists\,x\:\psi(a_1,\ldots,a_s)$. Thus, Tarski\,--\,Vaught Test passes and the  map $\iota$ is an elementary $L(\MA)$\=/embedding.
\end{proof}

The following result is one of the simple facts from the direction ``Theory of Interpretations and Elementary Theories''. This topic is covered in a later article in this series, but we need this simple fact from there right now.

\begin{lemma}\label{elem_emb}
Suppose $\kappa\colon \MB\to\nsB$ is an elementary $L(\MB)$\=/embedding, $\Gamma\colon L(\MA)\to L(\MB)$ is an interpretation code and $\bar p\in \MB$, $|\bar p|=\dim_\param\Gamma$. Then the algebraic $L(\MA)$\=/structure $\Gamma(\MB,\bar p)$ is well-defined if and only if the algebraic $L(\MA)$\=/structure $\Gamma(\nsB,\kappa(\bar p))$ is well-defined. In this case the map $\bar\kappa\colon U_\Gamma(\MB,\bar p)/{\sim_\Gamma}\to U_\Gamma(\nsB,\kappa(\bar p))/{\sim_\Gamma}$, which sends $\bar b/{\sim_\Gamma}$ to $\kappa(\bar b)/{\sim_\Gamma}$, is an elementary $L(\MA)$\=/embedding $\bar\kappa\colon \Gamma(\MB,\bar p)\to \Gamma(\nsB,\kappa(\bar p))$.
\end{lemma}

\begin{proof}
According to~\cite[Proposition~1]{Th_int1}, the algebraic $L(\MA)$\=/structure $\Gamma(\MB,\bar p)$ is well-defined if and only if one has $\MB\models\AC_\Gamma(\bar p)$, if and only if $\nsB\models\AC_\Gamma(\kappa(\bar p))$, if and only if the algebraic $L(\MA)$\=/structure $\Gamma(\nsB,\kappa(\bar p))$ is well-defined. Take any tuple $\bar b\in U_\Gamma(\MB,\bar p)$. One has $\kappa(\bar b)\in U_\Gamma(\nsB,\kappa(\bar p))$. Further, if $\bar b,\bar b^\prime\in U_\Gamma(\MB,\bar p)$ and $\MB\models E_\Gamma(\bar b,\bar b^\prime,\bar p)$, then $\nsB\models E_\Gamma(\kappa(\bar b),\kappa(\bar b^\prime),\kappa(\bar p))$, therefore, the map $\bar\kappa$ is well-defined. Let us do Tarski\,--\,Vaught Test (Fact~\ref{TV}) for the map $\bar\kappa$. For any constant symbol $c\in L(\MA)$ and $\bar b\in U_\Gamma(\MB,\bar p)$ one has $\MB\models c_\Gamma(\bar b,\bar p)$ if and only if $\nsB\models c_\Gamma(\kappa(\bar b),\kappa(
    \bar p))$, i.\,e., $\bar\kappa(c^{\Gamma(\MB,\bar p)})=c^{\Gamma(\nsB,\kappa(\bar p))}$. Similarly, one can produce for any functional symbol $f\in L(\MA)$. Suppose that for an $L(\MA)$\=/formula $\psi(x,x_1,\ldots,x_s)$ and elements $\bar b_1/{\sim_\Gamma},\ldots,\bar b_s/{\sim_\Gamma}\in \Gamma(\MB,\bar p)$, $\bar b/{\sim_\Gamma}\in \Gamma(\nsB,\kappa(\bar p))$ one has $\Gamma(\nsB,\kappa(\bar p))\models \psi(\bar b/{\sim_\Gamma}, \kappa(\bar b_1)/{\sim_\Gamma},\ldots,\kappa(\bar b_s)/{\sim_\Gamma})$. By~\cite[Corollary~2]{Th_int1}, it meams that $\nsB\models\psi_\Gamma(\bar b,\kappa(\bar b_1),\ldots,\kappa(\bar b_s),\kappa(\bar p))$. Therefore, one has $\MB\models \psi_\Gamma(\bar c,\bar b_1,\ldots,\bar b_s,\bar p)$ for some $\bar c\in \MB$, besides, by the Reduction Theorem~\ref{RT}, $\bar c\in U_\Gamma(\MB,\bar p)$. Thus, $\Gamma(\MB,\bar p)\models \psi(\bar c/{\sim_\Gamma},\bar b_1/{\sim_\Gamma},\ldots,\bar b_s/{\sim_\Gamma})$, as required. So, $\bar\kappa$ is an elementary embedding.
\end{proof}

\subsection{Composition of extended codes}\label{subsec:comp_codes}

In this subsection, we continue the discussion from~\cite[Subsection~4.1]{Th_int1}.

Suppose that $\MA=\langle A; L(\MA)\rangle$, $\MB=\langle B; L(\MB)\rangle$ and $\MC=\langle C; L(\MC)\rangle$ are algebraic structures;  
$(\Gamma,\bar p, \mu_\Gamma)\colon \MA\rightsquigarrow\MB$ and $(\Delta,\bar q, \mu_\Delta)\colon \MB\rightsquigarrow\MC$ are interpretations; and $(\Gamma\circ\Delta,(\bar {\bar p},\bar q), \mu_{\Gamma\circ\Delta})\colon \MA\rightsquigarrow\MC$ is their composition, where $\mu_{\Gamma\circ\Delta}=\mu_\Gamma\circ\mu_\Delta$ and $\bar{\bar p}\in \mu_\Delta^{-1}(\bar p)$. The following corollary from the Reduction Theorem is an important fact about the composition of codes and interpretations.

\begin{lemma}[{\cite[Lemma~11]{Th_int1}}]\label{ABC}
For any formula $\psi(x_1,\ldots,x_m)$ in $L(\MA)\cup A$ one has
$$
\psi_{\Gamma\circ\Delta,\,(\bar{\bar p},\bar q),\,\mu_{\Gamma\circ\Delta}}(\MC)= (\psi_{\Gamma,\bar p, \mu_\Gamma})_{\Delta,\bar q, \mu_\Delta}(\MC).
$$
\end{lemma}

Now we need to generalize this result to the case where there is no information about the interpretations $\MA\rightsquigarrow\MB$ and $\MB\rightsquigarrow\MC$.

Let $(\Gamma,\bar p, \gamma)\colon L(\MA)\cup A\to L(\MB)\cup B$ and $(\Delta,\bar q, \delta)\colon L(\MB)\cup B \to L(\MC)\cup C$ are extended codes. Then we refer to the extended code $(\Gamma\circ\Delta,(\delta(\bar p),\bar q), \delta\circ\gamma)\colon L(\MA)\cup A\to L(\MC)\cup C$ as the {\em composition} of the given extended codes. 

\begin{proposition}\label{ABC_ext}
Let $\MA=\langle A;L(\MA)\rangle$, $\MB=\langle B;L(\MB)\rangle$ and $\MC=\langle C;L(\MC)\rangle$ be algebraic structures; and $(\Gamma,\bar p, \gamma)\colon L(\MA)\cup A\to L(\MB)\cup B$ and $(\Delta,\bar q, \delta)\colon L(\MB)\cup B \to L(\MC)\cup C$ be extended codes, which satisfy to the condition~\ref{item:IRT} from Inverse Reduction Theorem~\ref{IRT}. Then their composition $(\Gamma\circ\Delta,(\delta(\bar p),\bar q), \delta\circ\gamma)\colon L(\MA)\cup A\to L(\MC)\cup C$ also satisfies to~\ref{item:IRT}. Moreover, for any formula $\psi(x_1,\ldots,x_m)$ in the normal form $L(\MA)\cup A$ one has
\begin{equation}\label{eq:ABC}
  \psi_{\Gamma\circ\Delta,\,(\delta({\bar p}),\bar q),\,\delta\circ\gamma}(\MC)= (\psi_{\Gamma,\bar p, \gamma})_{\Delta,\bar q, \delta}(\MC)  
\end{equation}
for any normal form of formula $\psi_{\Gamma,\bar p, \gamma}$.
\end{proposition}

\begin{proof}
By Inverse Reduction Theorem~\ref{IRT} the algebraic structures $\Gamma(\MB,\bar p)$ and $\Delta(\MC,\bar q)$ are well-defined and there exist elementary embeddings $\iota\colon \MA\to \Gamma(\MB,\bar p)$ and $\kappa\colon \MB\to \Delta(\MC,\bar q)$, such that one has the chain
$$ 
\begin{tikzcd}
\MA\arrow[r,"\iota"]& \Gamma(\MB,\bar p)\arrow[r,rightsquigarrow,"{\Gamma,\bar p,\mu_\Gamma}"] &[1cm] \MB \arrow[r,"\kappa"]& \Delta(\MC,\bar q)\arrow[r,rightsquigarrow,"{\Delta,\bar q,\mu_\Delta}"] &[1cm] \MC,
\end{tikzcd}
$$
where $\iota=\mu_\Gamma\circ\gamma$ and $\kappa=\mu_\Delta\circ\delta$. Denote by $\nsB=\langle \nsuB; L(\MB)\rangle$ the algebraic structure $\Delta(\MC,\bar q)$, which we will also consider as $L(\MB)\cup B$\=/structure, where symbols $b\in B$ are interpreted as $\kappa(b)$.
Let us apply Extended Reduction Theorem~\ref{ERT} to the segment $\MA\to\Gamma(\MB,\bar p)\rightsquigarrow\MB\to\nsB$, so we get that the condition~\ref{item:RT} holds. In other words, the condition~\ref{item:IRT} holds for the extended code $(\Gamma,\kappa(\bar p), \kappa\circ \gamma)\colon L(\MA)\cup A \to L(\MB)\cup B$ and algebraic structures $\MA$ and $\nsB$. Hence, by Inverse Reduction Theorems~\ref{IRT}, the $L(\MA)$\=/structure $\Gamma(\nsB, \kappa(\bar p))$ is well-defied and there exists an elementary $L(\MA)$\=/embedding $\iota^\prime\colon\MA\to \Gamma(\nsB, \kappa(\bar p))$ and the interpretation $(\Gamma,\kappa(\bar p),\mu_\Gamma^\prime)\colon \Gamma(\nsB, \kappa(\bar p)) \rightsquigarrow \nsB$, such that $ \iota^\prime=\mu^\prime_\Gamma\circ\kappa\circ\gamma$. 

Denote by $\nsA=\langle \nsuA; L(\MA)\rangle$ the algebraic structure $\Gamma(\nsB, \kappa(\bar p))$, which we will also consider as $L(\MA)\cup A$\=/structure, where symbols $a\in A$ are interpreted as $\iota^\prime(a)$. Then the composition of the interpretations $(\Gamma,\kappa(\bar p),\mu_\Gamma^\prime)\colon \nsA \rightsquigarrow \nsB$ and $(\Delta,\bar q,\mu_\Delta)\colon \nsB\rightsquigarrow\MC$ yields the interpretation $(\Gamma\circ\Delta,(\delta(\bar p),\bar q),\mu_\Gamma^\prime\circ\mu_\Delta)\colon \nsA \rightsquigarrow \MC$~\cite[Lemma~10]{Th_int1}. Therefore, we obtain the chain
\begin{equation}\label{eq:eq1}
\begin{tikzcd}
\MA\arrow[r,"{\iota^\prime}"] &\nsA \arrow[r,rightsquigarrow,"{\Gamma\circ\Delta,(\delta(\bar p),\bar q),\mu_\Gamma^\prime\circ\mu_\Delta}"] &[2cm] \MC.
\end{tikzcd}
\end{equation} 
Since  $\iota^\prime=\mu^\prime_\Gamma\circ\mu_\Delta\circ\delta\circ\gamma$, then for any $a\in A$ one has $\iota^\prime(a)\in (\mu^\prime_\Gamma\circ\mu_\Delta)^{-1}(\delta(\gamma(a)))$. As $\iota^\prime$ is an elementary embedding, there exists a map $\gamma^\prime\colon \nsuA\to U_{\Gamma\circ\Delta}(\MC,(\delta(\bar p),\bar q))$ such that $\mu^\prime_\Gamma\circ\mu_\Delta\circ\gamma^\prime=\id_{\nsuA}$ and $\gamma^\prime\circ\iota^\prime=\delta\circ\gamma$. Then we apply Extended Reduction Theorem~\ref{ERT} to the chain~\eqref{eq:eq1} and map $\gamma^\prime$. As result, the condition~\ref{item:RT} gives the condition~\ref{item:IRT} for the extended code $(\Gamma\circ\Delta,(\delta(\bar p),\bar q), \delta\circ\gamma)\colon L(\MA)\cup A\to L(\MC)\cup C$, as required.

To prove~\eqref{eq:ABC} take any formula $\psi(x_1,\ldots,x_m)$ in the language $L(\MA)\cup A$. 
Let us show that applying the result of Lemma~\ref{ABC} to the chain of interpretations
$$ 
\begin{tikzcd}
\nsA\arrow[r,rightsquigarrow,"{\Gamma,\kappa(\bar p), \mu^\prime_\Gamma}"] &[1cm]\nsB \arrow[r,rightsquigarrow,"{\Delta,\bar q,\mu_\Delta}"] &[1cm] \MC,
\end{tikzcd}
$$
we obtain what is required. Firstly, one has $\psi_{\Gamma\circ\Delta, (\delta(\bar p),\bar q),\mu^\prime_\Gamma\circ\mu_\Delta}=\psi_{\Gamma\circ\Delta, (\delta(\bar p),\bar q),\delta\circ\gamma}$. Secondly, take $\gamma^{\prime\prime}=\mu_\Delta\circ\gamma^\prime$ as a map $\gamma^{\prime\prime}\colon \nsuA\to U_\Gamma(\nsB,\kappa(\bar p))$, such that $\mu^\prime_\Gamma\circ\gamma^{\prime\prime}=\id_{\nsuA}$. Then $\gamma^{\prime\prime}\circ\iota^\prime=\mu_\Delta\circ\gamma^\prime\circ\iota^\prime=\mu_\Delta\circ\delta\circ\gamma=\kappa\circ\iota^\prime$, therefore, 
the $L(\MB)\cup B$\=/formula $\psi_{\Gamma,\bar p,\gamma}$ when considering as $L(\MB)\cup \nsuB$\=/formula is equal to $\psi_{\Gamma,\kappa(\bar p),\mu^\prime_\Gamma}$. Thirdly, $(\psi_{\Gamma,\bar p,\gamma})_{\Delta,\bar q,\delta}=(\psi_{\Gamma,\kappa(\bar p),\mu^\prime_\Gamma})_{\Delta,\bar q, \mu_\Delta}$. Thus, Lemma~\ref{ABC} does indeed allow us to draw the desired conclusion.
\end{proof}

\begin{corollary}\label{ABCD_ext}
    Suppose that $\MA=\langle A; L(\MA)\rangle$, $\MB=\langle B; L(\MB)\rangle$, $\MC=\langle C; L(\MC)\rangle$, $\MD=\langle D; L(\MD)\rangle$ are algebraic structures and $(\Gamma,\bar p,\gamma)\colon L(\MA)\cup A\to L(\MB)\cup B$, $(\Delta,\bar q,\delta)\colon L(\MB)\cup B\to L(\MC)\cup C$, $(\Sigma,\bar r,\sigma)\colon L(\MC)\cup C\to L(\MD)\cup D$ are extended codes, which satisfy to the condition~\ref{item:IRT} from Inverse Reduction Theorem~\ref{IRT}. Then for any formula $\psi(x_1,\ldots,x_m)$ in the normal form $L(\MA)\cup A$ one has
\begin{equation*}
\psi_{(\Gamma\circ\Delta)\circ\Theta,\,(\sigma(\delta({\bar p}),\bar q),\bar r),\,\sigma\circ\delta\circ\gamma}(\MD)= ((\psi_{\Gamma,\bar p, \gamma})_{\Delta,\bar q, \delta})_{\Theta,\bar r,\sigma}(\MD) = \psi_{\Gamma\circ(\Delta\circ\Theta),\,(\sigma(\delta({\bar p})),\sigma(\bar q),\bar r),\,\sigma\circ\delta\circ\gamma}(\MD)
\end{equation*}
for any normal forms of formulas $\psi_{\Gamma,\bar p, \gamma}$ and $(\psi_{\Gamma,\bar p, \gamma})_{\Delta,\bar q, \delta}$.
\end{corollary}

\section{Projective logical geometry}\label{sec:proj}

Let us fix an algebraic structure $\MA=\langle A; L(\MA)\rangle$. In this section, we show how the transition occurs from algebraic geometry over $\MA$ to logical geometry over $\MA$, and then from logical geometry over $\MA$ to projective logical geometry over $\MA$. 
The principal notions of {\em  algebraic geometry} over $\MA$ are set out in~\cite{DMR1, DMR2, DMR5, DMR6, Monograph, Plotkin1, Plotkin2, Plotkin3, Plotkin4}; and of {\em logical geometry} over $\MA$ in~\cite{Plotkin7, Plotkin8, Plotkin9, Plotkin6, Plotkin10, Plotkin5}.

We start this section by reminding you of some basic definitions from universal algebraic geometry, such as equation, algebraic set, term map, and the category of algebraic sets. Furthermore, we discuss the translation of these notions from universal algebraic geometry to logical geometry, and then to projective logical geometry. In our presentation, we will occasionally slightly modify the concepts and notations adopted in algebraic and logical geometry. This is necessary to ensure a reasonable and smooth transition to projective logical geometry.

In this section, we introduce the category of algebraic sets over $\MA$, but omit the category of coordinate algebras of algebraic sets over $\MA$ and other categories, which are studied in universal algebraic geometry. 
We do not describe here the role of categories in universal algebraic geometry at all, nor do we review results related to categories in universal algebraic geometry, nor do we demonstrate the complete transformation of these ideas in the transition to Plotkin's logical geometry. All of this is extremely interesting material, requiring a separate article.

\subsection{Algebraic geometry}\label{subsec:AG}

In the study of algebraic geometry over algebraic structure $\MA$, {\em equations} are understood as arbitrary atomic formulas $\psi$ in the language $L(\MA)$, i.\,e., formulas of the types $t_1=t_2$ or $R(t_1,\ldots,t_{n_R})$, where $t_i$ are terms in $L(\MA)$, and $R$ is a predicate from the language $L(\MA)$. A {\em system of equations} $S$ is any set of equations with variables in a finite set $\{x_1,\ldots,x_m\}$. 
An {\em algebraic set} over $\MA$ is any subset $X\subseteq A^m$, such that there exists a system of equations $S$ with variables in $\{x_1,\ldots,x_m\}$, such that 
\begin{equation}\label{eq:V(S)}
 X=\{(a_1,\ldots, a_m)\in A^m \mid \MA\models\psi (a_1,\ldots,a_m) \; \forall\, \psi \in S\}.   
\end{equation}
In this case, we write $X=\V_\MA(S)$. For any subset $X\subseteq A^m$ the {\em radical} $\Rad(X)$ is the set of all equations $\psi$ in $L(\MA)$ with variables in $\{x_1,\ldots,x_m\}$ such that $\MA\models\psi (a_1,\ldots,a_m)$ for all $(a_1,\ldots, a_m)\in X$. A subset $X\subseteq A^m$ is algebraic over $\MA$ if and only if $X=\V_\MA(\Rad (X))$.

One of the main problems of algebraic geometry over $\MA$ is the classification of algebraic sets over $\MA$ up to isomorphism. Isomorphism here is understood in relation to the category $\AS(\MA)$ of algebraic sets over $\MA$. Let's introduce this category.

For any non-empty sets $X\subseteq A^m$ and $Y\subseteq A^s$ over $\MA$ an {\em absolute term map} is a map $F\colon X\to Y$, for which there exists terms $t_1,\ldots,t_s$ in the language $L(\MA)$ and variables in $\{x_1,\ldots,x_m\}$, such that 
\begin{equation}\label{eq:abs_term_map}
  F(a_1,\ldots,a_m)=(t_1(a_1,\ldots,a_m),\ldots,t_s(a_1,\ldots,a_m))\quad \forall\,(a_1,\ldots,a_n)\in X.  
\end{equation}
Note that in universal algebraic geometry, absolute term maps are named just term maps. However, here it is convenient for us to change the name. All terms $t_1,\ldots,t_s$ are in the language $L(\MA)$, they have no parameters from $A$ other than those in the language $L(\MA)$, which explains the prefix ``absolute''. If $X$ and $Y$ are algebraic sets over $\MA$, then a tuple of terms $(t_1,\ldots,t_s)$ defines an absolute term map $F\colon X\to Y$ by the rule~\eqref{eq:abs_term_map} if and only if $\varphi(t_1,\ldots,t_s)\in \Rad(X)$ for any $\varphi\in\Rad(Y)$.

The category $\AS(\MA)$ consists of all algebraic sets over $\MA$ (from all affine spaces $A^m$, $m\in \N$) as objects and absolute term maps as morphisms. If the empty set $\emptyset$ is algebraic over $\MA$ it is an initial object in $\AS(\MA)$, i.\,e., there exists a unique arrow $\emptyset \to O$ from the empty set $\emptyset$ to any object $O$, and no arrows from a non-empty object to $\emptyset$. 

If the concepts of equation, algebraic set, radical, and category of algebraic sets entered universal algebraic geometry by analogy with classical algebraic geometry over a field, then Plotkin’s concept of geometric equivalence is unique to universal algebraic geometry. Let $\MA$ and $\MB$ be algebraic structures in the same language $L$. Geometric equivalence formalizes a profound question: when do $\MA$ and $\MB$ have the same algebraic geometries? So, $\MA$ and $\MB$ are called {\em geometrically equivalent} if for any $m\in \N$ and any system of equations $S$ in the language $L$ and variables $\{x_1,\ldots,x_m\}$ one has $\Rad(\V_\MA(S))=\Rad(\V_\MB(S))$.

\begin{theorem}[B.\,I.\,Plotkin \cite{Plotkin3}]\label{th:Plotkin1}
If algebraic structures $\MA$ and $\MB$ in a language $L$ are geometrically equivalent, then the categories $\AS(\MA)$ and $\AS(\MB)$ of algebraic sets over $\MA$ and $\MB$ are isomorphic.
\end{theorem}

The language, apparatus, and results of universal algebraic geometry are of a universal nature. For example, if we want to study coefficient-free algebraic geometry over a group $G$, then we will consider it as an algebraic structure in the language $L_\group=\{\,\cdot\,,\,^{-1},e\}$. If we want to study algebraic geometry over $G$ with coefficients in equations from some subgroup $H\leqslant G$ (for example, $H=G$), then we will take $L_\group\cup H$ as the language. All definitions and results of universal algebraic geometry are equally applicable to both the language $L_\group$ and the language $L_\group\cup H$.

\subsection{Logical geometry}\label{subsec:LG}

Logical geometry due to B.\,Plotkin is a natural generalization of algebraic geometry. We just start consider any first-order formulas as equations, not just atomic formulas. If $S$ is any set of first-order formulas in the language $L(\MA)$ and free variables $\{x_1,\ldots,x_m\}$, then the set $X=\V_\MA(S)$, defined  in~\eqref{eq:V(S)}, is called {\em algebraic set in first-order logic}. Further, we will refer to such sets as {\em absolute logical sets}. 
For any subset $X\subseteq A^m$ the {\em logical radical} $\LRad(X)$ is the set of all formulas $\psi(x_1,\ldots,x_m)$ in $L(\MA)$  such that $\MA\models\psi (a_1,\ldots,a_m)$ for all $(a_1,\ldots, a_m)\in X$. A subset $X\subseteq A^m$ is an absolute logical set over $\MA$ if and only if $X=\V_\MA(\LRad (X))$.

B.\,Plotkin introduced a category, which is a complete analogue of the category $\AS(\MA)$, but in logical geometry. The category $\LS_0^\term(\MA)$ consists of all absolute logical sets over $\MA$ (from all affine spaces $A^m$, $m\in \N$) as objects and absolute term maps as morphisms, and the empty set $\emptyset$ as an initial object. The category $\AS(\MA)$ is a full subcategory in $\LS^\term_0(\MA)$. Note that B.\,Plotkin uses different notations and names; the choice of our notations and names will become clear from the following discussion.

Algebraic structures $\MA$ and $\MB$ be a language $L$ are called {\em logically equivalent} if for any $m\in \N$ and any set of first-order formulas $S$ in the language $L$ and free variables $\{x_1,\ldots,x_m\}$ one has $\LRad(\V_\MA(S))=\LRad(\V_\MB(S))$.

\begin{theorem}[B.\,I.\,Plotkin \cite{Plotkin6}]\label{th:Plotkin2}
If algebraic structures $\MA$ and $\MB$ in a language $L$ are logically equivalent, then the categories $\LS^\term_0(\MA)$ and $\LS^\term_0(\MB)$ of absolute logical sets over $\MA$ and $\MB$ and absolute term maps as morphisms are isomorphic.
\end{theorem}

After Theorems~\ref{th:Plotkin1} and~\ref{th:Plotkin2} B.\,I.\,Plotkin became interested in the inversion of these results: what does isomorphism or equivalence of categories yield geometrically? There are already quite a lot of works on this topic within the framework of universal algebraic geometry.
 In fact, the authors' current work is devoted to answering these questions in relation to logical geometry.

Logical geometry is as flexible and universal as algebraic geometry. We can choose any language of interest and project the definitions and results of logical geometry onto it. However, there is a difference from algebraic geometry. In algebraic geometry, the most interesting case is usually the Diophantine case, when all the elements of the algebraic structure $\MA$ under study are added to the language, i.\,e., $L(\MA)=L(\MA)\cup A$. The role of the Diophantine case in logical geometry looks somewhat different from that in algebraic geometry. On the one hand, the pure problem of describing all logical sets in the Diophantine case degenerates, since every subset $X\subseteq A^m$ is logical in $\MA_A$(see Example~\ref{ex:1} below). On the other hand, the problem of describing logical sets up to isomorphism, i.\,e., the construction of a skeleton for the corresponding category, on the contrary, is deeply rooted in the specifics of an algebraic structure $\MA$. 

Furthermore, in logical geometry, it is convenient to consider the capacious Diophantine category as an enclosing category that contains many other important and interesting subcategories. In this sense, it is important to expand such an enclosing category to the maximum reasonable limit, so that it absorbs everything important to consider. Therefore, in our research, we expand Plotkin's logical category $\LS_0^\term(\MA)$ simultaneously in several ways: we (1)~add all the constants from $\MA$ and get $\LS^\term(\MA)$; (2)~we expand the class of morphisms, as term maps, to the class of morphisms as formula maps (see Definition~\ref{def:mor}) and get $\LS(\MA)$; (3)~we move from logical sets to projective logical sets (see Definition~\ref{def2}), and get the category $\PLS(\MA)$. To introduce the category of projective logical sets $\PLS(\MA)$ over $\MA$, we start with logical sets with parameters.
 
\subsection{Logical sets with parameters}\label{subsec:LS}

For a finite set $\{x_1,\ldots,x_m\}$ of variables and a subset $P\subseteq A$ we denote by ${\bf F}_{L(\MA)\cup P}(x_1,\ldots,x_m)$ (or briefly, by ${\bf F}_{L(\MA)\cup P}(\bar x)$) the set of all first-order formulas in the language $L(\MA)\cup P$ with free variables $\{x_1,\ldots,x_m\}$.

\begin{definition}\label{def1}
We say that a subset $X\subseteq A^m$ is {\em a logical set over $\MA$ with parameters in $P\subseteq A$}, if there exists a set $S\subseteq  {\bf F}_{L(\MA)\cup P}(x_1,\ldots,x_m)$ such that 
$$
X=\{(a_1,\ldots, a_m)\in A^m \mid \MA_P\models \psi (a_1,\ldots,a_m) \; \forall\, \psi \in S\}.
$$
In this case, we will say that the {\em system $S$  induces} the set $X$ and write $X=\V_\MA(S)$. We refer to $X$ as {\em logical set over $\MA$}, if it is a logical set over $\MA$ with parameters in $A$. If $\alpha\leqslant \beta$ are infinite cardinals, such that $|P|<\alpha$ and $|S|<\beta$, then we will also say that $X$ is an $(\alpha,\beta)$\=/logical set. The symbol $\infty$ in place of $\alpha$ or $\beta$ means that there are no restrictions for $|P|$ or $|S|$.
\end{definition}

In Definition~\ref{def1}, if $P=\emptyset$, then $X$ is called an absolute logical set (as in Subsection~\ref{subsec:LG}). We refer to an absolute logical set as $0$\=/logical set, or $(0,\beta)$\=/logical set, if $|S|<\beta$. For example, in the algebraic structure $\langle \Z; +, \cdot, 0, 1\rangle$ every element $z\in \Z$ is $0$\=/definable, therefore, all logical sets in $\Z$ are absolute.  

By definition, the class of all $(\omega,\omega)$\=/logical sets over $\MA$ is the class of all definable sets over $\MA$. The class of all $(\omega,\infty)$\=/logical sets over $\MA$ is the class of all logical sets over $\MA$ with a finite set of parameters. If an algebraic structure $\MA$ is finitely generated, then every logical set over $\MA$ is an $(\omega,\infty)$\=/logical set.

Note that in model theory, the term {\em type-definable set} is known, which is very close to our notion of logical sets. However, in different works it can be introduced with minor differences, which relate to the source language $L$, an $L$\=/structure $\MA$, the number of parameters used, the number of definable sets in the intersection, and so on~\cite{Dries, GarciaWagner, HartShami, HrushovskiPeterzilPillay, Johnson, Marikova, Milliet, Pillay1, Pillay2, Yaacov1, Yaacov2}. Thus, the concept of a type-definable set in small details may not match our notion of a logical set. 

\begin{notation}
Note that a given logical set $X$ over $\MA$ may be induced by several systems. However, we will write $S_X$ for some system, which induces $X$, and $P_X$ for its parameters. If $S_X$ is finite, i.\,e., the set $X$ is definable, then we always assume that $S_X$ is a single formula, which defines $X$. 
\end{notation}

\begin{example}\label{ex:1}
Every subset $X\subseteq A^m$ is a logical over $\MA$. Indeed,  $S_X=\{ \neg((x_1=a_1)\wedge\ldots\wedge(x_m=a_m)) \mid (a_1,\ldots,a_m)\in A^m\setminus X\}$. 
\end{example}

Since all subsets in $\MA$ are logical, it is interesting to consider just $(\alpha,\beta)$\=/logical sets with parameters in $P$ for significant cardinals $\alpha\leqslant\beta$ and set $P\subseteq A$. 

Recall that all definable sets are among the logical ones. The empty set $\emptyset$ is an absolutely definable set over $\MA$, since $\V_\MA(\{x\ne x\})=\emptyset$; the affine space $A^m$ is also absolutely definable over $\MA$, since $A^m=\V_\MA(\{x_1=x_1\wedge\ldots\wedge x_m=x_m\})$. The class of definable sets over $\MA$ is closed under complements, finite intersections, unions, and direct products. Let us list some trivial facts about logical sets.

\begin{fact}
Let $X$, $Y$ and $X_i$, $i\in I$, be $(\alpha,\beta)$\=/logical sets over $\MA$. Then
\begin{enumerate}[label=\alph*)]
\item $X\times Y$ is an $(\alpha,\beta)$\=/logical set over $\MA$: 
$$
S_{X\times Y}(\bar x,\bar y)=S_X(\bar x)\cup S_Y(\bar y),\quad P_{X\times Y}=P_X\cup P_Y;
$$
\item if $X, Y\subseteq A^m$, then $X\cup Y$ is an $(\alpha,\beta)$\=/logical set over $\MA$: 
$$
S_{X\cup Y}(\bar x)=\bigcup_{\phi\in S_X, \; \psi\in S_Y}\{\phi(\bar x)\vee\psi(\bar x)\},\quad P_{X\cup Y}=P_X\cup P_Y;
$$
\item if $X_i\subseteq A^m$, $i\in I$, then $\bigcap_{i\in I}X_i$ is an $(\alpha,\beta)$\=/logical set over $\MA$: 
$$
S_{\bigcap X_i}(\bar x)=\bigcup_{i\in I} S_{X_i}(\bar x),\quad P_{\bigcap X_i}=\bigcup_{i\in I} P_{X_i}.
$$
\end{enumerate}
\end{fact}

\subsection{Formula equivalencies}\label{subsec:for_eq}

We define projective logical sets as quotient-sets of logical sets by special formula equivalence relations. 

\begin{definition}\label{def:formula_eq}
Let $X\subseteq A^m$ be a non-empty logical set over $\MA$ and $\sim_X$ an equivalence relation on $X$. We name $\sim_X$ {\em a formula equivalence on $X$} if there exists a formula $E$ in the language $L(\MA)\cup P$, $P\subseteq A$, and variables $\{x_1,\ldots,x_m\}\sqcup \{x_1^\prime,\ldots,x_m^\prime\}$ such that
\begin{enumerate}[label=(\arabic*)]
\item\label{eq1}  for all  $\bar a,\bar a^\prime\in X$ one has $\bar a\sim_X \bar a^\prime$ if and only if  $\MA_P\models E(\bar a,\bar a^\prime)$;
\item\label{eq2} and $X$ is {\em closed under $E$}, i.\,e., for any $\bar a,\bar a^\prime\in A^m$ the conditions $\MA_P\models E(\bar a,\bar a^\prime)\vee E(\bar a^\prime,\bar a)$ and $\bar a\in X$ imply that $\bar a^\prime\in X$. 
\end{enumerate} 
In this case, we will say that the formula $E$ {\em induces} the equivalence relation $\sim_X$ on $X$. If $P=\emptyset$ for some such formula $E$, then we name $\sim_X$ an {\em absolute} formula equivalence. We name $\sim_X$ {\em homogeneous} if additionally
\begin{enumerate}[label=(\arabic*)]\addtocounter{enumi}{2}
\item\label{eq3} $E$ defines an equivalence relation on $A^m$, i.\,e., if there exists $E$, such that \ref{eq1}--\ref{eq3} hold.
\end{enumerate}
\end{definition}

\begin{notation}
Note that for a logical set $X$, there may be several formula equivalences $\sim_X$ on $X$, and for a given formula equivalence $\sim_X$ on $X$, there may be several formulas $E$ that induce it. However, we will use the notation $\sim_X$ for some of the formula equivalencies and $E_X$ for some formula, which induces $\sim_X$.
\end{notation}

We do not say that $E_X$ defines $\sim_X$, because $\sim_X$ may not be a definable relation. We'll discuss the relationship between these verbs in more detail later in this subsection. For now, note that even if the relation $\sim_X$ is definable, it's sometimes more convenient to work with a formula $E_X$ that merely induces it.

\begin{example}\label{ex:Em}
Let $E_m$ be the formula $(x_1=x_1)\wedge\ldots\wedge(x_m=x_m)$. The formula $E_m$ induces the identity relation $=$ on every non-empty logical set $X\subseteq A^m$; and $E_m$ does not define the relation $=$ on $X$, if $X\ne A^m$. It is clear that the identity relation $=$ is homogeneous.
\end{example}

Suppose that $X$ and $Z$ are non-empty logical sets over $\MA$ and $Z\subseteq X$. If $\sim_X$ is a formula equivalence on $X$, then the restriction $\sim_Z$ of $\sim_X$ on $Z$ is an equivalence relation. Sometimes we will denote $\sim_Z$ just by $\sim_X$. We will say that $Z$ is {\em closed under $\sim_X$}, if for any tuples $\bar a,\bar a^\prime\in X$, if $\bar a\in Z$ and $\bar a\sim_X \bar a^\prime$, then $\bar a^\prime\in Z$. So, if $Z$ is closed under $\sim_X$, then $\sim_X$ is a formula equivalence on $Z$. Converse statement is not true, i.\,e., it is possible that $Z$ is not closed under $\sim_X$, but the restriction of $\sim_X$ on $Z$ is a formula equivalence, for example, if $Z$ is definable. Thus, if $\sim_X$ is a homogeneous formula equivalence on $X\subseteq A^m$ and it is the restriction of formula equivalence $\sim$ on $A^m$ to $X$, then $X$ is closed under $\sim$.

Let $E_1$ and $E_2$ be formulas in $L(\MA)\cup A$, which induce two formula equivalencies $\sim_1$ and $\sim_2$ on a non-empty logical set $X$ over $\MA$. We will say that $\sim_1$ is {\em finer} then $\sim_2$, or $\sim_2$ is {\em coarser} then $\sim_1$ (correspondingly, $E_1$ is {\em finer} then $E_2$ and $E_2$ is coarser then $E_1$) and write $E_1\subseteq E_2$, if $\MA_A\models E_1(\bar a,\bar a^\prime)$ implies $\MA_A\models E_2(\bar a,\bar a^\prime)$ for all $\bar a,\bar a^\prime\in X$. For instance, the identity relation $=$ on $X$ is finer than any other formula equivalence $\sim_X$ on $X$.

{\bf Definability of formula equivalences.} 
Let us formulate several simple facts that are important for us to fix.

\begin{fact}
Let $X\subseteq A^m$ be a non-empty logical set over $\MA$ and $\sim_X$ be a formula equivalence on $X$, induced by a formula $E_X$ in $L(\MA)\cup A$. Then the following conditions are equivalent:
\begin{enumerate}[label=(\arabic*)]
\item the formula $E_X$ defines $\sim_X$;
\item for all $\bar a,\bar a^\prime\in A^m$ the condition $\MA_A\models E_X(\bar a, \bar a^\prime)$ implies that $\bar a,\bar a^\prime\in X$.
\end{enumerate}
\end{fact}

\begin{fact}
Let $X\subseteq A^m$ be a non-empty logical set over $\MA$ and $\sim_X$ be a formula equivalence on $X$, induced by a formula $E_X$ in the language $L(\MA)\cup P$, $P\subseteq A$. Then $\sim_X$ is $P$\=/definable if and only if $X$ is $P$\=/definable.
\end{fact}

\begin{proof}
If a formula $S_X(\bar x)$ defines $X$, then the formula $E(\bar x, \bar x^\prime)=S_X(\bar x)\wedge S_X(\bar x^\prime)\wedge E_X(\bar x,\bar x^\prime)$ defines the equivalence relation $\sim_X$. And, conversely, if a formula $E(\bar x,\bar x^\prime)$ defines $\sim_X$, then the formula $E(\bar x,\bar x)$ defines $X$.
\end{proof}

\begin{fact}\label{rem2}
Let $X\subseteq A^m$ be a non-empty definable set over $\MA$ and $\sim_X$ an equivalence relation on $X$. Then the following conditions are equivalent:
\begin{enumerate}[label=\arabic*)]
\item $\sim_X$ is definable;
\item $\sim_X$ is a formula equivalence, i.\,e., for some formula $E$ in $L(\MA)\cup A$  items~\ref{eq1} and~\ref{eq2} in Definition~\ref{def:formula_eq} hold;
\item for some formula $E$ in $L(\MA)\cup A$ just item~\ref{eq1} in Definition~\ref{def:formula_eq} hold.
\end{enumerate}
Moreover, any formula equivalence $\sim_X$ on $X$ is homogeneous.
\end{fact}

\begin{proof}
Let $X=\V_\MA(S_X)$. 
If a formula $E$ satisfies item~\ref{eq1} in Definition~\ref{def:formula_eq} then $E_X(\bar x,\bar x^\prime)=S_X(\bar x)\wedge S_X(\bar x^\prime)\wedge E(\bar x,\bar x^\prime)$  defines and induces the equivalent relation $\sim_X$, i.\,e., $\sim_X$ is definable and it is a formula equivalence. 
At the same time, the formula
$$
E^{\prime}_X(\bar x, \bar x^\prime)= (S_X(\bar x)\wedge S_X(\bar x^\prime) \longrightarrow E(\bar x,\bar x^\prime)) \,\wedge\,(\neg S_X(\bar x)\vee \neg S_X(\bar x^\prime) \longrightarrow (\bar x=\bar x^\prime))
$$
introduces an equivalence relation on $A^m$, such that the set $X$ is closed under $E^{\prime}_X$, wherein $E_X$ and $E^{\prime}_X$ coincide on $X$. Thus, $\sim_X$ is homogeneous.
\end{proof}

\subsection{Projective logical sets with parameters}\label{subsec:proj_sets}

Now we are going to introduce projective logical sets.

\begin{definition}\label{def2}
Let $X$ be a non-empty logical set over $\MA$ with parameters in $P\subseteq A$ and $\sim_X$ be a formula equivalence on $X$ induced by a formula $E_X$ in the language $L(\MA)\cup P$. Then we call the quotient-set $X/{\sim_X}$ a {\em projective logical set over $\MA$ with parameters in $P$} (or just {\em projective logical set over $\MA$} if $P=A$). If $X$ is a definable set over $\MA$, then $X/{\sim_X}$ is called {\em a projective definable set} over $\MA$. If $X$ and $\sim_X$ are absolute then $X/{\sim_X}$ is also called {\em absolute}. If $X$ is an $(\alpha,\beta)$\=/logical set, then $X/{\sim_X}$ is also called an {\em $(\alpha,\beta)$\=/projective logical set}.
\end{definition}

\begin{notation}
 It is convenient for us to use the notation $X/{\sim_X}$ for a projective logical set with the initial logical $X$, even when there exist several formula equivalences on $X$. Elements form the set $X/{\sim_X}$ we will denote by $\bar x/{\sim_X}$, where $\bar x\in X$.
\end{notation}

\begin{remark}
The identity relation $=$ is a particular case of a formula equivalence. Therefore, any non-empty logical set $X$ is a projective logical set as well. And any non-empty definable set is a projective definable set.
\end{remark}

Objects such as projective logical sets or those close to them are also found in the literature under the name {\em quotients of type-definable sets}~\cite{Dries, ElPet, HaskelPillay, Rzepecki}. Moreover, while we consider only formula (definable) equivalence relations $\sim_X$, the listed papers also study type-definable relations. For the needs of logical geometry, such a degree of generalization seems redundant. Projective definable sets are also known as {\em imaginary sets}~\cite{AvniMeiri}.

Next, we are interested in understanding which subsets of $(\alpha,\beta)$\=/projective logical sets are $(\alpha,\beta)$\=/projective logical sets. Suppose that $X\subseteq A^m$ is a non-empty $(\alpha,\beta)$\=/logical set over $\MA$ and $Z$ is a non-empty $(\alpha,\beta)$\=/logical subset of $X$, i.\,e., there exists a system $S_Z\subseteq {\bf F}_{L(\MA)\cup P}(x_1,\ldots,x_m)$, $P\subseteq A$, $|P|<\alpha$, $|S_Z|<\beta$, such that $Z=\V_\MA(S_Z)$. Assume that $\sim_X$ is a formula equivalence on $X$, and $Z$ is closed under $\sim_X$. Then the restriction $\sim_Z$ of $\sim_X$ on $Z$ is a formula equivalence on $Z$ and for every tuple $\bar a\in Z$ the coset $\bar a/{\sim_Z}$ coincides with $\bar a/{\sim_X}$. So in this case we will write $Z/{\sim_Z}\subseteq X/{\sim_X}$. Conversely, let $W$ be a subset of the projective logical set $X/{\sim_X}$ over $\MA$. Then $W$ is an $(\alpha,\beta)$\=/projective logical set over $\MA$ if and only if the set $Z=\{\bar a\in X\mid \bar a/{\sim_X}\in W\}$ is an $(\alpha,\beta)$\=/logical set over $\MA$, so that $W=Z/{\sim_X}$. For instance, if $\sim_X$ is a homogeneous a formula equivalence on $X$ and it is the restriction of formula equivalence $\sim$ on $A^m$ to $X$, then $X/{\sim_X}\subseteq A^m/{\sim}$.

Let us note a few more trivial facts.

\begin{fact}\label{fact:times}
Let $X/{\sim_X}$ and $Y/{\sim_Y}$ be $(\alpha,\beta)$\=/projective logical sets over $\MA$ (maybe $X=Y$), and $E_X(\bar x,\bar x^\prime)$ induces $\sim_X$, while $E_Y(\bar y,\bar y^\prime)$ induces $\sim_Y$. Then 
\begin{enumerate}[label=(\alph*)]
   \item the product $X/{\sim_X}\,\times\, Y/{\sim_Y}=(X\times Y)/{({{\sim_X}\times {\sim_Y}})}$, where $\sim_X \times \sim_Y$ is induced by $E_X(\bar x,\bar x^\prime)\wedge E_Y(\bar y, \bar y^\prime)$, is an $(\alpha,\beta)$\=/projective logical set over $\MA$;  
  \item the intersection $X/{\sim_X}\cap Y/{\sim_Y}=(X\cap Y)/{(\sim_X\cap \sim_Y)}$, where $\sim_X\cap \sim_Y$ is induced by $E_X(\bar x,\bar x^\prime)\wedge E_Y(\bar x, \bar x^\prime)$, is an $(\alpha,\beta)$\=/projective logical set over $\MA$, provided $X,Y\subseteq A^m$ and $X\cap Y\ne \emptyset$.
\end{enumerate}
In particular, if $X/{\sim_X}$ and $Y/{\sim_Y}$ are projective definable sets over $\MA$, then so are $X/{\sim_X}\times Y/{\sim_Y}$ and $X/{\sim_X}\cap Y/{\sim_Y}$.
\end{fact}

\begin{corollary}
The product operation ``$\times$'' on the set of all non-empty projective logical sets over $\MA$ is associative.    
\end{corollary}

\begin{notation}
If $\bar x\in X, \bar y\in Y$, then by $\bar x/{\sim_X}\times \bar y/{\sim_Y}$ we will denote the element $(\bar x \times \bar y)/{(\sim_X\times \sim_Y)}$ from $X/{\sim_X}\times Y/{\sim_Y}$.     For an integer $d\in \N$ the product of $d$ copies of $X/{\sim_X}$ we will denote in the standard way by $(X/{\sim_X})^d$ and its elements by $(\bar x/{\sim_X})^d$.
\end{notation}

Meanwhile, we have introduced objects of the category $\PLS(\MA)$, i.\,e., projective logical sets over $\MA$, and intend to describe its morphisms. 

\subsection{Formula maps}\label{subsec:fm}

Let $X/{\sim_X}$ and $Y/{\sim_Y}$ be non-empty projective logical sets over $\MA$, where $X\subseteq A^m$ and $Y\subseteq A^s$. 

\begin{definition}\label{def:mor}
We refer to a map $F\colon X/{\sim_X} \,\to Y/{\sim_Y}$ as {\em a formula map} with parameters in $P\subseteq A$, if there exists a formula $\varphi$ in the language $L(\MA)\cup P$ with free variables $\{x_1,\ldots,x_m\}\sqcup\{y_1,\ldots,y_s\}$, such that 
\begin{enumerate}[label=(\Roman*)]
    \item\label{eq:FXY1} for any $ \bar a\in X$, $\bar b\in Y$ one has $F(\bar a/{\sim_X})=\bar b/{\sim_Y}$ if and only if $\MA_P\models \varphi(\bar a,\bar b)$;
    \item\label{eq:FXY2} if $\bar a\in X$ and $\bar b\in A^s\setminus Y$ then $\MA_P\not\models \varphi(\bar a,\bar b)$.
\end{enumerate}
We say in this case that the formula $\varphi$ {\em induces} the map $F$. If $P=\emptyset$ for some appropriate formula $\varphi$, then $F$ is called an {\em absolute} formula map. And we name $F$ {\em full} it can be induced by a formula $\varphi$, such that in addition to conditions~\ref{eq:FXY1}--\ref{eq:FXY2}, the following condition is satisfied 
\begin{enumerate}[label=(\Roman*)]\addtocounter{enumi}{2}
    \item\label{eq:FXY3} if $\bar a\in A^m\setminus X$ and $\bar b\in Y$ then $\MA_P\not\models \varphi(\bar a,\bar b)$.
\end{enumerate} 
If $\varphi$ satisfies~\ref{eq:FXY1}--\ref{eq:FXY3}, we will say that $\varphi$ {\em fully induces} formula map $F$. 
\end{definition}

\begin{notation}
Note again that there may exist different formula maps from $X/{\sim_X}$ to $Y/{\sim_Y}$, but it is convenient to use the notation $F_{XY}$ for someone of them. And there may exist several formulas, which induce $F_{XY}$; however, we use the notation $\varphi_{XY}$ for someone of them.   
\end{notation}

Among all formula maps, we single out the so-called term maps between logical sets. 

\begin{definition}
Let $X\subseteq A^m$ and $Y\subseteq A^s$ be non-empty logical sets over $\MA$ and $F\colon X \to Y$ be a formula map. 
We name $F$ a {\em term map}, if it can be induced by a formula of the type
$$
\varphi(\bar x, \bar y)=(y_1=t_1(\bar x)\wedge\ldots\wedge y_s=t_s(\bar x)),
$$
where $t_1,\ldots,t_s$ are terms in the language $L(\MA)\cup A$ with free variables $x_1,\ldots,x_m$.
\end{definition}

Note that term maps, as usual, are not full.

\begin{example}\label{ex:non-empty1}
    Let $X/{\sim_X}$, $Y/{\sim_Y}$ be non-empty projective logical sets over $\MA$, a formula $E_Y$ induces $\sim_Y$, and $\bar b\in Y$ a point. Then the formula $\varphi_{XY}(\bar x, \bar y)=E_Y(\bar y,\bar b)$ induces a formula map $F_{XY}\colon X/{\sim_X}\,\to Y/{\sim_Y}$, such that $F_{XY}(\bar x/{\sim_X})=\bar b/{\sim_Y}$ for all $\bar x/{\sim_X}\in X/{\sim_X}$. Moreover, if $X/{\sim_X}$ and $Y/{\sim_Y}$ are logical sets, then $F_{XY}$ is a term map. In particular, if a logical set $Y$ is non-empty, then there always exists a term map $F\colon A^m\to Y$.
\end{example}

Let $\sim_1$ and $\sim_2$ be formula equivalencies on a non-empty logical set $X$, induced by formulas $E_1$ and $E_2$ in $L(\MA)\cup A$. We refer to a map $\pi\colon X/{\sim_1}\to X/{\sim_2}$ as a {\em natural surjection}, if $\sim_1$ is finer then $\sim_2$ and $\pi(\bar a/{\sim_1})=\bar a/{\sim_2}$ for all $\bar a\in X$.

\begin{fact}\label{fact:surjection}
A natural surjection $\pi\colon X/{\sim_1}\to X/{\sim_2}$ is a formula map, fully induced by the formula $E_2$. And inversely, if a formula map $\pi\colon X/{\sim_1}\to X/{\sim_2}$ is induced by $E_2$, then it is a natural surjection.
\end{fact}

\begin{proof}
If $\pi\colon X/{\sim_1}\to X/{\sim_2}$ is a natural surjection, then for any $\bar a,\bar b\in X$ one has $\pi(\bar a/{\sim_1})=\bar b/{\sim_2}$ if and only if $\bar a\sim_2 \bar b$, i.\,e., $\MA_A\models E_2(\bar a,\bar b)$. Since $X$ is closed under $E_2$, the formula $E_2$ fully induces the map $\pi$. Inversely, if $E_2$ induces $\pi$, then $\pi(\bar a/{\sim_1})=\bar a/{\sim_2}$ for any $\bar a\in X$. And if $\bar a\sim_1\bar a^\prime$ then $\bar a/{\sim_2}=\pi(\bar a/{\sim_1})=\pi(\bar a^\prime/{\sim_1})=\bar a^\prime/{\sim_2}$, i.\,e., $\bar a\sim_2\bar a^\prime$, for any $\bar a,\bar a^\prime\in X$. Therefore, $\sim_1$ is finer then $\sim_2$. Thus, $\pi$ is a natural surjection.
\end{proof}

\begin{example}\label{cor:id}
In particular, for any non-empty projective logical set $X/{\sim_X}$ over $\MA$ the identical map $\id_{X/{\sim_X}}\colon X/{\sim_X}\to X/{\sim_X}$ and the natural surjection $\pi_X\colon X\to X/{\sim_X}$ are full formula maps. If $X/{\sim_X}$ is absolute then $\id_{X/{\sim_X}}$ and $\pi_X$ are absolute as well.
\end{example}

Let $Z/{\sim_Z}$ and $X/{\sim_X}$ be non-empty projective logical sets over $\MA$, $Z=\V_\MA(S_Z)$, $X=\V_\MA(S_X)$, and $E_Z$, $E_X$ are formulas in $L(\MA)\cup A$, which induce $\sim_Z$ and $\sim_X$. We refer to a map $\varepsilon\colon Z/{\sim_Z}\to X/{\sim_X}$ as an {\em identical embedding}, if $Z/{\sim_Z}\subseteq X/{\sim_X}$ and $\varepsilon(\bar a/{\sim_Z})=\bar a/{\sim_X}$ for all $\bar a\in Z$. For further reference in constructing categorical functors, it is convenient for us to give the following descriptions of identical embeddings.

\begin{fact}\label{fact:subset}
An identical embedding $\varepsilon\colon Z/{\sim_Z}\to X/{\sim_X}$ is a formula map,  induced by the formula $E_Z$. And inversely, if $Z\subseteq X$ (i.\,e., $\V_\MA(S_Z)= \V_\MA(S_Z\cup S_X)$) and a formula map $\varepsilon\colon Z/{\sim_Z}\to X/{\sim_X}$ is induced by $E_Z$, then it is an identical embedding.
\end{fact}

\begin{proof}
If $\varepsilon$ is an identical embedding, then for any $\bar a\in Z$, $\bar b\in X$ one has $\varepsilon(\bar a/{\sim_Z})=\bar b/{\sim_X}$ if and only if $\bar a\sim_X\bar b$. Since $Z/{\sim_Z}\subseteq X/{\sim_X}$, conditions $\bar a\sim_X\bar b$ ans $\bar a\in Z$ imply that $\bar b\in Z$ and $\bar a\sim_Z\bar b$. Thus, $\varepsilon(\bar a/{\sim_Z})=\bar b/{\sim_X}$ if and only if $\MA_A\models E_Z(\bar a,\bar b)$. As $Z$ is closed under $E_Z$, the formula $E_Z$ induces the map $\varepsilon$. Inversely, if $Z\subseteq X$ and a formula map $\varepsilon\colon Z/{\sim_Z}\to X/{\sim_X}$ is induced by the formula $E_Z$, then $\varepsilon (\bar a/{\sim_Z})=\bar a/{\sim_X}$ for any $\bar a\in Z$. And if $\bar a\sim_X\bar b$, $\bar a\in Z$, $\bar b\in X$, then $\varepsilon (\bar a/{\sim_Z})=\bar b/{\sim_X}$, i.\,e., $\MA_A\models E_Z(\bar a,\bar b)$, therefore, $\bar b \in Z$. Hence, $Z$ is closed under $\sim_X$. Further, for any $\bar a,\bar b\in Z$ one has $\bar a\sim_Z\bar b$ if and only if $\MA_A\models E_Z(\bar a,\bar b)$, if and only if $\varepsilon (\bar a/{\sim_Z})=\bar b/{\sim_X}$, if and only if $\bar a\sim_X\bar b$. Thus, $\sim_Z$ is the restriction of $\sim_X$ on $Z$, so one has $Z/{\sim_Z}\subseteq X/{\sim_X}$. Hence, $\varepsilon$ is an identical embedding.
\end{proof}

\begin{example}\label{ex:non-empty2}
Let $X\subseteq A^m$ be a non-empty logical set over $\MA$. Then the identical embedding $\varepsilon\colon X\to A^m$ is a term map, which is induced by the formula $E_m$ from Example~\ref{ex:Em}. 
\end{example}

We need to formulate the conditions under which a given formula $\varphi_{XY}(\bar x,\bar y)$ induces some formula map.

\begin{lemma}\label{fm}
Let $X/{\sim_X}$ and $Y/{\sim_Y}$, $X\subseteq A^m$, $Y\subseteq A^s$, be non-empty projective logical sets over $\MA$. A formula $\varphi_{XY}(\bar x,\bar y)$, $|\bar x|=m$, $|\bar y|=s$, in the language $L(\MA)\cup A$ induces some formula map $F_{XY}\colon X/{\sim_X} \,\to Y/{\sim_Y}$ if and only if it satisfies the following conditions:
\begin{enumerate}[label=(\roman*)]
\item\label{fm1} for any $\bar a, \bar a^\prime\in X$ and $\bar b, \bar b^\prime\in Y$ if $\bar a\sim_X\bar a^\prime$ and $\bar b\sim_Y \bar b^\prime$, then
$$
\MA_A\models \varphi_{XY}(\bar a,\bar b) \iff \MA_A\models \varphi_{XY}(\bar a^\prime,\bar b^\prime);
$$
\item\label{fm2} for any $\bar a\in X$ and $\bar b,\bar b^\prime\in A^s$ if
$$
\MA_A\models\varphi_{XY}(\bar a, \bar b)\wedge \varphi_{XY}(\bar a, \bar b^\prime),
$$
then $\bar b,\bar b^\prime\in Y$ and $\bar b\sim_Y \bar b^\prime$;
\item\label{fm3} for any $\bar a\in X$ there exists $\bar b\in Y$, such that $\MA_A\models\varphi_{XY}(\bar a, \bar b)$.
\end{enumerate}
\end{lemma}

\begin{proof}
Straightforward.
\end{proof}

{\bf Definability of formula maps.} 
Let's get some simple facts straight.

\begin{fact}\label{fact10}
Let a formula $\varphi_{XY}$ in the language $L(\MA)\cup A$,  induces a formula map $F_{XY}\colon X/{\sim_X} \,\to Y/{\sim_Y}$ between non-empty projective logical sets over $\MA$. Then the following conditions are equivalent:
\begin{enumerate}[label=\arabic*)]
\item the formula $\varphi_{XY}$ defines the map $F_{XY}$;
\item for any $\bar a, \bar b\in \MA$ condition 
$\MA_A\models\varphi_{XY}(\bar a,\bar b)$ implies that $\bar a\in X$.
\end{enumerate}
\end{fact}

\begin{fact}\label{rem5}
If $X/{\sim_X}$ and $Y/{\sim_Y}$ are non-empty projective logical sets over $\MA$ and $F_{XY}\colon X/{\sim_X} \,\to Y/{\sim_Y}$ is a map, and $P\subseteq A$ is a subset, then the following conditions are equivalent:
\begin{enumerate}[label=\arabic*)]
\item $F_{XY}$ is $P$\=/definable, and, in particular, $F_{XY}$ is full;
\item $F_{XY}$ is a formula map with parameters in $P$ and $X$ is $P$\=/definable.
\end{enumerate}
\end{fact}

\begin{proof}
If $F_{XY}$ is $P$\=/definable, then it is a formula map and $X$ is $P$\=/definable due to Fact~\ref{rem0}. And, conversely, if a formula $\varphi_{XY}$ in $L(\MA)\cup P$ induces the map $F_{XY}$ and a formula $S_X$ in $L(\MA)\cup P$ defines the set $X$ then the formula $\varphi(\bar x,\bar y)=S_X(\bar x)\wedge \varphi_{XY}(\bar x,\bar y)$ defines the map $F_{XY}$.
\end{proof}

{\bf Compositions of formula maps.} 
It is obvious that for arbitrary non-empty sets $X\subseteq A^m$, $Y\subseteq A^s$, $Z\subseteq A^d$ and equivalence relations $\sim_X,\sim_Y,\sim_Z$ on them, and $P$\=/definable maps $F\colon X/{\sim_X}\to Y/{\sim_Y}$ and $G\colon Y/{\sim_Y}\to Z/{\sim_Z}$ the composition $G\circ F\colon X/{\sim_X}\to Z/{\sim_Z}$ is $P$\=/definable too. Now we need to check a similar fact for projective logical sets and formula maps.

\begin{lemma}\label{lemma:composition}
Let $X/{\sim_X}$, $Y/{\sim_Y}$, $Z/{\sim_Z}$ be non-empty projective logical sets over $\MA$, and $F_{XY}\colon X/{\sim_X}\,\to Y/{\sim_Y}$, $F_{YZ}\colon Y/{\sim_Y}\,\to Z/{\sim_Z}$ be formula maps with parameters in $P$, $P\subseteq A$, induced by formulas $\varphi_{XY}$ and $\varphi_{YZ}$ in $L(\MA)\cup P$ correspondingly. Then the composition $F_{XZ}=F_{YZ}\circ F_{XY}\colon X/{\sim_X}\,\to Z/{\sim_Z}$ is a formula map with parameters in $P$ as well and it is induced by the formula $
\varphi_{XZ}(\bar x,\bar z)=\exists\,\bar y \:(\varphi_{XY}(\bar x,\bar y) \wedge \varphi_{YZ}(\bar y,\bar z))$. Moreover, if formula maps $F_{XY}$ and $F_{YZ}$ are full, then $F_{XZ}$ is full too, i.\,e., if formula $\varphi_{XY}$ fully induces $F_{XY}$ and $\varphi_{YZ}$ fully induces $F_{YZ}$, then $\varphi_{XZ}$ fully induces $F_{XZ}$.
\end{lemma}

\begin{proof}
Straightforward.
\end{proof}

It is easy to see that in a special case of composition, when one of the maps is an identical embedding, the composition is induced by the initial formula.

\begin{fact}\label{remark:rest}
Let $F_{XY}\colon X/{\sim_X} \,\to Y/{\sim_Y}$ be a formula map between non-empty projective logical sets over $\MA$, induced by a formula $\varphi_{XY}$ in $L(\MA)\cup A$. Then for any non-empty projective logical subset $Z/{\sim_Z}\subseteq X/{\sim_X}$ the restriction of $F_{XY}$ on $Z/{\sim_Z}$ is a formula map, induced by the formula $\varphi_{XY}$. And similarly, if $W/{\sim_W}$ is a projective logical set over $\MA$, such that $Y/{\sim_Y}\subseteq W/{\sim_W}$, then $\varphi_{XY}$ induces the formula map $F_{XW}\colon X/{\sim_X}\to W/{\sim_W}$, which action coincides with action of $F_{XY}$.
\end{fact}

\begin{lemma}\label{cat_term}
Let $X, Y, Z$ be non-empty logical sets over $\MA$ and $F_{XY}\colon X\to Y$, $F_{YZ}\colon Y\to Z$ be term maps. Then $\id_X\colon X \to X$ and $F_{XZ}=F_{YZ}\circ F_{XY}\colon X\to Z$ are term maps.
\end{lemma}

\begin{proof}
The map $\id_X$ is induced by the formula $E_m$ from Example~\ref{ex:Em}, so it is a term map. If $F_{XY}$ induces by a formula $\varphi_{XY}(\bar x,\bar y)=\bigwedge_{j=1}^s (y_j=t_j(\bar x))$ and $F_{YZ}$ induces by a formula $\varphi_{YZ}(\bar y,\bar z)=\bigwedge_{k=1}^d (z_k=r_k(\bar y))$, where $t_j,r_k$ are terms, then $F_{XZ}$ induces by the formula $\varphi_{XZ}(\bar x,\bar z)=\bigwedge_{k=1}^d (z_k=r_k(t_1(\bar x),\ldots,t_s(\bar x)))$.
\end{proof}

{\bf The kernel of a formula map.} 
The following results reveal the idea of the kernel for formula maps. Note that just for full formula maps.

\begin{lemma}\label{lemma:F}
Let $X/{\sim_X}$, $Y/{\sim_Y}$, $W/{\sim_W}$ be non-empty projective logical sets over $\MA$, $F\colon X/{\sim_X}\,\to Y/{\sim_Y}$, $G\colon X/{\sim_X}\,\to W/{\sim_W}$ be formula maps, induced by formulas $\varphi(\bar x,\bar y)$, $\psi(\bar x,\bar w)$ in $L(\MA)\cup A$, and $F^\prime\colon W/{\sim_W}\,\to Y/{\sim_Y}$ a map, such that $F=F^\prime\circ G$, i.\,e., the diagram
\begin{equation}\label{eq:varphi}
  \begin{tikzcd}[column sep=huge, row sep=large]
X/{\sim_X} \arrow[d, "G"'] \arrow[r, "F"] & Y/{\sim_Y} \\
W/{\sim_W} \arrow[ru, "{F^\prime}"'] & 
\end{tikzcd}  
\end{equation}
is commutative. Suppose that $G$ is surjective and fully induced by $\psi$. Then the following hold:
\begin{enumerate}
    \item $F^\prime$ is a formula map, induced by the formula $\varphi_{WY}(\bar w,\bar y)=\exists\,\bar x\:(\varphi(\bar x,\bar y)\wedge \psi(\bar x,\bar w))$.
    \item If $\varphi$ fully induces $F$, then $\varphi_{WY}$ fully induces $F^\prime$.
    \item In particular, if $G\colon X/{\sim_1}\to X/{\sim_2}$ is a natural surjection and $\varphi$ fully induces $F$, then $\varphi$  fully induces $F^\prime$ as well.
\end{enumerate} 
\end{lemma}

\begin{proof}
Note that $F^\prime$ is a unique map, such that the diagram~\eqref{eq:varphi} is commutative, since $G$ is surjective. It is easy to verify by standard reasoning that the formula $\varphi_{WY}$ satisfies the conditions in Definition~\ref{def:mor} with respect to the map $F^\prime$, so the formula $\varphi_{WY}$ induces the map $F^\prime$. The second item is trivial. Suppose that $G$ is a natural surjection, then, by Fact~\ref{fact:surjection}, it is fully induced by the formula $E_2$, which induced the formula equivalence $\sim_2$. Let us check that formula $\varphi(\bar x,\bar y)$ fully induces $F^\prime$ as well as $\varphi_{XY}(\bar x,\bar y)=\exists\,\bar x^\prime\:(\varphi(\bar x^\prime,\bar y)\wedge E_2(\bar x^\prime,\bar x))$ does. Indeed, formula $\varphi(\bar x,\bar y)$ satisfies to the conditions~\ref{eq:FXY2} and~\ref{eq:FXY3} in Definition~\ref{def:mor}. Further, for any $\bar a\in X$, $\bar b\in Y$, if $\MA_A\models \varphi(\bar a,\bar b)$, then trivially $\MA_A\models \varphi_{XY}(\bar a,\bar b)$. And if $\MA_A\models \varphi_{XY}(\bar a,\bar b)$, then there exists $\bar a^\prime\in \MA$, such that $\MA_A\models \varphi(\bar a^\prime,\bar b)$ and $\MA_A\models E_2(\bar a^\prime,\bar a)$. Therefore, $\bar a^\prime\in X$,  $\bar a^\prime \sim_2 \bar a$ and $F(\bar a^\prime/{\sim_1})=\bar b/{\sim_Y}$. Since $F=F^\prime\circ G$, one has $F(\bar a/{\sim_1})=F^\prime(G(\bar a/{\sim_1}))=F^\prime(\bar a/{\sim_2})=F^\prime(\bar a^\prime/{\sim_2})=F^\prime(G(\bar a^\prime/{\sim_1}))=F(\bar a^\prime/{\sim_1})=\bar b/{\sim_Y}$. Thus, one has $\MA_A\models \varphi(\bar a,\bar b)$. We see that formulas $\varphi$ and $\varphi_{XY}$ are equivalent on the set $X\times Y$. Since $\varphi_{XY}$ satisfies to the condition~\ref{eq:FXY1} in Definition~\ref{def:mor} with respect to the map $F^\prime$, therefore, $\varphi$ does the same. 
\end{proof}

\begin{lemma}[on kernel]\label{kernel}
   Suppose that $X/{\sim_X}$ and $Y/{\sim_Y}$ are non-empty projective logical sets over $\MA$ and $F\colon X/{\sim_X} \,\to Y/{\sim_Y}$ is a full formula map, fully induced by a formula $\varphi(\bar x,\bar y)$ in the language $L(\MA)\cup A$. Then formula $E(\bar x,\bar x^\prime)=\exists\,\bar y\:(\varphi(\bar x,\bar y)\wedge\varphi(\bar x^\prime,\bar y))$ induces a formula equivalence ${\ker F}$ on $X$, which is coaster then $\sim_X$. Furthermore, the formula $\varphi$ fully induces the injective full formula map $F^\prime\colon X/{\ker F}\to Y/{\sim_Y}$, such that $F=F^\prime\circ\pi$, i.\,e., the diagram 
\begin{equation}\label{eq:pi}
\begin{tikzcd}[column sep=huge, row sep=large]
X/{\sim_X} \arrow[d, "\pi"'] \arrow[r, "F"] & Y/{\sim_Y} \\
X/{\ker F} \arrow[ru, "F^\prime"'] & 
\end{tikzcd}
\end{equation}
is commutative, where $\pi\colon X/{\sim_X}\to X/{\ker F}$ is the natural surjection. Furthermore, if $F$ is surjective, then $F^\prime$ is surjective as well.
\end{lemma}

\begin{proof}
Since for every $\bar a\in X$ there exists $\bar b\in Y$, such that $\MA_A\models \varphi(\bar a,\bar b)$, the relation $\ker F$ is reflexive on $X$. It is trivially symmetric. Suppose that for $\bar a,\bar a^\prime, \bar a^{\prime\prime}\in X$ one has $(\bar a,\bar a^\prime), (\bar a^\prime,\bar a^{\prime\prime})\in \ker F$, therefore, exist $\bar b,\bar b^\prime\in \MA$, such that $\MA_A\models \varphi(\bar a,\bar b)\wedge \varphi(\bar a^\prime,\bar b)\wedge \varphi(\bar a^\prime,\bar b^\prime)\wedge \varphi(\bar a^{\prime\prime},\bar b^{\prime})$. Hence, $\bar b,\bar b^\prime\in Y$ and $\bar b\sim_Y \bar b^\prime$, in particular, one has $\MA_A\models \varphi(\bar a,\bar b)\wedge\varphi(\bar a^{\prime\prime},\bar b)$, i.\,e., $(\bar a,\bar a^{\prime\prime})\in\ker F$ and $\ker F$ is transitive, so it is an equivalence relation on $X$. If for $\bar a,\bar a^\prime\in \MA$ one has $\bar a\in X$ and $(\bar a,\bar a^\prime)\in\ker F$, then there exists $\bar b\in \MA$, such that $\MA_A\models \varphi(\bar a,\bar b)\wedge \varphi(\bar a^\prime,\bar b)$. Therefore, $\bar b\in Y$. And since $F$ is full, one has $\bar a^\prime\in X$, so $X$ is closed under $E$, i.\,e., $\ker F$ is indeed a formula equivalence on $X$. By the definition of a formula map, we obtain that $\sim_X$ is finer than $\ker F$. Let $F^\prime$ be a map which makes the diagram~\eqref{eq:pi} commutative. It is obvious that $F^\prime$ is injective, and it is surjective if $F$ is surjective. By Lemma~\ref{lemma:F}, $F^\prime$ is fully induced by the formula $\varphi$.
\end{proof}

\begin{definition}
We refer to formula equivalence $\ker F$ from Lemma~\ref{kernel} as the {\em kernel} of the full formula map $F$.
\end{definition}

{\bf The image of the formula map.} 
For any formula map $F\colon X/{\sim_X} \,\to Y/{\sim_Y}$ the image $F(X/{\sim_X})$ is the set of all elements $\bar b/{\sim_Y}\in Y/{\sim_Y}$, such that there exists $\bar a/{\sim_X}\in X/{\sim_X}$, for which $F(\bar a/{\sim_X})=\bar b/{\sim_Y}$. However, the set $F(X/{\sim_X})$ is not always projective logical over $\MA$. Let us show that for full formula maps, the images are projective logical subsets.

\begin{lemma}[on image]\label{image}
Suppose that $F\colon X/{\sim_X} \,\to Y/{\sim_Y}$ is a full formula map between non-empty projective logical sets over $\MA$, fully induced by a formula $\varphi(\bar x,\bar y)$ in $L(\MA)\cup A$. Then the image $F(X/{\sim_X})$ is projective logical set over $\MA$, such that $F(X/{\sim_X})\subseteq Y/{\sim_Y}$, and the surjection $F\colon X/{\sim_X} \,\to F(X/{\sim_X})$ is a formula map, fully induced by the formula $\varphi$. 
\end{lemma}

\begin{proof}
Indeed, if $Y=\V_\MA(S_Y)$ and $\psi(\bar y)=\exists\,\bar x\: \varphi(\bar x,\bar y)$, then 
$$
W=\{\bar b\in Y \mid \exists \, \bar a/{\sim_X}\in X/{\sim_X}\text{ such that }F(\bar a/{\sim_X})=\bar b/{\sim_Y}\}=\V_\MA(S_Y\cup\{\psi\}).
$$
It is clear, that $W$ is closed under $\sim_Y$, so that $F(X/{\sim_X})=W/{\sim_Y}\subseteq Y/{\sim_Y}$, and $F\colon X/{\sim_X}\to W/{\sim_Y}$ is a surjective full formula map, fully induced by $\varphi$.
\end{proof}

{\bf Products and natural projections.}
In Fact~\ref{fact:times}, we have discussed the direct products of projective logical sets. Now we continue this consideration involving formula maps.

\begin{example}
For any non-empty projective set $X/{\sim_X}$ over $\MA$ and an integer $d\in \N$ the natural map $F\colon X/{\sim_X}\to (X/{\sim_X})^d$, which sends $x/{\sim_X}$ to $(x/{\sim_X})^d$, is a full formula map.
\end{example}

\begin{fact}\label{fact_times}
Let $F_j\colon X_j/{\sim_{X_j}}\to Y_j/{\sim_{Y_j}}$ be formula map between non-empty projective logical sets over $\MA$, which is induced by an $L(\MA)\cup A$\=/formula $\varphi_j(\bar x_j,\bar y_j)$, $j=1,\ldots,d$. Then the product of maps 
$$
F=F_1\times\ldots\times F_d\colon \;X_1/{\sim_{X_1}}\times\ldots\times X_d/{\sim_{X_d}} \;\to \;Y_1/{\sim_{Y_1}}\times\ldots\times Y_d/{\sim_{Y_d}}
$$
is a formula map, which is induced by the formula 
$$
\varphi(\bar x_1,\ldots,\bar x_d,\bar y_1,\ldots,\bar y_d)=\bigwedge\limits_{j=1}^d \varphi_j(\bar x_j,\bar y_j).
$$
Moreover, if all $F_j$ are full, then $F$ is full as well; if all $F_j$ are term maps, then $F$ is a term map as well.
\end{fact}

\begin{fact}\label{fact:proj}
Let $X_j/{\sim_{X_j}}$, $j=1,\ldots,d$,  be non-empty projective logical sets over $\MA$. Then for every $r\in\{1,\ldots,d\}$ the map $\Pi_r\colon X_1/{\sim_{X_1}}\times\ldots\times X_d/{\sim_{X_d}}\,\to X_r/{\sim_{X_r}}$, such that  $\Pi_r(\bar x_1/{\sim_{X_1}}\times\ldots\times \bar x_d/{\sim_{X_d}})=\bar x_r/{\sim_{X_r}}$ for any $\bar x_1/{\sim_{X_1}}\times\ldots\times \bar x_d/{\sim_{X_d}}\in X_1/{\sim_{X_1}}\times\ldots\times X_d/{\sim_{X_d}}$, is a formula map. 
\end{fact}

We will refer to formula maps like $\Pi_r$ in Fact~\ref{fact:proj} as {\em natural projections}.

\subsection{The category of projective logical sets and its significant subcategories}\label{subsec:cat}

Thus, we are ready to complete the definition of the category $\PLS(\MA)$ and its subcategories. 

\begin{definition}
Objects in the category $\PLS(\MA)$ are all projective logical sets over $\MA$. Morphisms between non-empty objects are formula maps. The empty set $\emptyset$ enters the category $\PLS(\MA)$ as an initial object, i.\,e., there exists a unique arrow $\emptyset \to O$ from the empty set $\emptyset$ to any object $O$, and no arrows from a non-empty object to $\emptyset$.   
\end{definition}

\begin{proposition}
$\PLS(\MA)$ is indeed a category, and it is small.
\end{proposition}

\begin{proof}
Follows from  Lemma~\ref{lemma:composition} and Example~\ref{cor:id}.
\end{proof}

The class of all objects and morphisms in $\PLS(\MA)$ will also be named the {\em projective logical geometry over $\MA$} or just {\em logical geometry over $\MA$}.

\begin{definition}
The categories $\LS(\MA)$, $\PDS(\MA)$, and $\DS(\MA)$  consist of logical, projective definable, and definable sets over $\MA$ as objects and formula maps as morphisms.    
\end{definition}

\begin{proposition}
 The categories $\LS(\MA)$, $\PDS(\MA)$ and $\DS(\MA)$ are full subcategories of the category $\PLS(\MA)$.
\end{proposition}

\begin{proof}
    Straightforward.
\end{proof}

\begin{proposition}
  Morphisms in the categories $\DS(\MA)$ and $\PDS(\MA)$ are all definable maps between objects.  
\end{proposition}

\begin{proof}
    Follows from Fact~\ref{rem5}.
\end{proof}

Note that the category $\DS(\MA)$ is known in the literature for a long time~\cite{AhZ}. 

For a subset $P\subseteq A$, we denote by $\PLS_P(\MA)$ the subcategory of the category $\PLS(\MA)$ of projective logical sets and formula maps with parameters in $P$. So $\PLS_A(\MA)=\PLS(\MA)$. In particular, $\PLS_0(\MA)$ is the category of absolute projective logical sets and absolute formula maps over $\MA$. Note that $\PLS_0(\MA_A)=\PLS(\MA)$. 
The categories $\LS_P(\MA)$, $\PDS_P(\MA)$ and $\DS_P(\MA)$ are defined similarly.  

For an infinite cardinals $\alpha\leqslant\beta$ or symbol $\infty$ we denote by $\PLS_{\alpha,\beta}(\MA)$ (similarly, $\LS_{\alpha,\beta}(\MA)$) the full subcategory of $\PLS(\MA)$, which objects are $(\alpha,\beta)$\=/projective logical sets (similarly, $(\alpha,\beta)$\=/logical sets) over $\MA$. Further, $\PLS_{P,\alpha,\beta}(\MA)$ and $\LS_{P,\alpha,\beta}(\MA)$ are their subcategories of objects and morphisms with parameters in $P$. If $P=\emptyset$, we write $\PLS_{0,\beta}$ and $\LS_{0,\beta}$, omitting index $_\alpha$. So, $\PLS(\MA)=\PLS_{\infty,\infty}(\MA)$, $\LS(\MA)=\LS_{\infty,\infty}(\MA)$, $\PDS(\MA)=\PLS_{\omega,\omega}(\MA)$ and $\DS(\MA)=\LS_{\omega,\omega}(\MA)$.

By $\LS^\term(\MA)$ ($\DS^\term(\MA)$) we will denote the category of logical (definable) sets over $\MA$ with term maps as morphisms (they are indeed categories due to Lemma~\ref{cat_term}). Its subcategory $\LS_0^\term(\MA)$ of absolute logical sets and absolute term maps is the logical category from B.\,I.\,Plotkin's articles~\cite{Plotkin6, Plotkin10, Plotkin5}, which we have mentioned in Subsection~\ref{subsec:LG}.

\begin{remark}\label{remark_emptyset}
Let us remind that the empty set $\emptyset$ is sometimes algebraic over $\MA$, sometimes not. At the same time, $\emptyset$ is always definable over $\MA$ ($\V_\MA(x\ne x)=\emptyset$). It makes sense to define category  $\PLS(\MA)$ only from non-empty sets (for comparison, type-definable sets are non-empty). However we leave the empty set $\emptyset$ in the category $\PLS(\MA)$ to preserve the chain of categorical inclusions $\AS(\MA)\subset \LS^\term_0(\MA)\subset \LS(\MA)\subset \PLS(\MA)$.
\end{remark}

Let $\SC$ be a category, such that $\DS^\term(\MA) \subseteq \SC \subseteq \LS(\MA)$, and $\F\colon \SC\to \PLS(\MB)$ be a categorical functor from the category $\SC$ to the category of projective logical set over an algebraic structure $\MB=\langle B; L(\MB)\rangle$. 

\begin{definition}
 We will name a functor $\F$ {\em degenerate}, if for every object $O\in \SC$ one has $\F(O)=\emptyset$; and {\em non-degenerate} otherwise.   
\end{definition}

\begin{fact}\label{fact:empty}
In the notations and assumptions above, suppose that there exists a non-empty set $O\in \SC$ such that $\F(O)=\emptyset$. Then the functor $\F$ is degenerate.    
\end{fact}

\begin{proof}
Indeed, for every $m\in \N$ there exists a term map $F\colon A^m\to O$ (see Example~\ref{ex:non-empty1}). Since the empty set $\emptyset$ is an initial object in $\PLS(\MB)$, it gives that $\F(A^m)=\emptyset$. Further, for any non-empty set $O^\prime\in C$ there exist $m\in \N$ and a term map $F^\prime\colon O^\prime\to A^m$ (see Example~\ref{ex:non-empty2}), therefore, $\F(O^\prime)=\emptyset$, as required.
\end{proof}

\begin{remark}
Besides the category $\PLS(\MA)$ itself and its distinct subcategories $\LS(\MA)$, $\PDS(\MA)$, $\DS(\MA)$  and so on, their skeletons  $\sk(\PLS(\MA))$, $\sk(\LS(\MA))$, $\sk(\PDS(\MA))$, $\sk(\DS(\MA))$ are of particular interest, or quotient categories of isomorphism types~\cite{Fritsch}. They possess a unique expressive power. 
\end{remark}

So, we will devote the next two subsections to the study of isomorphisms.

\subsection{Isomorphisms between projective logical sets}\label{subsec:iso_PLS}

In this and the next subsections, we are going to describe isomorphisms in the categories $\PLS(\MA)$ and $\PDS(\MA)$. All results here are also true for categories $\PLS_P(\MA)$ and $\PDS_P(\MA)$ for any subset $P\subseteq A$. 

We will denote the presence of categorical isomorphism between projective logical sets $X/{\sim_X}$ and $Y/{\sim_Y}$ in the standard way by $X/{\sim_X}\cong Y/{\sim_Y}$.

\begin{lemma}\label{le:iso}
Let $X/{\sim_X}$ and $Y/{\sim_Y}$ be non-empty projective logical sets over $\MA$, $X\subseteq A^m$, $Y\subseteq A^s$, and $F_{XY}\colon X/{\sim_X}\to Y/{\sim_Y}$ a morphism in the category $\PLS(\MA)$. Then the following conditions are equivalent:
\begin{enumerate}[label=\arabic*)]
\item\label{le:iso:item1} $F_{XY}$ is an isomorphism in the category $\PLS(\MA)$;
\item\label{le:iso:item2} $F_{XY}$ is bijective and full.
\end{enumerate}
In this case, if $\varphi_{XY}$ is a formula in $L(\MA)\cup A$, which fully induces $F_{XY}$, then  the inverse isomorphism $F_{YX}\colon Y/{\sim_Y}\to X/{\sim_X}$ can be fully induced by the formula 
$$
\varphi_{YX}(\bar y,\bar x)=\varphi_{XY}(\bar x,\bar y).
$$
\end{lemma}

\begin{proof}
\ref{le:iso:item1}$\Longrightarrow$\ref{le:iso:item2}: Suppose first that $F_{XY}$ is an isomorphism in $\PLS(\MA)$, i.\,e., there exists a morphism $F_{YX}\colon Y/{\sim_Y}\to X/{\sim_X}$ in $\PLS(\MA)$, such that $F_{XY}\circ F_{YX}=\id_{Y/{\sim_Y}}$ and $F_{YX}\circ F_{XY}=\id_{X/{\sim_X}}$. Let $\varphi_{XY}$ and $\varphi_{YX}$ be some formulas in $L(\MA)\cup A$,  which induce the formula maps $F_{XY}$ and $F_{YX}$. Therefore, for any tuples $\bar a_1,\bar a_2 \in X$ and $\bar b_1,\bar b_2 \in Y$ one has 
\[
\left.
\begin{array}{rclcl} 
\MA_A &\models&\exists\,\bar y \:(\varphi_{XY}(\bar a_1,\bar y) \wedge \varphi_{YX}(\bar y,\bar a_2)) & \iff &\bar a_1\sim_X \bar a_2, \\
\MA_A &\models&\exists\,\bar y \:(\varphi_{XY}(\bar y,\bar b_1) \wedge \varphi_{YX}(\bar b_2,\bar y))  &\iff &\bar b_1\sim_Y \bar b_2.
\end{array}
\right.
\]
It gives that for any tuples $\bar a\in X$ and $\bar b\in Y$ one has $\MA_A\models \varphi_{XY}(\bar a,\bar b)$ if and only if $\MA_A\models \varphi_{YX}(\bar b,\bar a)$. Indeed, suppose that $\MA_A\models \varphi_{XY}(\bar a,\bar b)$. Since $\bar a\sim_X \bar a$, there exists $\bar y$ with $\MA_A\models \varphi_{XY}(\bar a,\bar y) \wedge \varphi_{YX}(\bar y,\bar a)$, therefore, $\bar y\in Y$ and $F_{XY}(\bar a/{\sim_X})=\bar y/{\sim_Y}$, i.\,e., $\bar y\sim_Y \bar b$, and hence $\MA_A\models \varphi_{YX}(\bar b,\bar a)$. Similarly, one can establish that $\MA_A\models \varphi_{YX}(\bar b,\bar a)$ implies $\MA_A\models \varphi_{XY}(\bar a,\bar b)$. 

Since for every $\bar b\in Y$ there exists $\bar a \in X$ such that $F_{YX}(\bar b/{\sim_Y})=\bar a/{\sim_X}$, i.\,e., $\MA_A\models \varphi_{YX}(\bar b,\bar a)$, or $\MA_A\models \varphi_{XY}(\bar a,\bar b)$, or $F_{XY}(\bar a/{\sim_X})=\bar b/{\sim_Y}$, therefore $F_{XY}$ is surjective.  Further, if $\bar a_1,\bar a_2\in X$ and $F_{XY}(\bar a_1/{\sim_X})=F_{XY}(\bar a_2/{\sim_X})=\bar b/{\sim_Y}$, then $\MA_A\models\varphi_{XY}(\bar a_1,\bar b)\wedge\varphi_{XY}(\bar a_2,\bar b)$, therefore, $\MA_A\models\varphi_{YX}(\bar b, \bar a_1)\wedge\varphi_{YX}(\bar b, \bar a_2)$. And then $\bar a_1\sim_X\bar a_2$, so $F_{XY}$ is injective. Thus $F_{XY}$ is bijective as well.

Now we put $\varphi^\prime_{XY}(\bar x,\bar y)=\varphi_{XY}(\bar x,\bar y)\wedge \varphi_{YX}(\bar y,\bar x)$. The formula $\varphi^\prime_{XY}$ induces the formula map $F_{XY}$ as well as $\varphi_{XY}$, but additionally one has~\ref{eq:FXY3} from Definition~\ref{def:mor}. Correspondingly, the formula $\varphi^\prime_{YX}(\bar y,\bar x)=\varphi^\prime_{XY}(\bar x,\bar y)$ fully induces the inverse morphism $F_{YX}$.

\ref{le:iso:item2}$\Longrightarrow$\ref{le:iso:item1}: Now let's prove the assertion in the opposite direction. Suppose that the morphism $F_{XY}$ is a bijection map and it is fully induced by a formula $\varphi_{XY}$. Then we put $\varphi_{YX}(\bar y,\bar x)=\varphi_{XY}(\bar x,\bar y)$. Since $F_{XY}$ is a surjection, for any $\bar b\in Y$ there exits $\bar a\in X$ with $\MA_A\models\varphi_{YX}(\bar b, \bar a)$. If at the same time $\MA_A\models\varphi_{YX}(\bar b, \bar a^\prime)$ for some $\bar a^\prime\in A^m$, then $\MA_A\models\varphi_{XY}(\bar a^\prime, \bar b)$, so $\bar a^\prime \in X$ and $F_{XY}(\bar a^\prime/{\sim_X})=F_{XY}(\bar a/{\sim_X})=\bar b/{\sim_Y}$. Since $F_{XY}$ is a injection, we get $\bar a\sim_X \bar a^\prime$. Thus the formula $\varphi_{YX}$ satisfies the conditions in Lemma~\ref{fm}, therefore, it induces a morphism $F_{YX}\colon Y/{\sim_Y}\to X/{\sim_X}$. It is obvious that the morphisms $F_{XY}$ and $F_{YX}$ are inverse, so they are isomorphisms of the category $\PLS(\MA)$.
\end{proof}

\begin{corollary}
If $F_1,\ldots, F_d$ are isomorphisms between non-empty sets in $\PLS(\MA)$, then the product $F_1\times\ldots\times F_d$ is an isomorphism too.
\end{corollary}

\begin{proof}
Indeed, $F_1\times\ldots\times F_d$ is trivially bijective, and it is full by Fact~\ref{fact_times}.
\end{proof}

Let's formulate the expected result of the type ``the quotient by the kernel is isomorphic to the image''. Here we use Lemmas~\ref{kernel} and~\ref{image}.

\begin{theorem}[Isomorphism theorem]\label{th_hom}
Suppose that $X/{\sim_X}$ and $Y/{\sim_Y}$ are non-empty projective logical sets over $\MA$ and $F\colon X/{\sim_X} \,\to Y/{\sim_Y}$ is a full  formula map. Then the formula map $F^\prime\colon X/{\ker F}\to F(X/{\sim_X})$ is an isomorphism in $\PLS(\MA)$, which is fully induced by the same formula as $F$.
\end{theorem}

\begin{proof}
By Lemma~\ref{image}, $F\colon X/{\sim_X}\to F(X/{\sim_X})$ is a surjective full formula map. By Lemma~\ref{kernel}, $F^\prime\colon X/{\ker F}\to F(X/{\sim_X})$ is a bijective full formula, which is fully induced by the same formula as $F$. Therefore, by Lemma~\ref{le:iso}, $F^\prime$ is an isomorphism in $\PLS(\MA)$.
\end{proof}

Lemma~\ref{fm} specifies the conditions under which a formula $\varphi_{XY}(\bar x,\bar y)$ induces a categorical morphism. Now let's expand on these conditions to obtain a criterion for $\varphi_{XY}(\bar x,\bar y)$ to induce a categorical isomorphism.

\begin{lemma}\label{cor4}
Let $X/{\sim_X}$ and $Y/{\sim_Y}$, $X\subseteq A^m$, $Y\subseteq A^s$, be non-empty projective logical sets over $\MA$, and 
$\varphi_{XY}(\bar x,\bar y)$, $|\bar x|=m$, $|\bar y|=s$, a formula in the language $L(\MA)\cup A$. Then the following conditions are equivalent:
\begin{enumerate}[label=\arabic*)]
\item $\varphi_{XY}(\bar x,\bar y)$ induces an isomorphism $F_{XY}\colon X/{\sim_X} \,\to Y/{\sim_Y}$ in the category $\PLS(\MA)$ with inverse isomorphism $F_{YX}\colon Y/{\sim_Y}\to X/{\sim_X}$, induced by the formula $\varphi_{YX}(\bar y,\bar x)=\varphi_{XY}(\bar x,\bar y)$;
\item $\varphi_{XY}(\bar x,\bar y)$ satisfies to conditions~\ref{fm1}-- \ref{fm3} in Lemma~\ref{fm} and the following additional conditions:
\begin{enumerate}[label=(\roman*)]\addtocounter{enumii}{3}
\item for any $\bar a,\bar a^\prime\in A^m$ and $\bar b\in Y$ if
$$
\MA_A\models\varphi_{XY}(\bar a, \bar b)\wedge\varphi_{XY}(\bar a^\prime, \bar b),
$$
then $\bar a,\bar a^\prime\in X$ and $\bar a\sim_X \bar a^\prime$;
\item for any $\bar b\in Y$ there exists $\bar a\in X$ such that $\MA_A\models\varphi_{XY}(\bar a, \bar b)$.
\end{enumerate}
\end{enumerate}
\end{lemma}

\begin{proof}
Follows from Lemmas~\ref{fm} and~\ref{le:iso}.
\end{proof}

\begin{corollary}\label{cor:descent}
    Let $F_{XY}\colon X/{\sim_X}\to Y/{\sim_Y}$ be an isomorphism between non-empty projective logical sets over $\MA$, fully induced by a formula $\varphi_{XY}$ in $L(\MA)\cup A$. If $Z/{\sim_Z}\subseteq X/{\sim_X}$ and $W/{\sim_W}\subseteq Y/{\sim_Y}$ are projective logical subsets, such that $F(Z/{\sim_Z})=W/{\sim_W}$, then the restriction $F^\prime\colon Z/{\sim_Z}\to W/{\sim_W}$ of $F$ is an isomorphism, fully induced by the formula $\varphi_{XY}$.
\end{corollary}

\begin{proof}
    Indeed, formula $\varphi_{XY}$ satisfies to the condition~\ref{eq:FXY1} from Definition~\ref{def:mor} with respect to the map $F^\prime$. Since $Z$ is closed under $\sim_X$ and $W$ is closed under $\sim_Y$, and $F$ is bijection, then $\varphi_{XY}$ satisfies to the conditions~\ref{eq:FXY2} and~\ref{eq:FXY3} from Definition~\ref{def:mor} with respect to the map $F^\prime$. Therefore, $\varphi_{XY}$ fully induces $F^\prime$ and $F^\prime$ is an isomorphism.
\end{proof}

\subsection{Isomorphisms between projective definable sets}\label{subsec:iso_PDS}

Now let us note some simple consequences about isomorphisms in the category $\PDS(\MA)$ of projective definable sets over $\MA$.

\begin{lemma}\label{cor7}
Suppose that $X$ is a definable set and $F_{XY}\colon X/{\sim_X}\to Y/{\sim_Y}$ is a morphism in the category $\PLS(\MA)$ between non-empty objects. Then the following conditions are equivalent:
\begin{enumerate}[label=\arabic*)]
\item\label{cor7:item2} $F_{XY}$ is an isomorphism in the category $\PDS(\MA)$;
\item\label{cor7:item1} $F_{XY}$ is an isomorphism in the category $\PLS(\MA)$;
    \item\label{cor7:item3} $F_{XY}$ is bijective.   
\end{enumerate}
\end{lemma}

\begin{proof}
First of all, note that, by Fact~\ref{rem5}, $F_{XY}$ is definable and, in particular, full. Therefore, equivalence \ref{cor7:item1}$\Longleftrightarrow$\ref{cor7:item3} follows from Lemma~\ref{le:iso}. Implication~\ref{cor7:item2}$\Longrightarrow$\ref{cor7:item1} is trivial. 
\ref{cor7:item1},\ref{cor7:item3}$\Longrightarrow$\ref{cor7:item2}: If $F_{XY}$ is bijective, then by Fact~\ref{rem0}, $Y$ is definable, therefore, $F_{XY}$ is a morphism in $\PDS(\MA)$. Since $\PDS(\MA)$ is a full subcategory of $\PLS(\MA)$, then $F_{XY}$ is an isomorphism in $\PDS(\MA)$.
\end{proof}

\begin{corollary}\label{cor:iso_def}
A map $F_{XY}\colon X/{\sim_X}\to Y/{\sim_Y}$ between non-empty projective definable sets $X/{\sim_X}$ and $Y/{\sim_Y}$ over $\MA$ is an isomorphism in $\PLS(\MA)$ if and only if $F_{XY}$ is definable bijection.
\end{corollary}

\begin{proof}
It follows from Fact~\ref{rem5} and Lemma~\ref{cor7}.
\end{proof}

\begin{corollary}
Let $O_1, O_2$ be objects from the category $\PLS(\MA)$ and $O_1\cong O_2$. Then 
$$
O_1\in \PDS(\MA)\iff O_2\in \PDS(\MA).
$$
\end{corollary}

\begin{lemma}\label{cor8}
Suppose that $X/{\sim_X}$ and $Y/{\sim_Y}$ are non-empty projective logical sets over $\MA$ and $F_{XY}\colon X/{\sim_X}\to Y/{\sim_Y}$ is a map. Then the following conditions are equivalent:
\begin{enumerate}[label=\arabic*)]
    \item\label{cor8:item1} $X$ is definable and $F_{XY}$ is an isomorphism in $\PLS(\MA)$;
    \item\label{cor8:item2} $F_{XY}$ is bijective and definable.
\end{enumerate}
\end{lemma}

\begin{proof}
\ref{cor8:item1}$\Longrightarrow$\ref{cor8:item2}: If $X$ is definable and 
    $F_{XY}$ is a morphism in $\PLS(\MA)$, then $F_{XY}$ is definable due to Fact~\ref{rem5}. If $F_{XY}$ is an isomorphism in $\PLS(\MA)$, it is bijective due to Lemma~\ref{le:iso}. 
\ref{cor8:item2}$\Longrightarrow$\ref{cor8:item1}: If $F_{XY}$ is definable, then $X$ is definable due to Fact~\ref{rem5}. And if $F_{XY}$ is bijective, it is an isomorphism in $\PLS(\MA)$ due to Lemma~\ref{cor7}.
\end{proof}

Let us formulate and prove a detailed analog of Corollary~\ref{cor:descent} for the case when $Z$ is a definable set.
 
\begin{lemma}[on isomorphism descent]\label{lemma:descent}
Let $F_{XY}\colon X/{\sim_X}\to Y/{\sim_Y}$ be an isomorphism between non-empty projective logical sets over $\MA$ and $Z/{\sim_Z}\subseteq X/{\sim_X}$ is a non-empty projective definable subset, $Z=\V_\MA(S_Z)$. Then 
    \begin{enumerate}[label=\arabic*)]
        \item the restriction of $F_{XY}$ on $Z/{\sim_Z}$ is a full formula map $F_{ZY}\colon Z/{\sim_Z}\to Y/{\sim_Y}$,
        \item the image $F_{ZY}(Z/{\sim_Z})$ is a projective definable set over $\MA$,
        \item $F_{ZY}\colon Z/{\sim_Z}\to F_{ZY}(Z/{\sim_Z})$ is an isomorphism in $\PDS(\MA)$. 
    \end{enumerate}
Furthermore, if $X$ and $Y$ are definable sets and $\varphi_{XY}(\bar x,\bar y)$ is a formula in the language $L(\MA)\cup A$, which defines isomorphism $F_{XY}\colon X/{\sim_X}\to Y/{\sim_Y}$, then the formula $\varphi_{XY}(\bar z,\bar y)$ fully induces isomorphism $F_{ZY}\colon Z/{\sim_Z}\to F_{ZY}(Z/{\sim_Z})$, while the formula $\varphi_{XY}(\bar z,\bar y)\wedge S_Z(\bar z)$ defines it.
\end{lemma}

\begin{proof}
Since $Z$ is a definable set, then the restriction $F_{ZY}\colon Z/{\sim_Z}\to Y/{\sim_Y}$ of $F_{XY}$ is an injective, full, and definable formula map, by Fact~\ref{rem5}. Then by Lemma~\ref{image}, $F_{ZY}(Z/{\sim_Z})$ is a projective logical set and $F_{ZY}\colon Z/{\sim_Z}\to F_{ZY}(Z/{\sim_Z})$ is bijective formula map. Therefore, by Lemma~\ref{cor8}, $F_{ZY}\colon Z/{\sim_Z}\to F_{ZY}(Z/{\sim_Z})$ is an isomorphism in $\PLS(\MA)$; and by Lemma~\ref{cor7}, $F_{ZY}\colon Z/{\sim_Z}\to F_{ZY}(Z/{\sim_Z})$ is an isomorphisms in $\PDS(\MA)$, in particular, $F_{ZY}(Z/{\sim_Z})$ is a projective definable set over $\MA$.

Let $\varphi_{XY}$ be a formula in  $L(\MA)\cup A$, which defines isomorphism $F_{XY}\colon X/{\sim_X}\to Y/{\sim_Y}$. It is clear, that then $\varphi_{XY}(\bar z,\bar y)\wedge S_Z(\bar z)$ defines  isomorphism $F_{ZY}\colon Z/{\sim_Z}\to F_{ZY}(Z/{\sim_Z})$. Take any tuples $\bar a,\bar b\in \MA$, $|\bar a|=|\bar x|$, $|\bar b|=|\bar y|$. One has $\MA_A\models\varphi_{XY}(\bar a,\bar b)$ if and only if $\bar a\in X$, $\bar b\in Y$ and $F_{XY}(\bar a/{\sim_X})=\bar b/{\sim_Y}$. Suppose that $\MA_A\models\varphi_{XY}(\bar a,\bar b)$. If $\bar a\in Z$, then $\bar b/{\sim_Y}\in F_{ZY}(Z/{\sim_Z})$. And conversely, if $\bar b/{\sim_Y}\in F_{ZY}(Z/{\sim_Z})$, then $\bar a\in Z$, since $F_{XY}$ is injective. Therefore, $\varphi_{XY}$ fully induces isomorphism $F_{ZY}\colon Z/{\sim_Z}\to F_{ZY}(Z/{\sim_Z})$. 
\end{proof}

\begin{remark}\label{remark:P}
  All results about categorical isomorphisms from this and previous sections are true for categories $\PLS_P(\MA)$ and $\PDS_P(\MA)$ for any subset $P\subseteq A$, in particular, for $\PLS_0(\MA)$ and $\PDS_0(\MA)$. For example, if $X/{\sim_X}$ and $Y/{\sim_Y}$ are non-empty absolute projective logical sets and $F_{XY}\colon X/{\sim_X}\to Y/{\sim_Y}$ is a formula map, then it is an isomorphism in $\PLS_0(\MA)$ if and only if it is an isomorphism in $\PLS(\MA)$, which can be fully induced by a formula $\varphi_{XY}$ in the language $L(\MA)$, i.\,e., without parameters from $A$.
\end{remark}

\section{Interpretations and projective logical geometries}\label{sec:int}

In this section, we reveal the diverse and deep connections between interpretations and categories of projective logical sets. And we start by constructing a categorical functor that invariably accompanies each interpretation  $\MA\rightsquigarrow\MB$.

\subsection{Interpretations and interpretation functors}\label{subsec:int_func}

Here we show that if an algebraic structure $\MA$ is interpretable in an algebraic structure $\MB$, then the projective logical geometry over $\MB$ contains the projective logical geometry over $\MA$. The main goal here is to prove the following result.

\begin{theorem}\label{th:inter}
Let an algebraic structure $\MA=\langle A; L(\MA)\rangle$ be interpretable in an algebraic structure $\MB=\langle B; L(\MB)\rangle$. Then for any interpretation $(\Gamma,\bar p,\mu_\Gamma)$ of $\MA$ in $\MB$ there exists an embedding  $\F=\F_{\Gamma,\bar p,\mu_\Gamma}$ from the category $\PLS(\MA)$ of projective logical sets over $\MA$ to the category $\PLS(\MB)$ of projective logical sets over $\MB$. Moreover, for any infinite cardinals $\alpha\leqslant\beta$ the restriction of $\F$ to the full subcategory $\PLS_{\alpha,\beta}(\MA)$ gives rise to an embedding $\F\colon\PLS_{\alpha,\beta}(\MA)\to\PLS_{\alpha,\beta}(\MB),$ from the category $\PLS_{\alpha,\beta}(\MA)$ to the category $\PLS_{\alpha,\beta}(\MB)$. In particular, one gets the following embeddings:
\begin{gather*}
\F\colon\PLS(\MA)\to\PLS(\MB),\\
\F\colon\PDS(\MA)\to\PDS(\MB).
\end{gather*}
\end{theorem}

It should be noted that embedding in category theory is a faithful (injective on hom-sets) functor, which is also injective on objects.
Here and below, in similar circumstances, we denote the restrictions of a functor $\F$ by the same letter $\F$.

We start with some preliminary facts that are consequences of the existence of an interpretation $(\Gamma,\bar p,\mu_\Gamma)\colon\MA\rightsquigarrow\MB$. Denote by $n$ the dimension $\dim \Gamma$. Every time together with a coordinate map $\mu_\Gamma$ we will consider a map $\gamma\colon A\to U_\Gamma(\MB,\bar p)$, such that $\mu_\Gamma\circ\gamma=\id_A$ (see Subsection~\ref{subsec:ext_code}). Although there are several such functions $\gamma$, we will use this notation, since in the presence of an interpretation $(\Gamma,\bar p,\mu_\Gamma)\colon\MA\rightsquigarrow\MB$ the difference between them does not matter due to Corollary~\ref{cor:gamma}. 

We will construct a functor $\F=\F_{\Gamma,\bar p,\mu_\Gamma}=\F_{\Gamma,\bar p,\gamma}$ step-by-step and in parallel with this, note those of its properties (\ref{F1}--\ref{F7} below), by which it would be possible to construct it, even if no interpretation $\MA\rightsquigarrow\MB$ were given, but only an extended code $(\Gamma,\bar p,\gamma)\colon L(\MA)\cup A\to L(\MB)\cup B$, and additionally those properties (\ref{E1}--\ref{E3} below), which make $\F$ an embedding. All the time we will refer to the $(\Gamma,\bar p,\gamma)$\=/translation of formulas (see Subsection~\ref{subsec:ext_code}) and the Reduction Theorem~\ref{RT}.

{\bf Functor $\F$ on logical sets.}
Take any formula $\psi(x_1,\ldots,x_m)$ in the language $L(\MA)\cup A$. By $\psi_{\Gamma,\bar p,\gamma}(\bar x_1,\ldots,\bar x_m)$ we denote the $(\Gamma,\bar p,\gamma)$\=/translation of $\psi$. Here $|\bar x_i|=n$. For any set $S\subseteq {\bf F}_{L(\MA)\cup A}(x_1,\ldots,x_m)$ we put
$$
S_{\Gamma,\bar p,\gamma}=\{\psi_{\Gamma,\bar p,\gamma}\mid \psi\in S\}.
$$
So, $S_{\Gamma,\bar p,\gamma}\subseteq {\bf F}_{L(\MB)\cup B}(\bar x_1,\ldots,\bar x_m)$. We will use notations $X=\V_\MA(S)$ and $\bar X=\V_\MB(S_{\Gamma,\bar p,\gamma})$ for corresponding logical sets over $\MA$ and $\MB$. Note that $X\subseteq A^m$ and $\bar X\subseteq U_\Gamma^m(\MB,\bar p)$. 

\begin{remark}\label{remark:a,b}
It is clear, that if $X$ is an $(\alpha,\beta)$\=/logical set, then $\bar X$ is an $(\alpha,\beta)$\=/logical set too; in particular, if $X$ is a definable set, then $\bar X$ is a definable set too.   
\end{remark}

The following fact is a direct consequence of the Reduction Theorem~\ref{RT}.

\begin{fact}\label{fact1}
Let $(\Gamma,\bar p,\mu_\Gamma)\colon\MA\rightsquigarrow\MB$ be an interpretation. Then, in the notations and assumptions above, 
one has $\mu_\Gamma(\bar X)=X$ and $\mu_\Gamma^{-1}(X)=\bar X$.
In particular,
\begin{enumerate}[label=(F\arabic*)]
    \item\label{F1} for any system $S\subseteq {\bf F}_{L(\MA)\cup A}(x_1,\ldots,x_m)$ one has $\V_\MA(S)=\emptyset$ if and only if $\V_\MB(S_{\Gamma,\bar p,\gamma})=\emptyset$;
\item\label{F2} for any systems $S,S^\prime\subseteq {\bf F}_{L(\MA)\cup A}(x_1,\ldots,x_m)$ if $\V_\MA(S)=\V_\MA(S^\prime)$, then $\V_\MB(S_{\Gamma,\bar p,\gamma})=\V_\MB(S^\prime_{\Gamma,\bar p,\gamma})$.
\end{enumerate}
The inverse statement to~\ref{F2}  also holds, namely,
\begin{enumerate}[label=(E\arabic*)]
    \item\label{E1}   $\V_\MB(S_{\Gamma,\bar p,\gamma})=\V_\MB(S^\prime_{\Gamma,\bar p,\gamma})$ implies that $\V_\MA(S)=\V_\MA(S^\prime)$.
\end{enumerate}
\end{fact}

\begin{corollary}\label{cor:subset0}
In the notations and assumptions above, let  $S,S^\prime\subseteq {\bf F}_{L(\MA)\cup A}(x_1,\ldots,x_m)$. Then $\V_\MA(S)\subseteq \V_\MA(S^\prime)$ if and only if $\V_\MB(S_{\Gamma,\bar p,\gamma})\subseteq\V_\MB(S^\prime_{\Gamma,\bar p,\gamma})$.
\end{corollary}

\begin{remark}\label{cor:wedge0}
As a corollary of~\ref{F2}, for any formulas $\phi_1,\ldots,\phi_d$ in the language $L(\MA)\cup A$ we have 
\begin{equation*}
 \V_\MB((\bigwedge\limits_{j=1}^d \phi_j)_{\Gamma,\bar p,\gamma})=\V_\MB(\bigwedge\limits_{j=1}^d \phi_{j,\Gamma,\bar p,\gamma}).
\end{equation*}
\end{remark}

{\bf Functor $\F$ on projective logical sets.} 
Suppose now that $X$ is a non-empty set and $\sim_{\bar X}$ is a formula equivalence on $X$, induced by a formula $E_X\in {\bf F}_{L(\MA)\cup A}(\bar x,\bar x^\prime)$, $|\bar x|=|\bar x^\prime|=m$. Then we put
$$
E_{\bar X}=(E_X)_{\Gamma, \bar p, \gamma}.
$$

\begin{fact}\label{fact1.5}
Let $(\Gamma,\bar p,\mu_\Gamma)\colon\MA\rightsquigarrow\MB$ be an interpretation. Then, in the notations and assumptions above, 
for any $\bar {\bar b}, \bar {\bar b}^\prime\in U^m_\Gamma(\MB,\bar p)$ one has
\begin{equation}\label{eq:sim}
\MB_B\models E_{\bar X}(\bar {\bar b},\bar {\bar b}^\prime) \iff \MA_A\models E_X(\mu_\Gamma(\bar {\bar b }),\mu_\Gamma(\bar {\bar b}^\prime)).
\end{equation}
Thus, we obtain that
\begin{enumerate}[label=(F\arabic*)]\addtocounter{enumi}{2}
    \item\label{F3} if a formula $E_X\in {\bf F}_{L(\MA)\cup A}(\bar x,\bar x^\prime)$ induces a formula equivalence on $X=\V_A(S)$, then the formula $E_{\bar X}\in {\bf F}_{L(\MB)\cup B}(\bar {\bar x},\bar {\bar x}^\prime)$ induces a formula equivalence on $\bar X=\V_\MB(S_{\Gamma,\bar p,\gamma})$;
    \item\label{F4} if formulas $E_X,E^\prime_X\in {\bf F}_{L(\MA)\cup A}(\bar x,\bar x^\prime)$ induce one and the same formula equivalences  on $X=\V_A(S)$, then formulas $E_{\bar X}=(E_X)_{\Gamma, \bar p, \gamma}$ and $E^\prime_{\bar X}=(E^\prime_X)_{\Gamma, \bar p, \gamma}$ induce one and the same formula equivalences  on $\bar X=\V_\MB(S_{\Gamma,\bar p,\gamma})$.
\end{enumerate}
The inverse statements also hold, namely, if $E_{\bar X}$ induces a formula equivalence on $\bar X$ then $E_X$ induces a formula equivalence on $X$; and 
\begin{enumerate}[label=(E\arabic*)]\addtocounter{enumi}{1}
    \item\label{E2} if $E_{\bar X}$ and $E^\prime_{\bar X}$ induce one and the same formula equivalences on $\bar X$ then $E_X$ and $E_X^\prime$ induce one and the same formula equivalences on $X$.
\end{enumerate}
\end{fact}

\begin{proof}
The statement~\eqref{eq:sim} follows from the Reduction Theorem~\ref{RT}. To prove~\ref{F3}, suppose that $E_X$ induces a formula equivalence on $X$. Then the relation defined by the formula $E_{\bar X}$ on $\bar X$ is an equivalence relation on $\bar X$. By the definition of $\Gamma$\=/translation for any tuples $\bar {\bar b}, \bar {\bar b }^\prime\in B^{mn}$ the condition $\MB_B\models E_{\bar X}(\bar {\bar b},\bar {\bar b}^\prime)$ implies that $\bar {\bar b}, \bar {\bar b}^\prime\in U^m_\Gamma(\MB,\bar p)$. Therefore, if at least one of tuples $\bar {\bar b}$ and $\bar {\bar b}^\prime$ belongs to $\bar X$, say $\bar {\bar b}$, and $\MB_B\models E_{\bar X}(\bar {\bar b},\bar {\bar b}^\prime)$, then $\MA_A\models E_X(\mu_\Gamma(\bar {\bar b}),\mu_\Gamma(\bar {\bar b}^\prime))$ and $\mu_\Gamma(\bar {\bar b})\in X$. Since $X$ is closed under $E_X$, one has $\mu_\Gamma(\bar {\bar b}^\prime)\in X$, i.\,e., $\bar {\bar b}^\prime\in \bar X$, so $\bar X$ is closed under $E_{\bar X}$. Thus, $E_{\bar X}$ induces a formula equivalence on $\bar X$. To prove~\ref{F4}, suppose that formulas $E_X$ and $E^\prime_X$ induce the same formula equivalences on $X$. Then for all $\bar {\bar b},\bar {\bar b}^\prime\in \bar X$ one has $\MB_B\models E_{\bar X}(\bar {\bar b},\bar {\bar b}^\prime)$ if and only if $\MA_A\models E_X(\mu_\Gamma(\bar {\bar b}),\mu_\Gamma(\bar {\bar b}^\prime))$ if and only if $\MA_A\models E^\prime_X(\mu_\Gamma(\bar {\bar b}),\mu_\Gamma(\bar {\bar b}^\prime))$ if and only if $\MB_B\models E^\prime_{\bar X}(\bar {\bar b},\bar {\bar b}^\prime)$, so $E_{\bar X}$ and $E^\prime_{\bar X}$ induce one and the same formula equivalences on $\bar X$. The inverse statements can be substantiated using similar arguments. 
\end{proof}

\begin{corollary}\label{cor:nat_sur}
In the notations and assumptions above, let $E_X,E^\prime_X\in {\bf F}_{L(\MA)\cup A}(\bar x,\bar x^\prime)$ be formulas, which induce formula equivalences on $X$. Then $E_X\subseteq E^\prime_X$ if and only if $E_{\bar X}\subseteq E^\prime_{\bar X}$.
\end{corollary}

Let us denote by $\sim_{\bar X}$ the formula equivalence on $\bar X$ induced by $E_{\bar X}$. It is easy too see that if $\sim_X$ is homogeneous, then $\sim_{\bar X}$ is homogeneous too (if $E_X$ defines an equivalence relation on $A^m$, then $E_{\bar{X}}$ defines an equivalence relation on $U^m_\Gamma(\MB,\bar p)$, therefore, by Fact~\ref{rem2} the relation $\sim_{\bar X}$ is homogeneous). If the code $\Gamma$ and the formula equivalence $\sim_X$ are absolute, then the formula equivalence $\sim_{\bar X}$ is absolute too. 

\begin{corollary}\label{cor:subset}
    If $Z/{\sim_Z}$ is non-empty projective logical set over $\MA$, such that $Z/{\sim_Z}\subseteq X/{\sim_X}$, then $\bar Z/{\sim_{\bar Z}}\subseteq \bar X/{\sim_{\bar X}}$.
\end{corollary}

\begin{proof}
    It follows from Facts~\ref{fact1} and~\ref{fact1.5}.
\end{proof}

\begin{notation}
If a formula equivalence $\sim_X$ on $X$ is the identity relation $=$, then we will denote $\sim_{\bar X}$ by $\sim_\Gamma$, since due to procedure of $\Gamma$\=/translation in this case one has
\begin{equation}\label{eq:rel_Gamma}
(\bar b_1,\ldots,\bar b_m)\sim_\Gamma (\bar b_1^\prime,\ldots, \bar b^\prime_m) \iff \bar b_1\sim_\Gamma \bar b_1^\prime\:\wedge\:\ldots\:\wedge\: \bar b_m\sim_\Gamma \bar b_m^\prime
\end{equation}
for any tuples $\bar{\bar b}=(\bar b_1,\ldots,\bar b_m),\bar{\bar b}^\prime=(\bar b_1^\prime,\ldots,\bar b_m^\prime)\in \bar X$. 
\end{notation}

In particular, the relation $\sim_\Gamma$ on $\bar X$ is finer than any other formula equivalence $\sim_{\bar X}$, induced by a formula of the type $(E_X)_{\Gamma, \bar p, \gamma}$, where $E_X$ is a formula, which induces a formula equivalence on $X$.

\begin{remark}
If the code $\Gamma$ is injective and $\sim_X$ is the identity relation on $X$, then $\sim_{\bar X}$ is the identity relation on $\bar X$, i.\,e., if $X/{\sim_X}$ is a logical set, then $\bar X/{\sim_{\bar X}}$ is a logical set too.
\end{remark}

\begin{remark}
If the code $\Gamma$ and the projective logical set $X/{\sim_X}$ are absolute, then the projective logical set $\bar X/{\sim_{\bar X}}$ is absolute too.
\end{remark}

Note that the statement~\eqref{eq:sim} can be rewritten in the following form: for any tuples
$\bar b, \bar b^\prime\in U^m_\Gamma(\MB,\bar p)$ one has 
$$
\mu_\Gamma(\bar b)\sim_X \mu_\Gamma (\bar b^\prime)\iff \bar b \sim_{\bar X} \bar b^\prime,
$$
therefore, the map 
$$
\tilde\mu_\Gamma\colon \bar X/{\sim_{\bar X}} \to X/{\sim_X}, \quad \tilde\mu_\Gamma(\bar b/{\sim_{\bar X}})=\mu_\Gamma(\bar b)/{\sim_X},\quad \bar b\in \bar X,
$$
is well-defined. Actually the map $\tilde\mu_\Gamma$ depends on the set $X/{\sim_X}$, i.\,e., for each $X/{\sim_X}$ the coordinate map $\mu_\Gamma$ gives rise to the specific map $\tilde\mu_{\Gamma,X/{\sim_X}}$. But we will omit the lower index $_{X/\sim_X}$, unless it leads to a misunderstanding.

\begin{corollary}\label{cor3}
In the notation above, the map $\tilde\mu_{\Gamma,X/{\sim_X}}\colon \bar X/{\sim_{\bar X}} \to X/{\sim_X}$ is bijection.
\end{corollary}

\begin{example}\label{ex:Gamma}
If $X=A^m$, then $\bar X/{\sim_{\bar X}}=U_\Gamma^m(\MB,\bar p)/{\sim_\Gamma}$. In particular,  if $X=A$, i.\,e., $m=1$, then $\bar X/{\sim_{\bar X}}=U_\Gamma(\MB,\bar p)/{\sim_\Gamma}$. In the latter case we also obtain that $\tilde\mu_\Gamma=\bar\mu_\Gamma\colon U_\Gamma(\MB,\bar p)/{\sim_\Gamma}\to A$ is the isomorphism of the interpretation $(\Gamma,\bar p,\mu_\Gamma)\colon\MA\rightsquigarrow\MB$. 
\end{example}

{\bf Functor $\F$ on formula maps.} Now let us consider two non-empty projective logical sets $X/{\sim_X}$ and $Y/{\sim_Y}$ over $\MA$ and a formula map $F_{XY}\colon X/{\sim_X}\, \to Y/{\sim_Y}$, which is induced by a formula $\varphi_{XY}$ in the language $L(\MA)\cup A$. Then we put 
$$
\varphi_{\bar X \bar Y}=(\varphi_{XY})_{\Gamma, \bar p, \gamma}.
$$

\begin{fact}\label{fact14}
Let $(\Gamma,\bar p,\mu_\Gamma)\colon\MA\rightsquigarrow\MB$ be an interpretation. Then, in the notations and assumptions above, 
\begin{enumerate}[label=(F\arabic*)]\addtocounter{enumi}{4}
    \item\label{F5} if a formula $\varphi_{XY}\in {\bf F}_{L(\MA)\cup A}(\bar x,\bar y)$ induces some formula map $F_{XY}\colon X/{\sim_X}\, \to Y/{\sim_Y}$, then the formula $\varphi_{\bar X \bar Y}\in {\bf F}_{L(\MB)\cup B}(\bar {\bar x},\bar {\bar y})$ induces some formula map $F_{\bar X \bar Y}\colon \bar X/{\sim_{\bar X}} \,\to \bar Y/{\sim_{\bar Y}}$, for any non-empty projective logical sets $X/{\sim_X}=\V_\MA(S_X)/{\sim_X}$ and $Y/{\sim_Y}=\V_\MA(S_Y)/{\sim_Y}$.
\end{enumerate}
The inverse statement also holds, namely, if $\varphi_{\bar X \bar Y}$ induces some formula map $F_{\bar X \bar Y}\colon \bar X/{\sim_{\bar X}} \,\to \bar Y/{\sim_{\bar Y}}$ then $\varphi_{XY}$ induces some formula map $F_{XY}\colon X/{\sim_X}\, \to Y/{\sim_Y}$. Furthermore, $F_{XY}$ is full if and only if $F_{\bar X \bar Y}$ is full.
\end{fact}

\begin{proof}
Indeed, by the Reduction Theorem~\ref{RT}, the formula $\varphi_{\bar X \bar Y}$ satisfies to the conditions~\ref{fm1}--\ref{fm3} from Lemma~\ref{fm} if and only if the formula $\varphi_{X Y}$ does it.  
\end{proof} 

So, we denote by $F_{\bar X \bar Y}\colon \bar X/{\sim_{\bar X}} \,\to \bar Y/{\sim_{\bar Y}}$ the formula map, induced by the formula $\varphi_{\bar X \bar Y}$. By Corollary~\ref{cor3} the maps $\tilde\mu_{\Gamma,X/{\sim_X}}$, $\tilde\mu_{\Gamma,Y/{\sim_Y}}$ are bijections, so the maps $\tilde\mu^{-1}_{\Gamma,X/{\sim_X}}$, $\tilde\mu^{-1}_{\Gamma,Y/{\sim_Y}}$ are well-defined. The Reduction Theorem~\ref{RT} makes the following diagram commutative:
\begin{equation}\label{eq:F_mu}
 \begin{tikzcd}[column sep=huge, row sep=large]
X/{\sim_X} 
\arrow[r,leftarrow,shift left,"\tilde\mu_{\Gamma,X/{\sim_X}}"]
\arrow[r,rightarrow,shift right,swap,"\tilde\mu_{\Gamma,X/{\sim_X}}^{-1}"] 
\arrow{d}{F_{XY}} 
&\bar X/{\sim_{\bar X}} 
\arrow{d}{F_{\bar X \bar Y}}\\
Y/{\sim_Y} 
\arrow[r,leftarrow,shift left,"\tilde\mu_{\Gamma,Y/{\sim_Y}}"]
\arrow[r,rightarrow,shift right,swap,"\tilde\mu_{\Gamma,Y/{\sim_Y}}^{-1}"]  
&\bar Y/{\sim_{\bar Y}}.
\end{tikzcd}
\end{equation}
As a consequence, we obtain the following fact.

\begin{fact}\label{fact15}
Let $(\Gamma,\bar p,\mu_\Gamma)\colon\MA\rightsquigarrow\MB$ be an interpretation. Then, in the notations and assumptions above, 
\begin{enumerate}[label=(F\arabic*)]\addtocounter{enumi}{5}
    \item\label{F6} if formulas $\varphi_{XY}$ and $\varphi^\prime_{XY}$ from ${\bf F}_{L(\MA)\cup A}(\bar x,\bar y)$ induce one and the same formula map $F_{XY}\colon X/{\sim_X}\, \to Y/{\sim_Y}$, then the formulas $\varphi_{\bar X \bar Y}=(\varphi_{XY})_{\Gamma, \bar p, \gamma}$ and $\varphi^\prime_{\bar X \bar Y}=(\varphi^\prime_{XY})_{\Gamma, \bar p, \gamma}$ induce one and the same formula map $F_{\bar X \bar Y}\colon \bar X/{\sim_{\bar X}} \,\to \bar Y/{\sim_{\bar Y}}$, for any non-empty projective logical sets $X/{\sim_X}=\V_\MA(S_X)/{\sim_X}$ and $Y/{\sim_Y}=\V_\MA(S_Y)/{\sim_Y}$.
\end{enumerate}
The inverse statement also holds, namely, 
\begin{enumerate}[label=(E\arabic*)]\addtocounter{enumi}{2}
    \item\label{E3} if the formulas $\varphi_{\bar X \bar Y}$ and $\varphi^\prime_{\bar X \bar Y}$ induce one and the same formula map $F_{\bar X \bar Y}\colon \bar X/{\sim_{\bar X}} \,\to \bar Y/{\sim_{\bar Y}}$, then the formulas $\varphi_{XY}$ and $\varphi^\prime_{XY}$ induce one and the same formula map $F_{XY}\colon X/{\sim_X}\, \to Y/{\sim_Y}$.
\end{enumerate}
\end{fact}

Remind that if the set $X$ is definable, then the set $\bar X$ is definable too.
 
\begin{fact}\label{fact9}
In the notations and assumptions above, suppose that $X$ is definable. Then $\varphi_{XY}$ defines $F_{XY}$ if and only if $\varphi_{\bar X\bar Y}$ defines $F_{\bar X\bar Y}$.
\end{fact}

\begin{proof}
We use Fact~\ref {fact10} here. Suppose that $\varphi_{XY}$ defines $F_{XY}$. Then for any tuples $\bar a,\bar a^\prime\in A^m$ the condition $\MA_A\models\varphi_{XY}(\bar a,\bar a^\prime)$ implies that $\bar a\in X$. If for $\bar{\bar b},\bar{\bar b}^\prime\in B^{nm}$ one has $\MB_B\models\varphi_{\bar X\bar Y}(\bar{\bar b},\bar{\bar b}^\prime)$, then by the Reduction Theorem~\ref{RT}, it gives $\MA_A\models\varphi_{XY}(\mu_\Gamma(\bar{\bar b}),\mu_\Gamma(\bar{\bar b}^\prime))$, therefore, $\mu_\Gamma(\bar{\bar b})\in X$, then $\bar{\bar b}\in \bar X$. In particular, $\varphi_{\bar X\bar Y}$ defines $F_{\bar X\bar Y}$. The inverse statement can be proved similarly.
\end{proof}

\begin{remark}
If the code $\Gamma$ and the morphism $F_{XY}$ are absolute, then the morphism $F_{\bar X \bar Y}$ is absolute too.
\end{remark}

Suppose that $X/{\sim_X}, Y/{\sim_Z}, Z/{\sim_Z}$ are non-empty projective logical sets over $\MA$, $F_{XY}\colon X/{\sim_X}\, \to Y/{\sim_Y}$ and $F_{YZ}\colon Y/{\sim_Y}\, \to Z/{\sim_Z}$ are morphisms induced by formulas $\varphi_{XY}$ and $\varphi_{YZ}$ and $F_{XZ}\colon X/{\sim_X}\, \to Z/{\sim_Z}$ is the composition $F_{XZ}=F_{YZ}\circ F_{XY}$ induced by the formula $\varphi_{XZ}(\bar x,\bar z)=\exists\,\bar y \:(\varphi_{XY}(\bar x,\bar y) \wedge \varphi_{YZ}(\bar y,\bar z))$ (see Lemma~\ref{lemma:composition}).

\begin{fact}\label{fact17}
Let $(\Gamma,\bar p,\mu_\Gamma)\colon\MA\rightsquigarrow\MB$ be an interpretation. Then, in the notations and assumptions above, 
\begin{enumerate}[label=(F\arabic*)]\addtocounter{enumi}{6}
    \item\label{F7} if formulas $\varphi_{XY}\in {\bf F}_{L(\MA)\cup A}(\bar x,\bar y)$ and $\varphi_{YZ}\in {\bf F}_{L(\MA)\cup A}(\bar y,\bar z)$ induce some formula maps $F_{XY}\colon X/{\sim_X}\, \to Y/{\sim_Y}$ and $F_{YZ}\colon Y/{\sim_Y}\, \to Z/{\sim_Z}$; and if $F_{\bar X \bar Y}\colon \bar X/{\sim_{\bar X}} \,\to \bar Y/{\sim_{\bar Y}}$ and $F_{\bar Y \bar Z}\colon \bar Y/{\sim_{\bar Y}} \,\to \bar Z/{\sim_{\bar Z}}$ are formula maps, induced by the formulas $\varphi_{\bar X \bar Y}=(\varphi_{XY})_{\Gamma, \bar p, \gamma}$ and $\varphi_{\bar Y \bar Z}=(\varphi_{YZ})_{\Gamma, \bar p, \gamma}$, then the composition $F_{\bar Y \bar Z}\circ F_{\bar X \bar Y}\colon\bar X/{\sim_{\bar X}} \,\to \bar Z/{\sim_{\bar Z}}$ is induced by the formula $\varphi_{\bar X \bar Z}=(\varphi_{XZ})_{\Gamma, \bar p, \gamma}$, where $\varphi_{XZ}(\bar x,\bar z)=\exists\,\bar y \:(\varphi_{XY}(\bar x,\bar y) \wedge \varphi_{YZ}(\bar y,\bar z))$, for any non-empty projective logical sets $X/{\sim_X}=\V_\MA(S_X)/{\sim_X}$, $Y/{\sim_Y}=\V_\MA(S_Y)/{\sim_Y}$ and $Z/{\sim_Z}=\V_\MA(S_Z)/{\sim_Z}$.
\end{enumerate}
\end{fact}

\begin{proof}
Let $X\subseteq A^m$, $Y\subseteq A^s$, $Z\subseteq A^d$. Take any tuples $\bar{\bar b}\in \bar X$ and $\bar{\bar b}^{\prime\prime}\in \bar Z$. By the Reduction Theorem~\ref{RT} one has $\MB_B\models \varphi_{\bar X \bar Z}(\bar{\bar b}, \bar{\bar b}^{\prime\prime})$ if and only if $\MA_A\models \varphi_{ X  Z}(\mu_\Gamma(\bar{\bar b}), \mu_\Gamma(\bar{\bar b}^{\prime\prime}))$, if and only if there exists a tuple $\bar a^\prime\in A^s$ such that $\MA_A\models \varphi_{XY}(\mu_\Gamma(\bar{\bar b}),\bar a^\prime) \wedge \varphi_{YZ}(\bar a^\prime,\mu_\Gamma(\bar{\bar b}^{\prime\prime}))$ (and since $\mu_\Gamma(\bar{\bar b})\in X$, it gives that $a^\prime\in Y$, take any $\bar{\bar b}^\prime\in \mu_\Gamma^{-1}(\bar a^\prime)$, so $\bar{\bar b}^\prime\in \bar Y$), if and only if there exists a tuple $\bar{\bar b}^\prime\in \bar Y$ such that $\MA_A\models \varphi_{XY}(\mu_\Gamma(\bar{\bar b}),\mu_\Gamma(\bar{\bar b}^{\prime})) \wedge \varphi_{YZ}(\mu_\Gamma(\bar{\bar b}^{\prime}),\mu_\Gamma(\bar{\bar b}^{\prime\prime}))$, if and only if there exists a tuple $\bar{\bar b}^\prime\in \bar Y$ such that $\MB_B\models \varphi_{\bar X \bar Y}(\bar{\bar b},\bar{\bar b}^{\prime}) \wedge \varphi_{\bar Y \bar Z}(\bar{\bar b}^{\prime},\bar{\bar b}^{\prime\prime})$, if and only if $F_{\bar Y \bar Z}\circ F_{\bar X \bar Y}(\bar{\bar b}/{\sim_X})=\bar{\bar b}^{\prime\prime}/{\sim_Z}$. Further, if for some tuples $\bar{\bar b}\in \bar X$ and $\bar{\bar b}^{\prime\prime}\in B^{nd}$ one has $\MB_B\models \varphi_{\bar X \bar Z}(\bar{\bar b}, \bar{\bar b}^{\prime\prime})$, then  $\bar{\bar b}^{\prime\prime}\in U_\Gamma^{d}(\MB,\bar p)$. After repeating arguments above we obtain that there exists a tuple $\bar{\bar b}^\prime\in \bar Y$ such that $\MB_B\models  \varphi_{\bar Y \bar Z}(\bar{\bar b}^{\prime},\bar{\bar b}^{\prime\prime})$, and it implies that $\bar{\bar b}^{\prime\prime}\in \bar Z$, as required.  
\end{proof}

\begin{proof}[Proof of Theorem~\ref{th:inter}]
Let us construct a functor $\F\colon \PLS(\MA)\to \PLS(\MB)$. We naturally put
\begin{enumerate}[label=\arabic*.]
\item $\F(\emptyset)=\emptyset$; 
\item $\F(X/{\sim_X})=\bar X/{\sim_{\bar X}}$ for any non-empty projective logical set $X/{\sim_X}$ over $\MA$ (it is correct due to~\ref{F1}, \ref{F2}, \ref{F3} and~\ref{F4}); in particular, if $\sim_X$ is the identity relation $=$, then $\sim_{\bar X}$ is the relation $\sim_\Gamma$ from~\eqref{eq:rel_Gamma};
\item $\F(\emptyset \to O)=\emptyset\to \F(O)$ for any object $O$ in $\PLS(\MA)$,
\item $\F(F_{XY})=F_{\bar X \bar Y}$ for any morphism $F_{XY}$ from $\PLS(\MA)$ between non-empty objects (it is correct due to~\ref{F5} and~\ref{F6}).
\end{enumerate} 
Since for any non-empty projective logical set $X/{\sim_X}$ in $\PLS(\MA)$ the identical morphism $\id_{X/{\sim_X}}$ is induced by the formula $E_X$, then one has $\F(\id_{X/{\sim_X}})=\id_{\F(X/{\sim_X})}$. And the statement~\ref{F7} guarantees that $\F(F_{YZ}\circ F_{XY})=\F(F_{YZ})\circ \F(F_{XY})
$ for any any pair of suitable morphisms $F_{XY}$ and $F_{YZ}$ from $\PLS(\MA)$. Thus $\F$ satisfies all the axioms of a categorical functor. The conditions~\ref{F1}, \ref{E1} and~\ref{E2} show that the functor $\F$ is injective on objects; and \ref{E3} shows that the functor $\F$ is faithful, i.\,e., injective on hom-sets. So, $\F$ is an embedding. The result on the restrictions of the functor $\F$ follows from Remark~\ref{remark:a,b}.
\end{proof}

\begin{corollary}\label{cor:functor_subset}
Let $X/{\sim_{X}}$, $Z/{\sim_Z}$ and $X_j/{\sim_{X_j}}$, $j=1,\ldots,d$, be non-empty projective logical sets over $\MA$, $\sim_{1}$, $\sim_{2}$ formula equivalencies on $X$. Suppose that $\pi\colon X/{\sim_{1,X}}\to X/{\sim_{2,X}}$ is a natural surjection, $\varepsilon\colon Z/{\sim_Z}\to X/{\sim_X}$ is an identical embedding and $\Pi_r\colon X_1/{\sim_{X_1}}\times\ldots\times X_d/{\sim_{X_d}}\,\to X_r/{\sim_{X_r}}$ is a natural projection, $r\in\{1,\ldots,d\}$. Then one has 
    \begin{enumerate}[label=(\roman*)]
            \item\label{item1:subset} $\F(X_1/{\sim_{{X_1}}}\times \ldots \times X_d/{\sim_{X_d}})=\F(X_1/{\sim_{X_1}})\times\ldots\times\F(X_d/{\sim_{X_d}})$;
        \item\label{item2:subset} $\F(\Pi_r)\colon \F(X_1/{\sim_{X_1}})\times\ldots\times\F(X_d/{\sim_{X_d}})\to \F(X_r/{\sim_{X_r}})$ is a natural projection;
        \item\label{item4:subset} $\F(\pi)\colon \F(X/{\sim_{1}})\to \F(X/{\sim_{2}})$ is a natural surjection;
        \item\label{item3:subset} $\F(\varepsilon)\colon \F(Z/{\sim_Z})\to\F(X/{\sim_X})$ is an identical embedding.
    \end{enumerate}
    
\end{corollary}

\begin{proof}
All assertions follow from~\ref{F2} and the construction of the functor $\F$. In~\ref{item1:subset} we also use Remark~\ref{cor:wedge0}; in~\ref{item4:subset}~--- Fact~\ref{fact:surjection}; in~\ref{item3:subset}~--- Fact~\ref{fact:subset}. 
\end{proof}

\begin{remark}\label{remark:gamma}
Note that the construction of the functor $\F$ from Theorem~\ref{th:inter} is completely based on the $(\Gamma,\bar p,\mu_\Gamma)$\=/translation, or rather on the
$(\Gamma,\bar p,\gamma)$\=/translation of formulas in the language $L(\MA)\cup A$ into formulas in the language $L(\MB)\cup B$. If we take another map $\gamma^\prime\colon A\to U_\Gamma(\MB,\bar p)$, such that $\mu_\Gamma\circ\gamma^\prime=\id_A$, then we obtain that $\F_{\Gamma,\bar p,\gamma}=\F_{\Gamma,\bar p,\gamma^\prime}$, due to Corollary~\ref{cor:gamma}.
\end{remark}

The image of the restriction of the functor $\F$ on the subcategory $\PLS_{\alpha,\beta}(\MA)$ contains in $\PLS_{\alpha,\beta}(\MB)$. A similar statement about subcategories $\PLS_0(\MA)$ or $\LS(\MA)$ is not true in general, but only with the following specifications of the code $\Gamma$.

\begin{corollary}\label{cor:abs}
Let $(\Gamma,\emptyset,\mu_\Gamma)\colon\MA\rightsquigarrow\MB$ be an absolute interpretation and $\F=\F_{\Gamma,\emptyset,\mu_\Gamma}$ is the corresponding functor from Theorem~\ref{th:inter}. Then for any infinite cardinal $\beta$ one has $\F(\PLS_{0,\beta}(\MA))\subseteq \PLS_{0,\beta}(\MB)$, in particular,
\begin{gather*}
\F\colon\PLS_0(\MA)\to\PLS_0(\MB),\\
\F\colon\PDS_0(\MA)\to\PDS_0(\MB).
\end{gather*}
\end{corollary}

\begin{corollary}\label{cor:inj}
Let $(\Gamma,\bar p,\mu_\Gamma)\colon\MA\rightsquigarrow\MB$ be an injective interpretation and $\F=\F_{\Gamma,\bar p,\mu_\Gamma}$ is the corresponding functor from Theorem~\ref{th:inter}. Then for any infinite cardinals $\alpha\leqslant\beta$ one has $\F(\LS_{\alpha,\beta}(\MA))\subseteq \LS_{\alpha,\beta}(\MB)$, in particular,
\begin{gather*}
\F\colon\LS(\MA)\to\LS(\MB),\\
\F\colon\DS(\MA)\to\DS(\MB).
\end{gather*}
\end{corollary}

\begin{corollary}
Let $(\Gamma,\emptyset,\mu_\Gamma)\colon\MA\rightsquigarrow\MB$ be an absolute injective interpretation and $\F=\F_{\Gamma,\emptyset,\mu_\Gamma}$ is the corresponding functor from Theorem~\ref{th:inter}. Then for any infinite cardinal $\beta$ one has $\F(\LS_{0,\beta}(\MA))\subseteq \LS_{0,\beta}(\MB)$, in particular,
\begin{gather*}
\F\colon\LS_0(\MA)\to\LS_0(\MB),\\
\F\colon\DS_0(\MA)\to\DS_0(\MB).
\end{gather*}
\end{corollary}

\subsection{Translation functors in logical geometry}\label{subsec:functor}

As we see, the functor $\F$ from the proof of Theorem~\ref{th:inter} is a very special type of categorical functor. The basis of its construction and the justification of each step of this construction is the Reduction Theorem~\ref{RT}. Next, we want to understand under what conditions such a functor can exist between two categories $\PLS(\MA)$ and $\PLS(\MB)$, or just between $\PDS(\MA)$ and $\PDS(\MB)$, if we have no information about the existence of an interpretation $\MA\rightsquigarrow\MB$. Naturally, moving in the opposite direction, we will rely on Inverse Reduction Theorem~\ref{IRT}. We immediately note that we are interested only in non-degenerate functors due to
 Fact~\ref{fact:empty}. 

Let $(\Gamma,\bar p,\gamma)\colon L(\MA)\cup A\to L(\MB)\cup B$ be an extended code (see Definition~\ref{def:code}). 
Clearly, to define the non-degenerate functor $\F=\F_{\Gamma,\bar p,\gamma}\colon \PLS(\MA)\to \PLS(\MB)$ correctly, as we did it above, we need to require that the conditions~\ref{F1}--\ref{F7} be satisfied; and $\F$ will be an embedding if and only if the conditions~\ref{E1}--\ref{E3} hold. For any infinite cardinals $\alpha\leqslant\beta$ these conditions have natural weakened analogues; we denote them by~\ref{F1}$_{\alpha,\beta}$--\ref{F7}$_{\alpha,\beta}$ and \ref{E1}$_{\alpha,\beta}$--\ref{E3}$_{\alpha,\beta}$; they form a criterion for the functor $\F=\F_{\Gamma,\bar p,\gamma}\colon \PLS_{\alpha,\beta}(\MA)\to \PLS_{\alpha,\beta}(\MB)$ to exist and be an embedding. In the conditions~\ref{F1}$_{\alpha,\beta}$--\ref{F7}$_{\alpha,\beta}$ and \ref{E1}$_{\alpha,\beta}$--\ref{E3}$_{\alpha,\beta}$ all systems $S$ have less then $\beta$ formulas and less then $\alpha$ parameters from $A$. 

\begin{definition}
We refer to a functor $\F\colon \PLS_{\alpha,\beta}(\MA)\to \PLS_{\alpha,\beta}(\MB)$ as a {\em translation functor}, if there exists an extended code $(\Gamma,\bar p,\gamma)\colon L(\MA)\cup A\to L(\MB)\cup B$, such that \ref{F1}$_{\alpha,\beta}$--\ref{F7}$_{\alpha,\beta}$ hold, and $\F=\F_{\Gamma,\bar p,\gamma}$ is the functor, constructed as in the proof of Theorem~\ref{th:inter}, using the $(\Gamma,\bar p,\gamma)$\=/translation. In this case we will also say that $\F$ is the $(\Gamma,\bar p,\gamma)$\=/functor, or that the $(\Gamma,\bar p,\gamma)$\=/functor is well-defined. If the code $\Gamma$ is injective we refer to the restriction $\F\colon \LS_{\alpha,\beta}(\MA)\to \LS_{\alpha,\beta}(\MB)$ of $(\Gamma,\bar p,\gamma)$\=/functor as an {\em injective translation functor}. If the code $\Gamma$ is absolute we refer to the restriction $\F\colon \PLS_{0,\beta}(\MA)\to \PLS_{0,\beta}(\MB)$ of $(\Gamma,\emptyset,\gamma)$\=/functor as an {\em absolute translation functor} and name it {\em $\Gamma$\=/functor}. When $(\alpha,\beta)$ is $(\infty,\infty)$, we refer to translation functors as {\em interpretation functors}.
\end{definition}

It is clear that the restriction of the interpretation $(\Gamma,\bar p,\gamma)$\=/functor $\F\colon \PLS(\MA)\to\PLS(\MB)$ to the subcategory $\PLS_{\alpha,\beta}(\MA)$ is the translation $(\Gamma,\bar p,\gamma)$\=/functor $\F\colon \PLS_{\alpha,\beta}(\MA)\to\PLS_{\alpha,\beta}(\MB)$. Sometimes we will note such a restriction of the interpretation functor by $\F_{\alpha,\beta}$. All similar connections between translation functors are listed in Fact~\ref{fact:functor} below.

Mostly, we will consider interpretation functors and translation functors for the pair of cardinals $(\omega,\omega)$. If $\F$ is the interpretation functor of an interpretation $(\Gamma,\bar p,\mu_\Gamma)\colon\MA\rightsquigarrow\MB$, we will refer to it as the interpretation $(\Gamma,\bar p,\mu_\Gamma)$\=/functor and to its restriction to the category $\PDS(\MA)$ as the translation $(\Gamma,\bar p,\mu_\Gamma)$\=/functor.

The main goal of this subsection is to prove the following two theorems.

\begin{theorem}[about translation functors]\label{th:functor1}
For any algebraic structures $\MA=\langle A; L(\MA)\rangle$ and $\MB=\langle B; L(\MB)\rangle$ and any extended code $(\Gamma,\bar p,\gamma)\colon L(\MA)\cup A\to L(\MB)\cup B$ the following conditions are equivalent:
    \begin{enumerate}[label=(\arabic*)]
        \item\label{item:t1} there exist an elementary extension $\MA\preceq\nsA$ and an interpretation $(\Gamma,\bar p, \mu_\Gamma)\colon \nsA\rightsquigarrow\MB$, such that $\mu_\Gamma \circ\gamma=\id_A$;
        \item\label{item:t2} there exists a translation $(\Gamma,\bar p,\gamma)$\=/functor $\F\colon\PDS(\MA)\to \PDS(\MB)$;
        \item\label{item:t3} $(\Gamma,\bar p,\gamma)$ satisfies the condition~\ref{item:IRT} from Inverse Reduction Theorem~\ref{IRT}.
    \end{enumerate}
In this case, $\gamma(A)\subseteq U_\Gamma(\MB,\bar p)$ and $\F$ is an embedding.
\end{theorem}

\begin{theorem}[about interpretation functors]\label{th:functor2}
For any algebraic structures $\MA=\langle A; L(\MA)\rangle$ and $\MB=\langle B; L(\MB)\rangle$ the following hold:
    \begin{enumerate}[label=(\arabic*)]
        \item\label{item:IF1} For any interpretation $(\Gamma,\bar p, \mu_\Gamma)\colon \MA\rightsquigarrow\MB$, there exist an extended code $(\Gamma,\bar p,\gamma)\colon L(\MA)\cup A\to L(\MB)\cup B$, such that $\mu_\Gamma \circ\gamma=\id_A$, and the interpretation $(\Gamma,\bar p,\gamma)$\=/functor $\F\colon\PLS(\MA)\to \PLS(\MB)$;
            \item\label{item:IF2} For any extended code $(\Gamma,\bar p,\gamma)\colon L(\MA)\cup A\to L(\MB)\cup B$ if there exists the interpretation $(\Gamma,\bar p,\gamma)$\=/functor $\F\colon\PLS(\MA)\to \PLS(\MB)$, then there exists an interpretation $(\Gamma,\bar p, \mu_\Gamma)\colon \MA\rightsquigarrow\MB$, such that $\mu_\Gamma \circ\gamma=\id_A$.
    \end{enumerate}
In both cases, $\gamma(A)\subseteq U_\Gamma(\MB,\bar p)$ and $\F$ is an embedding, which is also a functor $\F\colon\PLS_{\alpha,\beta}(\MA)\to \PLS_{\alpha,\beta}(\MB)$ for any infinite cardinals $\alpha\leqslant\beta$.
\end{theorem}

We start the proofs of the theorems above with the following remark. 

\begin{remark}\label{remark:FE}
Logically, each of the conditions~\ref{F1}$_{\omega,\omega}$--\ref{F7}$_{\omega,\omega}$ and~\ref{E1}$_{\omega,\omega}$--\ref{E3}$_{\omega,\omega}$ for a given pair of algebraic structures $\MA,\MB$ and an extended code $(\Gamma,\bar p,\gamma)\colon L(\MA)\cup A\to L(\MB)\cup B$ has one of the following three forms
\begin{enumerate}[label=(FE)$_{\omega,\omega}$]
\item\label{EF} for any integer $m\in\N$ and any formulas $\phi_1,\ldots,\phi_d\in {\bf F}_{L(\MA)\cup A}(x_1,\ldots,x_m)$ one has 
$$
\MA_A\models \psi_{\MA} \ \underset{\text{\normalsize $\Longleftarrow$}}{\overset{\text{\normalsize $\Longrightarrow$}}{\Longleftrightarrow}} \ \MB_B\models \psi_\MB, 
$$
\end{enumerate}
where $\psi_\MA$ and $\psi_\MB$ are suitable sentences in the languages $L(\MA)\cup A$ and $L(\MB)\cup B$, constructed by formulas $\phi_1,\ldots,\phi_d$, using algorithms specified for each of the conditions \ref{F1}$_{\omega,\omega}$--\ref{F7}$_{\omega,\omega}$ and  \ref{E1}$_{\omega,\omega}$--\ref{E3}$_{\omega,\omega}$. Here for \ref{F5}$_{\omega,\omega}$ we use Lemma~\ref{fm}. 
\end{remark}

Let us consider an important particular case of a translation functor, when $\Gamma=\Id_{L(\MA)}$. It is an absolute code $\Gamma\colon L(\MA)\to L(\MA)$ of dimension $1$~\cite[Subsection~2.4]{Th_int1}. Suppose also that $\nsA=\langle \nsuA; L(\MA)\rangle$ is an $L(\MA)$\=/structure.

\begin{lemma}\label{lemma:iota}
For a map $\iota\colon A\to \nsuA$ the following conditions are equivalent:
    \begin{enumerate}[label=(\arabic*)]
    \item $\iota$ is an elementary $L(\MA)$\=/embedding;
    \item the translation $(\Id_{L(\MA)},\emptyset,\iota)$\=/functor $\F\colon \PDS(\MA)\to \PDS(\nsA)$ is well-defined.
\end{enumerate}
In this case functor $\F$ is an embedding.
\end{lemma}

\begin{proof}
    For any $L(\MA)$\=/formula $\psi$ the translation $\psi_{\Id_{L(\MA)}}$ is first-order equivalent to $\psi$~\cite[Remark~11]{Th_int1}. Therefore, each of the conditions~\ref{F1}$_{\omega,\omega}$--\ref{F7}$_{\omega,\omega}$ and  \ref{E1}$_{\omega,\omega}$--\ref{E3}$_{\omega,\omega}$ may be rewritten in one of the following three forms
    \begin{equation*}
      \MA\models \psi(a_1,\ldots,a_s) \ \underset{\text{\normalsize $\Longleftarrow$}}{\overset{\text{\normalsize $\Longrightarrow$}}{\Longleftrightarrow}} \  \nsA\models \psi (\iota(a_1),\ldots, \iota(a_s)),  
    \end{equation*}
where $\psi$ is a suitable $L(\MA)$\=/formula and $a_1,\ldots,a_s\in A$. Therefore, if $\iota$ is an elementarily embedding, then the conditions~\ref{F1}$_{\omega,\omega}$--\ref{F7}$_{\omega,\omega}$ and  \ref{E1}$_{\omega,\omega}$--\ref{E3}$_{\omega,\omega}$ hold, hence, the $(\Id_{L(\MA)},\emptyset,\iota)$\=/functor is well-defined and it is an embedding. Conversely, suppose there exists the $(\Id_{L(\MA)},\emptyset,\iota)$\=/functor $\F\colon \PDS(\MA)\to \PDS(\nsA)$, then the conditions~\ref{F1}$_{\omega,\omega}$--\ref{F7}$_{\omega,\omega}$ hold. In particular,~\ref{F1}$_{\omega,\omega}$ gives that the condition~\ref{EE} from Subsection~\ref{sec:IRT} holds, i.\,e., $\iota$ is an elementary $L(\MA)$\=/embedding.
\end{proof}

\begin{lemma}\label{lem1}
Let $\iota\colon \MA\to\nsA$ be an elementary $L(\MA)$\=/embedding,  $(\Gamma,\bar p,\mu_\Gamma)\colon\nsA\rightsquigarrow\MB$ an interpretation and $\gamma\colon \nsuA\to U_\Gamma(\MB,\bar p)$ a map, such that $\mu_\Gamma\circ\gamma=\id_{\nsuA}$, so that one has
$$
\begin{tikzcd}
\MA\arrow[r,"{\iota}"] &\nsA \arrow[r,rightsquigarrow,"{\Gamma,\bar p,\mu_\Gamma}"] & \MB.
\end{tikzcd}
$$
Then the translation $(\Gamma,\bar p,\gamma\circ\iota)$\=/functor $\F\colon \PDS(\MA)\to \PDS(\MB)$ is well-defined.
\end{lemma}

\begin{proof}
Since there exists an interpretation $(\Gamma,\bar p,\mu_\Gamma)\colon\nsA\rightsquigarrow\MB$ the conditions~\ref{F1}$_{\omega,\omega}$--\ref{F7}$_{\omega,\omega}$ and \ref{E1}$_{\omega,\omega}$--\ref{E3}$_{\omega,\omega}$ hold for the pair of algebraic structures $\nsA,\MB$ and the extended code $(\Gamma,\bar p,\gamma)\colon L(\MA)\cup \nsuA\to L(\MB)\cup B$. Let $\nsA_A$ be an $L(\MA)\cup A$\=/structure, where any constant symbol $a\in A$ is interpreted as $\iota(a)$. Then $\MA_A\equiv \nsA_A$, in particular, $\MA_A\models\psi_\MA$ if and only if $\nsA_{A}\models \psi_\MA$ for any formula $\psi_\MA$ from Remark~\ref{remark:FE}. Therefore, the conditions~\ref{F1}$_{\omega,\omega}$--\ref{F7}$_{\omega,\omega}$ and \ref{E1}$_{\omega,\omega}$--\ref{E3}$_{\omega,\omega}$ hold for the pair of algebraic structures $\MA,\MB$ and the extended code $(\Gamma,\bar p,\gamma\circ\iota)$. Hence, the translation $(\Gamma,\bar p,\gamma\circ\iota)$\=/functor $\F\colon \PDS(\MA)\to \PDS(\MB)$ is well-defined.
\end{proof}

\begin{proof}[Proof of Theorem~\ref{th:functor1}]
\ref{item:t1}$\Longrightarrow$\ref{item:t2}: Suppose that there exist an elementary extension $\MA\preceq\nsA$ and an interpretation $(\Gamma,\bar p, \mu_\Gamma)\colon \nsA\rightsquigarrow\MB$. Then by Lemma~\ref{lem1} there exists the translation $(\Gamma,\bar p,\gamma)$\=/functor $\F\colon \PDS(\MA)\to \PDS(\MB)$, where $\mu_\Gamma\circ\gamma=\id_A$; and $\F$ is an embedding. \ref{item:t2}$\Longrightarrow$\ref{item:t3}: If there exists a translation $(\Gamma,\bar p,\gamma)$\=/functor $\F\colon\PDS(\MA)\to \PDS(\MB)$, then the  conditions~\ref{F1}$_{\omega,\omega}$--\ref{F7}$_{\omega,\omega}$ holds for the pair of algebraic structures $\MA,\MB$ and the extended code $(\Gamma,\bar p,\gamma)\colon L(\MA)\cup A\to L(\MB)\cup B$. In particular, \ref{F1}$_{\omega,\omega}$ means that the condition~\ref{item:IRT} from Inverse Reduction Theorem~\ref{IRT} holds.  \ref{item:t3}$\Longrightarrow$\ref{item:t1}: If the condition~\ref{item:IRT} from Inverse Reduction Theorem~\ref{IRT} holds, then $\gamma(A)\subseteq U_\Gamma(\MB,\bar p)$ and the $L(\MA)$\=/structure $\Gamma(\MB,\bar p)$ is well-defined; and there exists an elementary $L(\MA)$\=/embedding $\iota\colon \MA\to \Gamma(\MB,\bar p)$, such that $\iota(a)=\gamma(a)/{\sim_\Gamma}$ for all $a\in A$. Let $\nsA$ be an elementary extension of $\MA$, which is isomorphic to $\Gamma(\MB,\bar p)$, and $\bar \iota\colon \nsA\to \Gamma(\MB,\bar p)$ be an $L(\MA)$\=/isomorphism, such that $\bar\iota(a)=\iota(a)$ for any $a\in A$. Hence, there exists an interpretation $(\Gamma,\bar p)\colon\nsA\to\MB$. If $\mu_\Gamma\colon U_\Gamma(\MB,\bar p)\to \nsuA$ is the coordinate map, such that $\bar\mu_\Gamma=\bar\iota^{-1}$, then $\mu_\Gamma\circ\gamma=\id_A$, as required.
\end{proof}

\begin{remark}\label{remark:gamma_tr}
Suppose that the condition~\ref{item:t1} from Theorem~\ref{th:functor1} holds. Then  the construction of the functor $\F\colon \PDS(\MA)\to \PDS(\MB)$ is completely based on the $(\Gamma,\bar p,\mu_\Gamma)$\=/translation of the interpretation $\nsA\rightsquigarrow\MB$. Therefore, by Remark~\ref{remark:gamma}, if $\gamma,\gamma^\prime\colon A\to U_\Gamma(\MB,\bar p)$ are functions, such that $\mu_\Gamma\circ\gamma=\mu_\Gamma\circ\gamma^\prime=\id_A$, then $\F_{\Gamma,\bar p,\gamma}=\F_{\Gamma,\bar p,\gamma^\prime}$. 
\end{remark}

\begin{corollary}\label{cor:inj1}
The $(\Gamma,\bar p,\gamma)$\=/translation functor $\F\colon \PDS(\MA)\to \PDS(\MB)$ is well-defined if and only if the condition~\ref{F1}$_{\omega,\omega}$ holds.
\end{corollary}

For translation functors, one has an analog of Corollary~\ref{cor:functor_subset}, which may be proved by a similar argument.

\begin{corollary}\label{cor:tr_functor_subset}
Let $X/{\sim_{X}}$, $Z/{\sim_Z}$ and $X_j/{\sim_{X_j}}$, $j=1,\ldots,d$, be non-empty projective definable sets over $\MA$, $\sim_{1}$, $\sim_{2}$ formula equivalencies on $X$. Suppose that $\pi\colon X/{\sim_{1,X}}\to X/{\sim_{2,X}}$ is a natural surjection, $\varepsilon\colon Z/{\sim_Z}\to X/{\sim_X}$ is an identical embedding and $\Pi_r\colon X_1/{\sim_{X_1}}\times\ldots\times X_d/{\sim_{X_d}}\,\to X_r/{\sim_{X_r}}$ is a natural projection, $r\in\{1,\ldots,d\}$. Then one has~\ref{item1:subset}--\ref{item3:subset} from Corollary~\ref{cor:functor_subset}.  
\end{corollary}
 
\begin{proof}[Proof of Theorem~\ref{th:functor2}]
Item~\ref{item:IF1} is given by Theorem~\ref{th:inter}. To prove the inverse implication~\ref{item:IF2} assume that there exists an interpretation $(\Gamma,\bar p,\gamma)$\=/functor $\F\colon\PLS(\MA)\to \PLS(\MB)$, then the  conditions~\ref{F1}--\ref{F7} holds for the pair of algebraic structures $\MA,\MB$ and the extended code $(\Gamma,\bar p,\gamma)\colon L(\MA)\cup A\to L(\MB)\cup B$. As above, the condition~\ref{F1} guarantees that the condition~\ref{item:IRT} from Inverse Reduction Theorem~\ref{IRT} holds. In particular, $\gamma(A)\subseteq U_\Gamma(\MB,\bar p)$. Further, the $L(\MA)$\=/structure $\Gamma(\MB,\bar p)$ is well-defined and there exist an elementary $L(\MA)$\=/embedding $\iota\colon \MA\to \Gamma(\MB,\bar p)$, such that 
$\iota(a)=\gamma(a)/{\sim_\Gamma}$ for all $a\in A$. Let us consider the system $S=\{\neg(x=a)\mid a\in A\}$ in one variable $x$; one has $S_{\Gamma,\bar p,\gamma}=\{U_\Gamma(\bar x,\bar p)\wedge \neg E_\Gamma(\bar x,\gamma(a),\bar p)\mid a\in A\}$. Since $\V_\MA(S)=\emptyset$, then $\V_\MB(S_{\Gamma,\bar p,\gamma})=\emptyset$, i.\,e., for any $\bar b\in U_\Gamma(\MB,\bar p)$ there exists $a\in A$, such that $\bar b\sim_\Gamma \gamma(a)$. Hence, $\iota\colon \MA\to \Gamma(\MB,\bar p)$ is surjective, therefore, it is an $L(\MA)$\=/isomorphism. Thus, there exists an interpretation $(\Gamma,\bar p)\colon \MA\rightsquigarrow \MB$. We chose the coordinate map $\mu_\Gamma\colon U_\Gamma(\MB,\bar p)\to A$, such that $\bar\mu_\Gamma=\iota^{-1}$, then $\mu_\Gamma\circ\gamma=\id_A$, as required.
\end{proof}

\begin{corollary}\label{cor:surj}
In the notations and assumptions of Theorem~\ref{th:functor2}, every element from $\Gamma(\MB,\bar p)$ has a form $\gamma(a)/{\sim_\Gamma}$ for some $a\in A$.
\end{corollary}

\begin{corollary}\label{cor:inj2}
The $(\Gamma,\bar p,\gamma)$\=/interpretation functor $\F\colon \PLS(\MA)\to \PLS(\MB)$ is well-defined if and only if the condition~\ref{F1} holds.
\end{corollary}

\begin{remark}\label{remark:gamma_int}
Theorem~\ref{th:functor2} defines two operators ${\bf F}$ and ${\bf I}$. The first ${\bf F}$ matches interpretation functors to interpretations: ${\bf F}((\Gamma,\bar p,\mu_\Gamma))= \F_{\Gamma,\bar p,\mu_\Gamma}$; and the second ${\bf I}$ matches interpretations to interpretation functors: ${\bf I}(\F)=(\Gamma,\bar p,\mu_\Gamma)$. Due to Remark~\ref{lemma:gamma}, for any interpretation $(\Gamma,\bar p,\mu_\Gamma)$ one has ${\bf I}({\bf F}((\Gamma,\bar p,\mu_\Gamma)))=(\Gamma,\bar p,\mu_\Gamma)$. Due to 
Remark~\ref{remark:gamma}, for any interpretation functor $\F$ one has ${\bf F}({\bf I}(\F))=\F$.
\end{remark}

\begin{remark}\label{remark:ID}
For any algebraic structure $\MA=\langle A; L(\MA)\rangle$ the identical extended code $(\Id_{L(\MA)},\emptyset,\id_A)\colon L(\MA)\cup A\to L(\MA)\cup A$ gives the identical interpretation and translation functors $\id_{\PLS(\MA)}$, $\id_{\LS(\MA)}$, $\id_{\PDS(\MA)}$, $\id_{\DS(\MA)}$, $\id_{\PLS_0(\MA)}$, $\id_{\LS_0(\MA)}$, $\id_{\PDS_0(\MA)}$, $\id_{\DS_0(\MA)}$, (see~\cite[Remark~11]{Th_int1}).
\end{remark}

\subsection{Compositions of translation functors}\label{subsec:ABC}

In conclusion to our discussion of transfer and interpretation functors, we formulate several results about compositions of such functors. In the constructions below, we refer to~\cite[Subsection~4.1]{Th_int1}. In particular, note that the composition of injective codes is an injective code~\cite[Lemma~9~(3)]{Th_int1}.

\begin{proposition}\label{th:ABC}
Let $(\Gamma,\bar p, \mu_\Gamma)\colon\MA\rightsquigarrow\MB$ and $(\Delta,\bar q, \mu_\Delta)\colon\MB\rightsquigarrow\MC$ be interpretations and $(\Gamma\circ\Delta,(\bar{\bar p},\bar q),\mu_{\Gamma\circ\Delta})\colon\MA\rightsquigarrow\MC$ be their composition, where $\mu_{\Gamma\circ\Delta}=\mu_\Gamma\circ\mu_\Delta$. Then the composition of the interpretation $(\Gamma,\bar p,\mu_\Gamma)$\=/functor  $\F\colon\PLS(\MA)\to\PLS(\MB)$ and the interpretation $(\Delta,\bar q,\mu_\Delta)$\=/functor  $\G\colon\PLS(\MB)\to\PLS(\MC)$ is the interpretation $(\Gamma\circ\Delta,(\bar{\bar p},\bar q),\mu_{\Gamma\circ\Delta})$\=/functor  $\G\circ\F\colon\PLS(\MA)\to\PLS(\MC)$. 
\end{proposition}

\begin{proof}
The result directly follows from Lemma~\ref{ABC}.  
\end{proof}

\begin{corollary}\label{cor16}
In the assumptions and notations of Proposition~\ref{th:ABC}, suppose that $X/{\sim_X}$ is a non-empty projective logical set over $\MA$ and $\tilde\mu_{\Gamma}\colon \F(X/{\sim_X})\to X/{\sim_X}$, $\tilde\mu_{\Delta}\colon \G(\F(X/{\sim_X}))\to \F(X/{\sim_X})$, $\tilde\mu_{\Gamma\circ\Delta}\colon \G\circ\F(X/{\sim_X})\to X/{\sim_X}$ are bijections from Corollary~\ref{cor3}. Then the following diagram is commutative:
\begin{equation}\label{eq:fish1}
\begin{tikzcd}[column sep=huge,
row sep=huge]
X/{\sim_X} 
\ar[rr, "\tilde\mu_{\Gamma\circ\Delta}^{-1}"', bend right=30] 
\ar[r,shift right, "\tilde\mu_\Gamma^{-1}"']
& \F(X/{\sim_X}) 
\ar[l,shift right, "\tilde\mu_\Gamma"'] 
\ar[r,shift right, "\tilde\mu_\Delta^{-1}"']
& \G\circ\F(X/{\sim_X}) 
\ar[l,shift right, "\tilde\mu_\Delta"']
\ar[ll, "\tilde\mu_{\Gamma\circ\Delta}"', bend right=30].
\end{tikzcd}
\end{equation}
\end{corollary}

\begin{proof}
Since bijections $\tilde\mu_{\Gamma}$, $\tilde\mu_{\Delta}$, $\tilde\mu_{\Gamma\circ\Delta}$ are formed coordinate-wise and $\mu_{\Gamma\circ\Delta}=\mu_\Gamma\circ\mu_\Delta$, we get that required. 
\end{proof}

Continuing the discussion in Example~\ref{ex:Gamma}, it is interesting to consider a special case
of Corollary~\ref{cor16}, when $X=A$.

\begin{corollary}\label{ex:Gamma_Delta}
In the assumptions and notations of Proposition~\ref{th:ABC}, in the diagram~\ref{eq:fish1} for $X=A$,
\begin{equation*}
\begin{tikzcd}[column sep=huge,
row sep=huge]
A 
\ar[rr, "\tilde\mu_{\Gamma\circ\Delta}^{-1}"', bend right=30] 
\ar[r,shift right, "\tilde\mu_\Gamma^{-1}"']
& U_\Gamma(\MB,\bar p)/{\sim_\Gamma} 
\ar[l,shift right, "\tilde\mu_\Gamma"'] 
\ar[r,shift right, "\tilde\mu_\Delta^{-1}"']
& U_{\Gamma\circ\Delta}(\MC,(\bar{\bar p},\bar q))/{\sim_{\Gamma\circ\Delta}} 
\ar[l,shift right, "\tilde\mu_\Delta"']
\ar[ll, "\tilde\mu_{\Gamma\circ\Delta}"', bend right=30],
\end{tikzcd}
\end{equation*}
bijections $\tilde \mu_\Gamma$, $\tilde\mu_\Delta$, $\tilde\mu_{\Gamma\circ\Delta}$ are $L(\MA)$\=/isomorphisms, such that $\tilde\mu_\Gamma=\bar\mu_\Gamma$, $\tilde\mu_{\Gamma\circ\Delta}=\bar\mu_{\Gamma\circ\Delta}$ and $\bar\mu_{\Gamma\circ\Delta}=\bar\mu_\Gamma\circ\tilde\mu_\Delta$.
\end{corollary}

We use the composition of extended codes in the following results.

\begin{proposition}\label{prop:comp1}
The composition of an interpretation $(\Gamma,\bar p,\gamma)$\=/functor $\F\colon\PLS(\MA)\to\PLS(\MB)$ and an interpretation $(\Delta,\bar q,\delta)$\=/functor  $\G\colon\PLS(\MB)\to \PLS(\MC)$ in the interpretation $(\Gamma\circ\Delta,(\delta(\bar p),\bar q),\delta\circ\gamma)$\=/functor $\G\circ\F\colon \PLS(\MA)\to \PLS(\MC)$. 
\end{proposition}

\begin{proof}
By Theorem~\ref{th:functor2} there exists interpretations $(\Gamma,\bar p, \mu_\Gamma)\colon\MA\rightsquigarrow\MB$ and $(\Delta,\bar q, \mu_\Delta)\colon\MB\rightsquigarrow\MC$, such that $\F$ is the interpretation $(\Gamma,\bar p, \mu_\Gamma)$\=/functor and $\G$ is the interpretation $(\Delta,\bar q, \mu_\Delta)$\=/functor. Therefore, by Proposition~\ref{th:ABC}, the composition $\G\circ\F$ is the interpretation $(\Gamma\circ\Delta,(\bar{\bar p},\bar q),\mu_{\Gamma}\circ\mu_{\Delta})$\=/functor. Since the $(\Gamma\circ\Delta,(\delta(\bar p),\bar q),\delta\circ\gamma)$\=/translation is the same as the $(\Gamma\circ\Delta,(\bar{\bar p},\bar q),\mu_{\Gamma}\circ\mu_{\Delta})$\=/translation, we obtain the required. 
\end{proof}

\begin{proposition}\label{prop:comp2}
The composition of a translation $(\Gamma,\bar p,\gamma)$\=/functor $\F\colon\PDS(\MA)\to\PDS(\MB)$ and a translation $(\Delta,\bar q,\delta)$\=/functor  $\G\colon\PDS(\MB)\to \PDS(\MC)$ in the translation $(\Gamma\circ\Delta,(\delta(\bar p),\bar q),\delta\circ\gamma)$\=/functor $\G\circ\F\colon \PDS(\MA)\to \PDS(\MC)$. 
\end{proposition}

\begin{proof}
By Theorem~\ref{th:functor1}, the conditions~\ref{item:IRT} from Inverse Reduction Theorem~\ref{IRT} hold for extended codes $(\Gamma,\bar p,\gamma)\colon L(\MA)\cup A\to L(\MB)\cup B$ and $(\Delta,\bar q,\delta)\colon L(\MB)\cup B \to L(\MC)\cup C$. Therefore, by Proposition~\ref{ABC_ext}, the condition~\ref{item:IRT} holds for the extended code $(\Gamma\circ\Delta,(\delta(\bar p),\bar q),\delta\circ\gamma)\colon L(\MA)\cup A\to L(\MC)\cup C$. Hence, the translation $(\Gamma\circ\Delta,(\delta(\bar p),\bar q),\delta\circ\gamma)$\=/functor $\G\circ\F\colon \PDS(\MA)\to \PDS(\MC)$ is well-defined, by Theorem~\ref{th:functor1}. And by the identity~\eqref{eq:ABC} from Proposition~\ref{ABC_ext}, functor $\G\circ\F$ equals to the composition $\G\circ\F$.
\end{proof}

\begin{problem}
Is it true that the composition of $(\alpha,\beta)$\=/translation functors is an $(\alpha,\beta)$\=/translation functor, when $\alpha\leqslant\beta$ are infinite cardinals, and $(\alpha,\beta)\ne (\omega,\omega), (\infty,\infty)$?
\end{problem}

\subsection{Homotopic interpretations and natural isomorphisms}\label{subsec:hom}

In this subsection, we discuss the connection between interpretation functors, corresponding to homotopic interpretations. In our presentation we rely on~\cite[Subsection~3.1]{Th_int1} and notations from Subsection~\ref{subsec:int_func}.

{\bf Strong homotopy implies natural isomorphism.} 
Remind that every homotopy of interpretations may be edited to become strong~\cite[Remark~19]{Th_int1}. Let us prove the following theorem.

\begin{theorem}\label{th:homotopy}
Let $(\Gamma_1,\bar p_1,\mu_{\Gamma_1})$ and $(\Gamma_2,\bar p_2,\mu_{\Gamma_2})$ be strongly homotopic interpretations of an algebraic structure $\MA=\langle A; L(\MA)\rangle$ into an algebraic structure $\MB=\langle B; L(\MB)\rangle$, and $\lambda\colon\Gamma_1(\MB,\bar p_1)\to\Gamma_2(\MB,\bar p_2)$ be homotopy isomorphism, i.\,e., $L(\MA)$\=/isomorphism definable in $\MB$, which makes the diagram 
\begin{equation}\label{eq:homotopic}
\begin{tikzcd}
& \Gamma_1(\MB,\bar p_1) \ar[dl, "\bar{\mu}_{\Gamma_1}"'] 
\ar[dd, shift right, "\lambda"']\\
\MA  & \\
& \Gamma_2(\MB,\bar p_2) 
\ar[ul, "\bar{\mu}_{\Gamma_2}"] 
\ar[uu, shift right, "\lambda^{-1}"']
\end{tikzcd} 
\end{equation} 
commutative. Suppose also that $\F_1\colon \PLS(\MA)\to\PLS(\MB)$ is the interpretation $(\Gamma_1,\bar p_1,\mu_{\Gamma_1})$\=/functor and $\F_2\colon \PLS(\MA)\to\PLS(\MB)$ is the interpretation $(\Gamma_2,\bar p_2,\mu_{\Gamma_2})$\=/functor. Then there exists a natural isomorphism $\tilde\lambda\colon\F_1\to\F_2$, such that for any non-empty projective logical set $X/{\sim_X}$ over $\MA$ the following diagram is commutative:
\begin{equation}\label{eq:homotopic2}
\begin{tikzcd}[column sep=huge, row sep=huge]
& \F_1(X/{\sim_X})
\ar[dl, shift right, "\tilde{\mu}_{\Gamma_1,X/{\sim_X}}"'] 
\ar[dd, shift right, "\tilde\lambda_{X/{\sim_X}}"']\\
X/{\sim_X} 
\ar[ur, shift right, "\tilde{\mu}^{-1}_{\Gamma_1,X/{\sim_X}}"'] 
\ar[dr, shift left, "\tilde{\mu}^{-1}_{\Gamma_2,X/{\sim_X}}"] 
& \\
& \F_2(X/{\sim_X}) 
\ar[ul, shift left, "\tilde{\mu}_{\Gamma_2, X/{\sim_X}}"] 
\ar[uu, shift right, "\tilde\lambda^{-1}_{X/{\sim_X}}"'],
\end{tikzcd} 
\end{equation}
where $\tilde\lambda_{X/{\sim_X}}\colon \F_1(X/{\sim_X})\to\F_2(X/{\sim_X})$ is isomorphism in $\PLS(\MB)$, which corresponds to $X/{\sim_X}$, and $\tilde\mu_{\Gamma_1,X/{\sim_X}}$, $\tilde\mu_{\Gamma_2,X/{\sim_X}}$ are bijections from Corollary~\ref{cor3}.
\end{theorem}

\begin{proof}
We will use the notation $\theta$ for a connector of the homotopy, i.\,e., a formula $\theta(\bar x_1,\bar x_2)$, $|\bar x_j|=\dim\Gamma_j$, $j=1,2$, in $L(\MB)\cup B$, which defines the isomorphism $\lambda$. 
Take any projective logical sets $X/{\sim_X}$ and $Y/{\sim_Y}$ from $\PLS(\MA)$ and any morphism $F_{XY}\colon X/{\sim_X} \to Y/{\sim_Y}$. For a brevity let us use the following notations: $\bar X_1/{\sim_{\bar X_1}}=\F_1(X/{\sim_X})$, $\bar X_2/{\sim_{\bar X_2}}=\F_2(X/{\sim_X})$, $\bar Y_1/{\sim_{\bar Y_1}}=\F_1(Y/{\sim_Y})$, $\bar Y_2/{\sim_{\bar Y_2}}=\F_2(Y/{\sim_Y})$, $F_{\bar X_1 \bar Y_1}=\F_1(F_{XY})$ and $F_{\bar X_2 \bar Y_2}=\F_2(F_{XY})$. 
To show the existence of natural isomorphism $\tilde\lambda$, we need to construct isomorphisms $\tilde\lambda_X$ for each projective logical set $X/{\sim_X}$ in $\PLS(\MA)$, such that for every $Y/{\sim_Y}$ in $\PLS(\MA)$ and any morphism $F_{XY}$ the diagram
\begin{equation}\label{eq:tau}
\begin{tikzcd}[column sep=huge, row sep=large]
\bar X_1/{\sim_{\bar X_1}} 
\arrow[r,leftarrow,shift right,"\tilde\lambda_X^{-1}"']
\arrow[r,rightarrow,shift left,swap,"\tilde\lambda_X"'] 
\arrow{d}{F_{\bar X_1 \bar Y_1}} 
& \bar X_2/{\sim_{\bar X_2}}
\arrow{d}{F_{\bar X_2 \bar Y_2}}\\
\bar Y_1/{\sim_{\bar Y_1}}
\arrow[r,leftarrow,shift right,"\tilde\lambda_Y^{-1}"']
\arrow[r,rightarrow,shift left,swap,"\tilde\lambda_Y"']  
& \bar Y_2/{\sim_{\bar Y_2}} 
\end{tikzcd}    
\end{equation}
will be commutative~\cite[Section~1.4]{MacLane}. 

For $X=\emptyset$ we put $\tilde\lambda_\emptyset=\id_\emptyset$. For non-empty $X\subseteq A^m$ we denote by $F_{\bar X_1 \bar X_2}\colon \bar X_1/{\sim_{\bar X_1}}\to \bar X_2/{\sim_{\bar X_2}}$ the composition of bijective maps $\tilde\mu^{-1}_{\Gamma_2, X/{\sim_X}}\circ\tilde\mu_{\Gamma_1, X/{\sim_X}}$. In the special case, when $\sim_X$ is the identical formula equivalence and the corresponding equivalence $\sim_{\bar X_1}$ coincides with $\sim_{\Gamma_1}$, and $\sim_{\bar X_2}$~--- with $\sim_{\Gamma_2}$ from~\eqref{eq:rel_Gamma}, we will denote $F_{\bar X_1 \bar X_2}$ by $F^\prime_{\bar X_1 \bar X_2}$. Then we put
\begin{equation}\label{eq:F}
   \begin{gathered}
   \theta_m(\bar {\bar x}_1, \bar{\bar x}_2)=\bigwedge\limits_{i=1}^m \theta(\bar x_{1,i},\bar x_{2,i}), \\ 
\varphi_{\bar X_1 \bar X_2}(\bar {\bar x}_1, \bar{\bar x}_2)= \exists\,\bar{\bar x}\:( \theta_m(\bar {\bar x}_1, \bar{\bar x})\wedge E_{\bar X_2}(\bar {\bar x},\bar {\bar x}_2)),
\end{gathered} 
\end{equation}
where $\bar{\bar x}_j=(\bar x_{j,1},\ldots,\bar x_{j,m})$, $|\bar x_{j,i}|=\dim\Gamma_j$, $j=1,2$. Take any tuples $\bar{\bar b}_1=(\bar b_{1,1},\ldots,\bar b_{1,m})\in B^{m\cdot\dim\Gamma_1}$ and $\bar{\bar b}_2=(\bar b_{2,1},\ldots,\bar b_{2,m})\in B^{m\cdot\dim\Gamma_2}$. One has $\MB_B\models\theta_m(\bar{\bar b}_1,\bar{\bar b}_2)$ if and only if $\bar b_{1,i}\in U_{\Gamma_1}(\MB,\bar p_1)$, $\bar b_{2,i}\in U_{\Gamma_2}(\MB,\bar p_2)$ and $\lambda(\bar b_{1,i}/{\sim_{\Gamma_1}})=\bar b_{2,i}/{\sim_{\Gamma_2}}$ for all $i=1,\ldots,m$. Due to~\eqref{eq:homotopic}, one has $\lambda(\bar b_{1,i}/{\sim_{\Gamma_1}})=\bar b_{2,i}/{\sim_{\Gamma_2}}$ if and only if $\mu_{\Gamma_1}(\bar b_{1,i})=\mu_{\Gamma_2}(\bar b_{2,i})$. By definition, for $\bar{\bar b}_1\in \bar X_1$, $\bar{\bar b}_2\in \bar X_2$ one has $F_{\bar X_1 \bar X_2}(\bar{\bar b}_1/{\sim_{\Gamma_1}})=\bar{\bar b}_2/{\sim_{\Gamma_2}}$ if and only if $\tilde\mu_{\Gamma_1,X/{\sim_X}}(\bar{\bar b}_1/{\sim_{\Gamma_1}})=\tilde\mu_{\Gamma_2,X/{\sim_X}}(\bar{\bar b}_2/{\sim_{\Gamma_2}})$, i.\,e.,  $\mu_{\Gamma_1}(\bar b_{1,i})=\mu_{\Gamma_2}(\bar b_{2,i})$ for all $i=1,\ldots,m$. So, the condition $F_{\bar X_1 \bar X_2}(\bar{\bar b}_1/{\sim_{\Gamma_1}})=\bar{\bar b}_2/{\sim_{\Gamma_2}}$ is equivalent to $\MB_B\models\theta_m(\bar{\bar b}_1,\bar{\bar b}_2)$. If $\MB_B\models\theta_m(\bar{\bar b}_1,\bar{\bar b}_2)$ and, for instance, $\bar {\bar b}_1\in \bar X_1$, but $\bar {\bar b}_2\not\in \bar X_2$, then $\mu_{\Gamma_1}(\bar{\bar b}_1)=\mu_{\Gamma_2}(\bar{\bar b}_2)$. However, by Fact~\ref{fact1}, $\bar {\bar b}_1\in \bar X_1$ implies that $\mu_{\Gamma_1}(\bar{\bar b}_1)\in X$ and $\bar {\bar b}_2\not\in \bar X_2$ implies that $\mu_{\Gamma_2}(\bar{\bar b}_2)\not\in X$. Thus, $\bar{\bar b}_1\in \bar X_1$ if and only if $\bar{\bar b}_2\in \bar X_2$, i.\,e., formula $\theta_m$ fully induces the map $F^\prime_{\bar X_1 \bar X_2}\colon \bar X_1/{\sim_{\Gamma_1}}\to \bar X_2/{\sim_{\Gamma_2}}$. So, by Lemma~\ref{le:iso}, $F^\prime_{\bar X_1 \bar X_2}$ is an isomorphism in $\PLS(\MB)$. 

In general, let us consider  the map $G=\pi\circ F^\prime_{\bar X_1 \bar X_2}\colon \bar X_1/{\sim_{\Gamma_1}}\to \bar X_2/{\sim_{\bar X_2}}$, where $\pi\colon \bar X_2/{\sim_{\Gamma_2}} \to \bar X_2/{\sim_{\bar X_2}}$ is the natural surjection, fully induced by the formulas $E_{\bar X_2}$ (see Fact~\ref{fact:surjection}). By Lemma~\ref{lemma:composition}, $G$ is a full surjective formula map, which is fully induced by the formula $\varphi_{\bar X_1 \bar X_2}$ from~\eqref{eq:F}. By Lemma~\ref{kernel}, $G$ gives rise the formula equivalence $\ker G$ on $\bar X_1$. Due to~\eqref{eq:homotopic} and Fact~\ref{fact1.5}, $\ker G$ coincides with $\sim_{\bar X_1}$. Therefore, by Theorem~\ref{th_hom}, the formula $\varphi_{\bar X_1 \bar X_2}$ fully induces the isomorphism $F_{\bar X_1 \bar X_2}$. 

So, we put $\tilde\lambda_{X/{\sim_X}}=F_{\bar X_1 \bar X_2}$. And along the way, we obtain the commutativity of the diagram~\eqref{eq:homotopic2}. 
It remains only to verify the commutativity of the diagram~\eqref{eq:tau}, and since $\tilde\lambda_{X/{\sim_X}}$, $\tilde\lambda_{Y/{\sim_Y}}$ are isomorphisms, it suffices to do this just in the southeast direction. We use here the commutativity of the diagrams~\eqref{eq:F_mu} and~\eqref{eq:homotopic2}:
\begin{multline*}
\tilde\lambda_{Y/{\sim_Y}}\circ F_{\bar X_1 \bar Y_2}=(\tilde\mu_{\Gamma_2,Y/{\sim_Y}}^{-1}\circ \tilde\mu_{\Gamma_1,Y/{\sim_Y}})\circ (\tilde\mu_{\Gamma_1,Y/{\sim_Y}}^{-1}\circ F_{XY}\circ\tilde\mu_{\Gamma_1,X/{\sim_X}})=  \\
=\tilde\mu_{\Gamma_2,Y/{\sim_Y}}^{-1}\circ F_{XY}\circ\tilde\mu_{\Gamma_1,X/{\sim_X}}=\\
=(\tilde\mu_{\Gamma_2,Y/{\sim_Y}}^{-2}\circ F_{XY}\circ\tilde\mu_{\Gamma_2,X/{\sim_X}})\circ (\tilde\mu_{\Gamma_2,X/{\sim_X}}^{-1}\circ \tilde\mu_{\Gamma_1,X/{\sim_X}}) =F_{\bar X_2 \bar Y_2}\circ \tilde\lambda_{X/{\sim_X}}.
\end{multline*}
This completes the proof of the theorem.
\end{proof}

\begin{corollary}\label{cor:homotopy}
 If, in the notations and assumptions of Theorem~\ref{th:homotopy}, $\theta$ is a formula in $L(\MB)\cup B$, which defines the homotopy isomorphism $\lambda$, then the isomorphism $\tilde\lambda_{A}$ is defined by the formula $\theta$.
\end{corollary}

{\bf Natural isomorphism implies strong homotopy.} 
Now we want to explore the reverse situation, that is, to understand what the existence of a natural isomorphism $\eta\colon \F_1\to\F_2$ between interpretation functors $\F_1,\F_2\colon\PLS(\MA)\to\PLS(\MB)$ or between translation functors $\F_1,\F_2\colon\PDS(\MA)\to\PDS(\MB)$  gives. We start with natural isomorphisms between translation functors. Remind that isomorphisms in the categories $\PDS(\MB)$ and $\PLS(\MB)$ between non-empty projective definable sets are all definable bijective maps (see Corollary~\ref{cor:iso_def}). The main goal here is to prove the following two theorems. 

\begin{theorem}\label{th:natur_iso1}
Let $\MA=\langle A; L(\MA)\rangle$ and $\MB=\langle B; L(\MB)\rangle$ be algebraic structures, $(\Gamma_1,\bar p_1,\gamma_1),(\Gamma_2,\bar p_2,\gamma_2)\colon L(\MA)\cup A\to L(\MB)\cup B$ extended codes, such that there exist the translation $(\Gamma_1,\bar p_1,\gamma_1)$- and $(\Gamma_2,\bar p_2,\gamma_2)$\=/functors $\F_1,\F_2\colon\PDS(\MA)\to\PDS(\MB)$, and a natural isomorphism $\eta\colon \F_1\to\F_2$ between them. Then there exist an elementary extension $\MA\preceq \nsA$ and strongly homotopic interpretations $(\Gamma_1,\bar p_1,\mu_{\Gamma_1}),(\Gamma_2,\bar p_2,\mu_{\Gamma_2})\colon \nsA \rightsquigarrow \MB$, such that $\mu_{\Gamma_1} \circ\gamma_1=\mu_{\Gamma_2} \circ\gamma_2=\id_A$. Moreover, if $\theta$ is a formula in the language $L(\MB)\cup B$, which defines the categorical isomorphism $\lambda=\eta_A\colon \F_1(A)\to \F_2(A)$, then it is a connector of the homotopy $(\theta)\colon (\Gamma_1,\bar p_1,\mu_{\Gamma_1})\to(\Gamma_2,\bar p_2,\mu_{\Gamma_2})$. In other words, $\lambda$ is also an $L(\MA)$\=/isomorphism and the diagram   
\begin{equation*}
\begin{tikzcd}
& \Gamma_1(\MB,\bar p_1) \ar[dl, "\bar{\mu}_{\Gamma_1}"'] 
\ar[dd, shift right, "\lambda"']\\
\MA\preceq\nsA  & \\
& \Gamma_2(\MB,\bar p_2) 
\ar[ul, "\bar{\mu}_{\Gamma_2}"] 
\ar[uu, shift right, "\lambda^{-1}"']
\end{tikzcd} 
\end{equation*} 
is commutative.    
\end{theorem}

\begin{theorem}\label{th:natur_iso2}
Let $\MA=\langle A; L(\MA)\rangle$ and $\MB=\langle B; L(\MB)\rangle$ be algebraic structures, $(\Gamma_1,\bar p_1,\gamma_1),(\Gamma_2,\bar p_2,\gamma_2)\colon L(\MA)\cup A\to L(\MB)\cup B$ extended codes, such that there exist the interpretation $(\Gamma_1,\bar p_1,\gamma_1)$- and $(\Gamma_2,\bar p_2,\gamma_2)$\=/functors $\F_1,\F_2\colon\PLS(\MA)\to\PLS(\MB)$, and a natural isomorphism $\eta\colon \F_1\to\F_2$ between them. Then there exist strongly homotopic interpretations $(\Gamma_1,\bar p_1,\mu_{\Gamma_1}),(\Gamma_2,\bar p_2,\mu_{\Gamma_2})\colon \MA \rightsquigarrow \MB$, such that $\mu_{\Gamma_1} \circ\gamma_1=\mu_{\Gamma_2} \circ\gamma_2=\id_A$. Moreover, if $\theta$ is a formula in the language $L(\MB)\cup B$, which defines the categorical isomorphism $\lambda=\eta_A\colon \F_1(A)\to \F_2(A)$, then it is a connector of the homotopy $(\theta)\colon (\Gamma_1,\bar p_1,\mu_{\Gamma_1})\to(\Gamma_2,\bar p_2,\mu_{\Gamma_2})$. In other words, $\lambda$ is also an $L(\MA)$\=/isomorphism and the diagram~\eqref{eq:homotopic} is commutative.  
\end{theorem}

We start proving with auxiliary lemmas.

\begin{lemma}\label{le:natur_iso0}
    In the assumptions and notations of Theorem~\ref{th:natur_iso1}, for any $m\in \N$ the categorical isomorphism $\eta_{A^m}\colon \F_1(A^m)\to \F_2(A^m)$ is defined by the formula $\theta_m$~\eqref{eq:F}.
\end{lemma}

\begin{proof}
Remind that by construction of functors $\F_1$ and $\F_2$ one has $\F_1(A)=U_{\Gamma_1}(\MB,\bar p_1)/{\sim_{\Gamma_1}}$ and $\F_2(A)=U_{\Gamma_2}(\MB,\bar p_2)/{\sim_{\Gamma_2}}$. Let's consider natural projections $\Pi_r\colon A^m\to A$, $r=1,\ldots,m$. By Corollary~\ref{cor:tr_functor_subset}, one has $\F_1(A^m)=(U_{\Gamma_1}(\MB,\bar p_1)/{\sim_{\Gamma_1}})^m$ and $\F_2(A^m)=(U_{\Gamma_2}(\MB,\bar p_2)/{\sim_{\Gamma_2}})^m$. Also it gives that $\F_1(\Pi_r)\colon (U_{\Gamma_1}(\MB,\bar p_1)/{\sim_{\Gamma_1}})^m\to U_{\Gamma_1}(\MB,\bar p_1)/{\sim_{\Gamma_1}}$ and $\F_2(\Pi_r)\colon (U_{\Gamma_2}(\MB,\bar p_2)/{\sim_{\Gamma_2}})^m\to U_{\Gamma_2}(\MB,\bar p_2)/{\sim_{\Gamma_2}}$ are natural projections $\Pi_r$. Therefore, the commutativity of the diagrams
\begin{equation*}
\begin{tikzcd}[column sep=huge, row sep=large]
(U_{\Gamma_1}(\MB,\bar p_1)/{\sim_{\Gamma_1}})^m 
\arrow[r,leftarrow,shift right,"\eta_{A^m}^{-1}"']
\arrow[r,rightarrow,shift left,swap,"\eta_{A^m}"'] 
\arrow{d}{\Pi_r} 
& (U_{\Gamma_2}(\MB,\bar p_2)/{\sim_{\Gamma_2}})^m 
\arrow{d}{\Pi_r}\\
U_{\Gamma_1}(\MB,\bar p_1)/{\sim_{\Gamma_1}}
\arrow[r,leftarrow,shift right,"\lambda^{-1}"']
\arrow[r,rightarrow,shift left,swap,"\lambda"']  
& U_{\Gamma_2}(\MB,\bar p_2)/{\sim_{\Gamma_2}}
\end{tikzcd}    
\end{equation*}
for all $r=1,\ldots,m$ gives that the isomorphism $\eta_{A^m}$ equals to $\lambda^m$, so it is defined by the formula $\theta_m$~\eqref{eq:F}.
\end{proof}

\begin{lemma}\label{le:natur_iso1}
    In the assumptions and notations of Theorem~\ref{th:natur_iso1}, suppose that $X\subseteq A^m$ is a non-empty definable set over $\MA$. Then the categorical isomorphism $\eta_X\colon \F_1(X)\to \F_2(X)$ is fully induced by the formula $\theta_m$~\eqref{eq:F}.
\end{lemma}

\begin{proof}
Suppose that $\varepsilon\colon X\to A^m$ is the identical embedding. By Corollary~\ref{cor:tr_functor_subset}, $\F_1(\varepsilon)\colon \F_1(X)\to \F_1(A^m)$ and $\F_2(\varepsilon)\colon \F_2(X)\to \F_2(A^m)$ are identical embeddings too. The commutativity of the diagram
\begin{equation*}
\begin{tikzcd}[column sep=huge, row sep=large]
\F_1(X) 
\arrow[r,leftarrow,shift right,"\eta_{X}^{-1}"']
\arrow[r,rightarrow,shift left,swap,"\eta_{X}"'] 
\arrow{d}{\F_1(\varepsilon)} 
& \F_2(X) 
\arrow{d}{\F_2(\varepsilon)}\\
(U_{\Gamma_1}(\MB,\bar p_1)/{\sim_{\Gamma_1}})^m 
\arrow[r,leftarrow,shift right,"\eta_{A^m}^{-1}"']
\arrow[r,rightarrow,shift left,swap,"\eta_{A^m}"']  
& (U_{\Gamma_2}(\MB,\bar p_2)/{\sim_{\Gamma_2}})^m
\end{tikzcd}    
\end{equation*}
shows that $\eta_X$ is
the restriction of $\lambda^m$ on $\F_1(X)$ and $\eta_X(\F_1(X))=\F_2(X)$. Therefore, by Lemma~\ref{lemma:descent}, $\eta_X$ is fully induced by the formula $\theta_m$~\eqref{eq:F}.  
\end{proof}

Thus, we see that the categorical isomorphism $\lambda=\eta_A\colon U_{\Gamma_1}(\MB,\bar p_1)/{\sim_{\Gamma_1}}\to U_{\Gamma_2}(\MB,\bar p_2)/{\sim_{\Gamma_2}}$ is a bijective map, which is defined by the formula $\theta$.

\begin{lemma}\label{le:natur_iso2}
    In the assumptions and notations of Theorem~\ref{th:natur_iso1}, algebraic $L(\MA)$\=/structures $\Gamma_1(\MB,\bar p_1)$ and $\Gamma_2(\MB,\bar p_2)$ are well-defined and the categorical isomorphism $\lambda$ is an $L(\MA)$\=/isomorphism between them.
\end{lemma}

\begin{proof}
Algebraic $L(\MA)$\=/structures $\Gamma_1(\MB,\bar p_1)$ and $\Gamma_2(\MB,\bar p_2)$ are well-defined due to Theorem~\ref{th:functor1}. 
We need to check that $\lambda$ preserves constants, functions, and predicates from the language $L(\MA)$. Let us use Lemma~\ref{le:natur_iso1} for it. 
For a constant symbol $c\in L(\MA)$ and $X=\V_\MA(\{x=c\})$ we obtain that the formula $\theta$ induces a categorical isomorphism  between projective definable sets $c_{\Gamma_1}(\MB,\bar p_1)/{\sim_{\Gamma_1}}$ and $c_{\Gamma_2}(\MB,\bar p_2)/{\sim_{\Gamma_2}}$. It means that $\lambda(c)=c$. Similarly, for a functional symbol $f\in L(\MA)$ and $X=\V_\MA(\{f(x_1,\ldots,x_{n_f})=x_0\})$ we obtain that the formula $\theta_{n_f+1}$ induces a categorical isomorphism  between projective definable sets $f_{\Gamma_1}(\MB,\bar p_1)/{\sim_{\Gamma_1}}$ and $f_{\Gamma_2}(\MB,\bar p_2)/{\sim_{\Gamma_2}}$; and the same for a predicate symbol $R\in L(\MA)$. It follows that $\lambda$ is an $L(\MA)$\=/isomorphism.
\end{proof}

\begin{lemma}\label{le:natur_iso3}
In the assumptions and notations of Theorem~\ref{th:natur_iso1}, one has $\lambda(\gamma_1(a)/{\sim_{\Gamma_1}})=\gamma_2(a)/{\sim_{\Gamma_2}}$ for every $a\in A$.
\end{lemma}

\begin{proof}
Take an element $a\in A$. When applying  Lemma~\ref{le:natur_iso1} to the definable set $X=\V_\MA(\{x=a\})$, we see that the formula $\theta$ induces a categorical isomorphism between projective definable sets $\F_1(X)=(U_{\Gamma_1}(\MB,\bar p_1)\cap E_{\Gamma_1}(\MB,\gamma_1(a),\bar p))/{\sim_{\Gamma_1}}$ and $\F_2(X)=(U_{\Gamma_2}(\MB,\bar p_2)\cap E_{\Gamma_2}(\MB,\gamma_2(a),\bar p))/{\sim_{\Gamma_2}}$. Since $\F_1(X)=\{\gamma_1(a)/{\sim_{\Gamma_1}}\}$ and $\F_2(X)=\{\gamma_2(a)/{\sim_{\Gamma_2}}\}$, we conclude that  $\lambda(\gamma_1(a)/{\sim_{\Gamma_1}})=\gamma_2(a)/{\sim_{\Gamma_2}}$.
\end{proof}

\begin{proof}[Proof of Theorem~\ref{th:natur_iso1}]
By Theorem~\ref{th:functor1} there exist an elementary extension $\MA\preceq \nsA$ and an interpretation $(\Gamma_1,\bar p_1,\mu_{\Gamma_1})\colon \nsA\rightsquigarrow \MB$, such that $\mu_{\Gamma_1}\circ\gamma_1=\id_A$. By Lemma~\ref{le:natur_iso2},  $\lambda\colon\Gamma_1(\MB,\bar p_1)\to \Gamma_2(\MB,\bar p_2)$ is an $L(\MA)$\=/isomorphism. Hence there exists an interpretation $\nsA\stackrel{\Gamma_2,\bar p_2}{\rightsquigarrow}\MB$. Let $\bar \mu_{\Gamma_2}=\bar\mu_{\Gamma_1}\circ\lambda^{-1}\colon \Gamma_2(\MB,\bar p_2)\to \nsA$ be $L(\MA)$\=/isomorphism and $\mu_{\Gamma_2}\colon U_{\Gamma_2}(\MB,\bar p_2)\to \nsuA$ be the corresponding coordinate map. Then interpretations $(\Gamma_1,\bar p_1,\mu_{\Gamma_1}),(\Gamma_2,\bar p_2,\mu_{\Gamma_2})\colon \nsA\rightsquigarrow \MB$ are strongly homotopic with connector $\theta$. Furthermore, by Lemma~\ref{le:natur_iso3},  $\lambda(\gamma_1(a)/{\sim_{\Gamma_1}})=\gamma_2(a)/{\sim_{\Gamma_2}}$ for every $a\in A$. Thus, $\mu_{\Gamma_2}(\gamma_2(a))=\bar\mu_{\Gamma_2}(\gamma_2(a)/{\sim_{\Gamma_2}})=\bar\mu_{\Gamma_1}(\lambda^{-1}(\gamma_2(a)/{\sim_{\Gamma_2}}))=\bar\mu_{\Gamma_1}(\gamma_1(a)/{\sim_{\Gamma_1}})=\mu_{\Gamma_1}(\gamma_1(a))=a$. Hence, one has $\mu_{\Gamma_2}\circ\gamma_2=\id_A$. It proves the theorem.
\end{proof}

\begin{corollary}\label{cor:natur_iso0}
If, in the notations and assumptions of Theorem~\ref{th:natur_iso1}, $\MA=\MB$, $(\Gamma_1,\bar p_1,\gamma_1)$ is $(\Gamma,\bar p,\gamma)$ and $(\Gamma_2,\bar p_2,\gamma_2)$ is $(\Id_{L(\MA)},\emptyset,\id_A)$, then $\lambda$ is an $L(\MA)$\=/isomorphism $\lambda\colon \Gamma(\MA,\bar p)\to \MA$, such that $\lambda(\gamma(a)/{\sim_\Gamma})=a$ for all $a\in A$.
\end{corollary}

\begin{proof}
Indeed, in this case $\Gamma_2(\MA,\bar p_2)=\MA$. Since isomorphism $\bar\mu_{\Gamma_2}\colon \MA\to \nsA$ is such that $\mu_{\Gamma_2}\circ\gamma_2=\id_A$, then $\bar\mu_{\Gamma_2}=\gamma_2=\id_A$, therefore, $\nsA=\MA$ and $\lambda=\bar\mu_{\Gamma_1}$.
\end{proof}

\begin{proof}[Proof of Theorem~\ref{th:natur_iso2}]
By Theorem~\ref{th:functor2}, there exist interpretations $(\Gamma_1,\bar p_1,\mu_{\Gamma_1})\colon \MA\rightsquigarrow \MB$ and $(\Gamma_2,\bar p_2,\mu_{\Gamma_2})\colon \MA\rightsquigarrow \MB$, such that $\mu_{\Gamma_1}\circ\gamma_1=\mu_{\Gamma_2}\circ\gamma_2=\id_A$. The restriction of natural isomorphism $\eta$ gives a natural isomorphisms between the translation $(\Gamma_1,\bar p_1,\gamma_1)$- and $(\Gamma_2,\bar p_2,\gamma_2)$\=/functors $\F_1,\F_2\colon \PDS(\MA)\to\PDS(\MB)$. By Lemma~\ref{le:natur_iso2}, $\lambda\colon\Gamma_1(\MB,\bar p_1)\to \Gamma_2(\MB,\bar p_2)$ is an $L(\MA)$\=/isomorphism. And by Lemma~\ref{le:natur_iso3}, $\lambda(\gamma_1(a)/{\sim_{\Gamma_1}})=\gamma_2(a)/{\sim_{\Gamma_2}}$ for every $a\in A$. Therefore, $\bar\mu_{\Gamma_1}(\gamma_1(a)/{\sim_{\Gamma_1}})=a$ and $\bar\mu_{\Gamma_2}(\lambda(\gamma_1(a)/{\sim_{\Gamma_1}}))=\bar\mu_{\Gamma_2}(\gamma_2(a)/{\sim_{\Gamma_2}})=a$. By Corollary~\ref{cor:surj}, every element from $\Gamma_1(\MB,\bar p_1)$ has a form $\gamma_1(a)/{\sim_{\Gamma_1}}$ for some $a\in A$. Hence, one has $\bar\mu_{\Gamma_1}=\bar\mu_{\Gamma_2}\circ\lambda$, i.\,e., interpretations $(\Gamma_1,\bar p_1,\mu_{\Gamma_1}),(\Gamma_2,\bar p_2,\mu_{\Gamma_2})\colon \MA\rightsquigarrow \MB$ are strongly homotopic with connector $\theta$, as required. 
\end{proof}

\begin{corollary}\label{cor:natur_iso_inter}
Let $(\Gamma_1,\bar p_1,\mu_{\Gamma_1}),(\Gamma_2,\bar p_2,\mu_{\Gamma_2})\colon \MA\rightsquigarrow\MB$ be interpretations, such that the interpretation $(\Gamma_1,\bar p_1,\mu_{\Gamma_1})$- and $(\Gamma_2,\bar p_2,\mu_{\Gamma_2})$\=/functors $\F_1,F_2\colon\PLS(\MA)\to\PLS(\MB)$ are naturally isomorphic; and $\eta\colon\F_1\to\F_2$ is a natural isomorphism. Then interpretations $(\Gamma_1,\bar p_1,\mu_{\Gamma_1})$ and $(\Gamma_2,\bar p_2,\mu_{\Gamma_2})$ are strongly homotopic with homotopy isomorphism $\lambda=\eta_A$. 
\end{corollary}

\begin{proof}
Let $\gamma_1\colon A\to U_{\Gamma_1}(\MB,\bar p_1)$ and $\gamma_2\colon A\to U_{\Gamma_2}(\MB,\bar p_2)$ be functions such that $\mu_{\Gamma_1}\circ\gamma_1=\mu_{\Gamma_2}\circ\gamma_2=\id_A$. Then $\eta$ is a natural isomorphism between the interpretation $(\Gamma_1,\bar p_1,\gamma_1)$- and $(\Gamma_2,\bar p_2,\gamma_2)$\=/functors. So by Theorem~\ref{th:natur_iso2} there exist strongly homotopic interpretations $(\Gamma_1,\bar p_1,\mu^\prime_{\Gamma_1}),(\Gamma_2,\bar p_2,\mu^\prime_{\Gamma_2})\colon \MA \rightsquigarrow \MB$, such that $\mu^\prime_{\Gamma_1} \circ\gamma_1=\mu^\prime_{\Gamma_2} \circ\gamma_2=\id_A$. And due to  Remark~\ref{lemma:gamma}, $\mu_{\Gamma_1}=\mu^\prime_{\Gamma_1}$ and $\mu_{\Gamma_2}=\mu^\prime_{\Gamma_2}$. Hence, the initial interpretations are strongly homotopic.  
\end{proof}

\begin{remark}\label{remark:DS}
When we prove Theorems~\ref{th:natur_iso1},~\ref{th:natur_iso2} and Lemmas~\ref{le:natur_iso0}, \ref{le:natur_iso1}, \ref{le:natur_iso2}, \ref{le:natur_iso3} we do not use the natural isomorphism $\eta$ in full force, but only its restriction to subcategory $\DS(\MA)$, i.\,e., we use only natural isomorphism between functors $\F_1,\F_2\colon \DS(\MA)\to\PDS(\MB)$.
\end{remark}

Although we didn't fully utilize the natural isomorphism $\eta$ in Theorem~\ref{th:natur_iso2}, more detailed information about it is of interest for further discussion. We formulate and prove two results that extend Lemma~\ref{le:natur_iso1}.

\begin{lemma}\label{le:natur_iso4}
In the assumptions and notations of Theorem~\ref{th:natur_iso2}, suppose that $X\subseteq A^m$ is a non-empty logical set over $\MA$. Then the categorical isomorphism $\eta_X\colon \F_1(X)\to \F_2(X)$ is fully induced by the formula $\theta_m$~\eqref{eq:F}.
\end{lemma}

\begin{proof}
According to Corollary~\ref{cor:functor_subset}, as in Lemma~\ref{le:natur_iso1}, we see that $\F_1(\varepsilon)\colon \F_1(X)\to \F_1(A^m)$ and $\F_2(\varepsilon)\colon \F_2(X)\to \F_2(A^m)$ are identical embeddings and 
 $\eta_X$ is the restriction of $\lambda^m$ on $\F_1(X)$ and $\eta_X(\F_1(X))=\F_2(X)$. Therefore, by Corollary~\ref{cor:descent}, $\eta_X$ is fully induced by the formula $\theta_m$~\eqref{eq:F}.  
\end{proof}

\begin{lemma}\label{le:natur_iso6}
Suppose that $\eta$ is a natural isomorphism between interpretations functors $\F_1,\F_2\colon \PLS(\MA)\to \PLS(\MB)$ and $X/{\sim_X}$ is a non-empty projective logical set over $\MA$. If the isomorphism $\eta_X\colon \F_1(X)\to \F_2(X)$ is fully induced by a formula $\psi(\bar{\bar x}_1,\bar{\bar x}_2)$ in $L(\MB)\cup B$, then the isomorphism $\eta_{X/{\sim_X}}\colon \F_1(X/{\sim_X})\to \F_2(X/{\sim_X})$ is fully induced by the formula $\varphi(\bar{\bar x}_1,\bar{\bar x}_2)=\exists \, \bar{\bar x}\:(\psi(\bar{\bar x}_1,\bar{\bar x})\wedge E_2(\bar{\bar x},\bar{\bar x}_2))$, where $E_2$ is a formula, which induces the formula equivalence of the projective logical set $\F_2(X/{\sim_X})$. The similar result is true for translation functors $\F_1,\F_2\colon \PDS(\MA)\to \PDS(\MB)$ and a projective definable set $X/{\sim_X}$. 
\end{lemma}

\begin{proof}
We will rely here on the notations from the proof of Theorem~\ref{th:homotopy}. Thus, $\F_1(X/{\sim_X})=\bar X_1/{\sim_{\bar X_1}}$ and $\F_2(X/{\sim_X})=\bar X_2/{\sim_{\bar X_2}}$. Consider the natural surjection $\pi\colon X\to X/{\sim_X}$, fully induced by the formula $E$, which induces the equivalence $\sim_X$ (see Fact~\ref{fact:surjection}). According to Corollary~\ref{cor:functor_subset} (or Corollary~\ref{cor:tr_functor_subset} for translation functors), $\F_1(\pi)\colon \bar X_1\to \bar X_1/{\sim_{\bar X_1}}$ and $\F_2(\pi)\colon \bar X_2\to \bar X_2/{\sim_{\bar X_2}}$ are the natural surjections. Thus, the formula map $\F_2(\pi)$ is fully induced by the formula $E_2$. Therefore, due to Lemma~\ref{lemma:composition}, the composition $F=\F_2(\pi)\circ\eta_X$ is fully induced by the formula $\varphi$. Since the  diagram
\begin{equation*}
\begin{tikzcd}[column sep=huge, row sep=large]
\bar X_1 
\arrow[r,leftarrow,shift right,"\eta_{X}^{-1}"']
\arrow[r,rightarrow,shift left,swap,"\eta_{X}"'] 
\arrow{d}{\F_1(\pi)} 
& \bar X_2
\arrow{d}{\F_2(\pi)}\\
\bar X_1/{\sim_{\bar X_1}}  
\arrow[r,leftarrow,shift right,"\eta_{X/{\sim_X}}^{-1}"']
\arrow[r,rightarrow,shift left,swap,"\eta_{X/{\sim_X}}"']  
& \bar X_2/{\sim_{\bar X_2}}
\end{tikzcd}    
\end{equation*}
is commutative, then  $F=\eta_{X/{\sim_X}}\circ\F_1(\pi)$. Hence, by Lemma~\ref{lemma:F}, $\eta_{X/{\sim_X}}$ is fully induced by the formula $\varphi$ as well.   
\end{proof}

\begin{corollary}\label{le:natur_iso5}
In the assumptions and notations of Theorem~\ref{th:natur_iso2}, suppose that $X\subseteq A^m$ is a non-empty logical set over $\MA$ and $\sim_X$ is a formula equivalence on $X$. Then the categorical isomorphism $\eta_{X/{\sim_X}}\colon \F_1(X/{\sim_X})\to \F_2(X/{\sim_X})$ is fully induced by the formula $\varphi_{\bar X_1 \bar X_2}$~\eqref{eq:F}.
\end{corollary}

\begin{proof}
Indeed, by Lemma~\ref{le:natur_iso4}, isomorphism $\eta_X\colon \F_1(X)\to \F_2(X)$ is fully induced by the formula $\theta_m$~\eqref{eq:F}. Therefore, by Lemma~\ref{le:natur_iso6}, isomorphism $\eta_{X/{\sim_X}}\colon \F_1(X/{\sim_X})\to \F_2(X/{\sim_X})$ is fully induced by the formula $\varphi_{\bar X_1 \bar X_2}$~\eqref{eq:F}.
\end{proof}

\begin{lemma}\label{cor:natur_iso}
Let $\eta,\tau\colon \F_1\to \F_2$ be natural isomorphisms between a pair of interpretation functors $\F_1,\F_2\colon \PLS(\MA)\to\PLS(\MB)$. Then one has $\eta=\tau$ if and only if $\eta_A=\tau_A$. And the similar is true for translation functors $\F_1,\F_2\colon \PDS(\MA)\to\PDS(\MB)$, for injective interpretation functors $\F_1,\F_2\colon \LS(\MA)\to\LS(\MB)$ and for injective translation functors $\F_1,\F_2\colon \DS(\MA)\to\DS(\MB)$.
\end{lemma}

\begin{proof}
Let $\theta$ be the formula in $L(\MB)\cup B$, which defines the isomorphism $\eta_A=\tau_A\colon \F_1(A)\to \F_2(A)$. As in Lemmas~\ref{le:natur_iso1} and~\ref{le:natur_iso4}, for any non-empty logical or definable set $X\subseteq A^m$ over $\MA$ both isomorphisms $\eta_X$ and $\tau_X$ are fully induced by the formula $\theta_m$~\eqref{eq:F}. By Lemma~\ref{le:natur_iso6}, for any formula equivalence $\sim_X$ on $X$ both isomorphisms $\eta_{X/{\sim_X}}$ and $\tau_{X/{\sim_X}}$ are fully induced by the formula $\varphi_{\bar X_1 \bar X_2}$~\eqref{eq:F}. Thus, one has $\eta=\tau$.
\end{proof}

\begin{notation}
Let us denote by ${\bf N}$ the operator from Theorem~\ref{th:homotopy}, which constructs a natural isomorphism $\tilde\lambda$ for every strong homotopy $(\theta)$ of interpretations, and by ${\bf H}$ the operator from Theorem~\ref{th:natur_iso2}, which constructs a strong homotopy $(\theta)$ of interpretations for every natural isomorphism $\eta$ of interpretation functors. 
\end{notation}

\begin{lemma}\label{le:natur_iso}
For any strong homotopy $(\theta)\colon (\Gamma_1,\bar p_1,\mu_{\Gamma_1})\to(\Gamma_2,\bar p_2,\mu_{\Gamma_2})$ of interpretations of $\MA$ in $\MB$ one has ${\bf H}({\bf N}((\theta)))=(\theta)$. And for any natural isomorphism $\eta\colon \F_1\to\F_2$
 between interpretation $(\Gamma_1,\bar p_1,\mu_{\Gamma_1})$- and $(\Gamma_2,\bar p_2,\mu_{\Gamma_2})$\=/functors $\F_1,\F_2\colon \PLS(\MA)\to \PLS(\MB)$ one has ${\bf N}({\bf H}(\eta))=\eta$.
 \end{lemma}

 \begin{proof}
Let  $(\theta)\colon (\Gamma_1,\bar p_1,\mu_{\Gamma_1})\to(\Gamma_2,\bar p_2,\mu_{\Gamma_2})$ be a strong homotopy of interpretations, and $\tilde\lambda={\bf N}((\theta))$ the corresponding natural isomorphism. By Corollary~\ref{cor:homotopy}, the isomorphism $\tilde\lambda_A$ is defined by the formula $\theta$. Hence, by Corollary~\ref{cor:natur_iso_inter}, ${\bf H}(\tilde\lambda)=(\theta)$, as required.
 
Take a natural isomorphism $\eta\colon \F_1\to\F_2$ between interpretation $(\Gamma_1,\bar p_1,\mu_{\Gamma_1})$- and $(\Gamma_2,\bar p_2,\mu_{\Gamma_2})$\=/functors $\F_1,\F_2\colon \PLS(\MA)\to \PLS(\MB)$. Suppose that $(\theta)={\bf H}(\eta)$ is the corresponding strong homotopy $(\theta)\colon (\Gamma_1,\bar p_1,\mu_{\Gamma_1})\to(\Gamma_2,\bar p_2,\mu_{\Gamma_2})$. By Corollary~\ref{cor:natur_iso_inter}, $\theta$ is the formula which defines the isomorphism $\eta_A$. Therefore, $\eta_A={\bf N}((\theta))_A$. Thus, according to Lemma~\ref{cor:natur_iso}, one has 
${\bf N}((\theta))=\eta$, as required.  
\end{proof}

{\bf Equal interpretations and equal homotopies.} Among all the strongly homotopic interpretations, we will emphasize those that are equal. Let $(\Gamma_1,\bar p_1,\mu_{\Gamma_1}),(\Gamma_2,\bar p_2,\mu_{\Gamma_2}),  (\Gamma^\prime_1,\bar p^\prime_1,\mu_{\Gamma^\prime_1}), (\Gamma^\prime_2,\bar p^\prime_2,\mu_{\Gamma^\prime_2})\colon \MA\rightsquigarrow\MB$ be interpretations.

\begin{definition}\label{def:equal}
We say that two interpretations $(\Gamma_1,\bar p_1,\mu_{\Gamma_1}),(\Gamma_2,\bar p_2,\mu_{\Gamma_2})$ are {\em equal} and write $(\Gamma_1,\bar p_1,\mu_{\Gamma_1})=_{\rm int}(\Gamma_2,\bar p_2,\mu_{\Gamma_2})$, if they are strongly homotopic, such that the corresponding homotopy isomorphism is the identical map, i.\,e., if $\Gamma_1(\MB,\bar p_1)=\Gamma_2(\MB,\bar p_2)$ and $\mu_{\Gamma_1}=\mu_{\Gamma_2}$. 
\end{definition}

\begin{remark}
It is clear that equality $=_{\rm int}$ is an equivalence relation on the set ${\rm Int}^+(\MA,\MB)$ of all coordinatizations of all interpretations of $\MA$ in $\MB$~\cite[Definition~15]{Th_int1}.
\end{remark}

\begin{example}\label{ex:approx}
Suppose that $(\Gamma,\bar p,\mu_{\Gamma})\colon \MA\rightsquigarrow\MB$ is an interpretation and a code $\Gamma^\prime\colon L(\MA)\to L(\MB)$ is equivalent to $\Gamma$, $\Gamma^\prime\approx\Gamma$ (see~\cite[Definition~3]{Th_int1}). Then there exists an interpretation $(\Gamma^\prime,\bar p,\mu_{\Gamma})\colon \MA\rightsquigarrow\MB$ and it equals to $(\Gamma,\bar p,\mu_{\Gamma})\colon \MA\rightsquigarrow\MB$.
\end{example}

\begin{lemma}\label{le:equal}
Interpretations $(\Gamma_1,\bar p_1,\mu_{\Gamma_1}),(\Gamma_2,\bar p_2,\mu_{\Gamma_2})\colon \MA\rightsquigarrow\MB$ are equal if and only if the interpretation $(\Gamma_1,\bar p_1,\mu_{\Gamma_1})$- and $(\Gamma_2,\bar p_2,\mu_{\Gamma_2})$-functors $\F_1,\F_2\colon \PLS(\MA)\to \PLS(\MB)$ are equal.
\end{lemma}

\begin{proof}
If interpretations $(\Gamma_1,\bar p_1,\mu_{\Gamma_1})$ and $(\Gamma_2,\bar p_2,\mu_{\Gamma_2})$ are equal, then the formula $\theta(\bar x_1,\bar x_2)=U_{\Gamma_1}(\bar x_1,\bar p_1)\wedge U_{\Gamma_2}(\bar x_2,\bar p_2)\wedge E_{\Gamma_2}(\bar x_1,\bar x_2)$ is their connector. Thus for any projective logical set $X/{\sim_X}$ over $\MA$ formulas~\eqref{eq:F} induce the identical map on $\bar X_2/{\sim_{\bar X_2}}$, therefore, the natural isomorphism ${\bf N}((\theta))$ is identical, i.\,e., $\F_1=\F_2$. Inversely, if $\F_1=\F_2$, then the identical natural isomorphism $\eta\colon \F_1\to\F_2$ gives the strong homotopy $(\theta)={\bf H}(\eta)$, there $\theta$ defines the identical map $\id_{\Gamma_1(\MB,\bar p_1)}$. 
\end{proof}

For further discussions, it is also important to establish the concept of equal homotopies. 

\begin{definition}
We say that two strong homotopies $(\theta)\colon (\Gamma_1,\bar p_1,\mu_{\Gamma_1})\to (\Gamma_2,\bar p_2,\mu_{\Gamma_2})$ and $(\theta^\prime)\colon (\Gamma^\prime_1,\bar p^\prime_1,\mu_{\Gamma^\prime_1})\to (\Gamma^\prime_2,\bar p^\prime_2,\mu_{\Gamma^\prime_2})$ are equal and write $(\theta)=_{\rm hom}(\theta^\prime)$, if $(\Gamma_1,\bar p_1,\mu_{\Gamma_1})=_{\rm int}(\Gamma^\prime_1,\bar p^\prime_1,\mu_{\Gamma^\prime_1})$ and $(\Gamma_2,\bar p_2,\mu_{\Gamma_2})=_{\rm int}(\Gamma^\prime_2,\bar p^\prime_2,\mu_{\Gamma^\prime_2})$, and $\theta$ and $\theta^\prime$ define one and same $L(\MA)$\=/isomorphism $\lambda\colon \Gamma_1(\MB,\bar p_1)\to\Gamma_2(\MB,\bar p_2)$.
\end{definition}

\begin{remark}
The equality $=_{\rm hom}$ is an equivalence relation on the set of all strong homotopies of interpretation of $\MA$ in $\MB$.
\end{remark}

{\bf ``Code-forgetting functors''.} Note that both input and output of operators ${\bf N}$, ${\bf H}$ and ${\bf F}$, ${\bf I}$ from Remark~\ref{remark:gamma_int}   contain information about the extended code $(\Gamma,\bar p,\gamma)$ or interpretation $(\Gamma,\bar p,\mu_\Gamma)$. Along with these operators, it's important to have their counterparts that ``erase information about interpretation codes on output''.

\begin{notation}\label{hat}
We denote such ``forgetful codes'' operators by ${\bf \hat N}$, ${\bf \hat H}$ and ${\bf \hat F}$, ${\bf \hat I}$.
\end{notation}

Thus, we reformulate Lemma~\ref{le:equal} and Remark~\ref{remark:gamma_int} as follows.

\begin{fact}\label{fact:hat1}
 Interpretations $(\Gamma_1,\bar p_1,\mu_{\Gamma_1}),(\Gamma_2,\bar p_2,\mu_{\Gamma_2})\colon \MA\rightsquigarrow\MB$ are equal if and only if ${\bf \hat F}((\Gamma_1,\bar p_1,\mu_{\Gamma_1}))={\bf \hat F}((\Gamma_2,\bar p_1,\mu_{\Gamma_2}))$; and interpretation functors $\F_1,\F_2\colon \PLS(\MA)\to\PLS(\MB)$ are equal if and only if ${\bf \hat I}(\F_1)=_{\rm int}{\bf \hat I}(\F_2)$. \end{fact}

 \begin{fact}\label{fact:hat2}
For any interpretation $(\Gamma,\bar p,\mu_\Gamma)\colon \MA\rightsquigarrow\MB$ one has ${\bf \hat I}({\bf \hat F}((\Gamma,\bar p,\mu_\Gamma)))=_{\rm int}(\Gamma,\bar p,\mu_\Gamma)$. And for any interpretation functor $\F\colon \PLS(\MA)\to \PLS(\MB)$ one has ${\bf \hat F}({\bf \hat I}(\F))=\F$.
\end{fact}

The following results are reformulations of Lemmas~\ref{cor:natur_iso} and~\ref{le:natur_iso}.

\begin{fact}\label{fact:hat3}
Strong homotopies $(\theta)\colon (\Gamma_1,\bar p_1,\mu_{\Gamma_1})\to (\Gamma_2,\bar p_2,\mu_{\Gamma_2})$ and $(\theta^\prime)\colon (\Gamma^\prime_1,\bar p^\prime_1,\mu_{\Gamma^\prime_1})\to (\Gamma^\prime_2,\bar p^\prime_2,\mu_{\Gamma^\prime_2})$ are equal if and only if ${\bf \hat N}((\theta))={\bf \hat N}((\theta^\prime))$. Natural isomorphisms $\eta,\eta^\prime\colon \F_1\to\F_2$ between interpretation functors $\F_1,\F_2\colon \PLS(\MA)\to \PLS(\MB)$ are equal if and only if ${\bf \hat H}(\eta)=_{\rm hom}{\bf \hat H}(\eta^\prime)$.
\end{fact}

\begin{fact}\label{fact:hat4}
For any strong homotopy $(\theta)\colon (\Gamma_1,\bar p_1,\mu_{\Gamma_1})\to (\Gamma_2,\bar p_2,\mu_{\Gamma_2})$ one has ${\bf \hat H}({\bf \hat N}((\theta)))=_{\rm hom}(\theta)$. And for any natural isomorphism $\eta\colon \F_1\to\F_2$ between interpretation functors $\F_1,\F_2\colon \PLS(\MA)\to \PLS(\MB)$ one has ${\bf \hat N}({\bf \hat H}(\eta))=\eta$.
\end{fact}

\subsection{Restrictions and extensions of functors and natural isomorphisms}\label{subsec:re}

In this subsection, we gather together simple facts and remarks about restrictions and extensions of interpretation and translation functors and natural isomorphisms between them.

Let $\MA=\langle A; L(\MA)\rangle$,  $\MB=\langle B; L(\MB)\rangle$, $\MC=\langle C; L(\MC)\rangle$  be algebraic structures and $(\Gamma,\bar p,\gamma), (\Gamma_1,\bar p_1,\gamma_1), (\Gamma_2,\bar p_2,\gamma_2)\colon L(\MA)\cup A\to L(\MB)\cup B$, $(\Delta,\bar q,\delta)\colon L(\MB)\cup B\to L(\MC)\cup C$ extended codes, and $\alpha\leqslant \beta$, $\kappa\leqslant\lambda$ infinite ordinals, such that $\alpha\leqslant\kappa$, $\beta\leqslant \lambda$.

\begin{fact}[on restrictions and extensions of functors]\label{fact:functor} 
The following holds:
\begin{enumerate}
    \item If the translation $(\Gamma,\bar p,\gamma)$\=/functor $\F_{\kappa,\lambda}\colon \PLS_{\kappa,\lambda}(\MA)\to \PLS_{\kappa,\lambda}(\MB)$ is well-defined, then the translation $(\Gamma,\bar p,\gamma)$\=/functor $\F_{\alpha,\beta}\colon \PLS_{\alpha,\beta}(\MA)\to \PLS_{\alpha,\beta}(\MB)$ is well-defined too. 
    \item If the interpretation code $\Gamma$ is injective and the injective translation $(\Gamma,\bar p,\gamma)$\=/functor $\F_{\kappa,\lambda}\colon \LS_{\kappa,\lambda}(\MA)\to \LS_{\kappa,\lambda}(\MB)$ is well-defined, then the injective translation $(\Gamma,\bar p,\gamma)$\=/functor $\F_{\alpha,\beta}\colon \LS_{\alpha,\beta}(\MA)\to \LS_{\alpha,\beta}(\MB)$ is well-defined too. 
    \item If the interpretation code $\Gamma$ is absolute and the translation $(\Gamma,\emptyset,\gamma)$\=/functor $\F_{\kappa,\lambda}\colon \PLS_{\kappa,\lambda}(\MA)\to \PLS_{\kappa,\lambda}(\MB)$ is well-defined, then the absolute translation $\Gamma$\=/functor $\F_{0,\lambda}\colon \PLS_{0,\lambda}(\MA)\to \PLS_{0,\lambda}(\MB)$ is well-defined too; and if $\F_{0,\lambda}\colon \PLS_0(\MA)\to \PLS_0(\MB)$ is well-defined, then the absolute translation $\Gamma$\=/functor $\F_{0,\beta}\colon \PLS_{0,\beta}(\MA)\to \PLS_{0,\beta}(\MB)$ is well-defined too.
    \item If the interpretation code $\Gamma$ is absolute and injective and the translation $(\Gamma,\emptyset,\gamma)$\=/functor $\F_{\kappa,\lambda}\colon \PLS_{\kappa,\lambda}(\MA)\to \PLS_{\kappa,\lambda}(\MB)$ is well-defined, then the absolute injective translation $\Gamma$\=/functor $\F_{0,\lambda}\colon \LS_{0,\lambda}(\MA)\to \LS_{0,\lambda}(\MB)$ is well-defined too; and if $\F_{0,\lambda}\colon \LS_0(\MA)\to \LS_0(\MB)$ is well-defined, then the absolute injective translation $\Gamma$\=/functor $\F_{0,\beta}\colon \LS_{0,\beta}(\MA)\to \LS_{0,\beta}(\MB)$ is well-defined too.
    \item If the interpretation code $\Gamma$ is injective, then the interpretation $(\Gamma,\bar p,\gamma)$\=/functor $\F\colon \PLS(\MA)\to \PLS(\MB)$ is well-defined if and only if the injective interpretation $(\Gamma,\bar p,\gamma)$\=/functor $\F\colon \LS(\MA)\to \LS(\MB)$ is well-defined. 
    \item If the interpretation code $\Gamma$ is injective, then the translation $(\Gamma,\bar p,\gamma)$\=/functor $\F\colon \PDS(\MA)\to \PDS(\MB)$ is well-defined if and only if the injective translation $(\Gamma,\bar p,\gamma)$\=/functor $\F\colon \DS(\MA)\to \DS(\MB)$ is well-defined.
\end{enumerate}
\end{fact}

\begin{proof}
If the injective translation $(\Gamma,\bar p,\gamma)$\=/functor $\F\colon \DS(\MA)\to \DS(\MB)$ is well-defined then the condition~\ref{F1}$_{\omega,\omega}$ holds, therefore the translation $(\Gamma,\bar p,\gamma)$\=/functor $\F\colon \PDS(\MA)\to \PDS(\MB)$ is well-defined too due to Corollary~\ref{cor:inj1}. And if the injective interpretation $(\Gamma,\bar p,\gamma)$\=/functor $\F\colon \LS(\MA)\to \LS(\MB)$ is well-defined, then the condition~\ref{F1} holds, and therefore the interpretation $(\Gamma,\bar p,\gamma)$\=/functor $\F\colon \PLS(\MA)\to \PLS(\MB)$ is well-defined too due to Corollary~\ref{cor:inj2}. All other items are trivial since the required functors are restrictions of given functors.
\end{proof}

\begin{remark}[on restrictions of compositions of functors]\label{cor:comp_rest}
Construction of interpretation and translation functors such that their synchronous restrictions to subcategories $\PLS_{\alpha,\beta}(\MA)$, $\LS_{\alpha,\beta}(\MA)$, $\PLS_{0,\beta}(\MA)$, $\LS_{0,\beta}(\MA)$ commute with composition. For example, let $\F\colon \PLS(\MA)\to\PLS(\MB)$ and $\G\colon \PLS(\MB)\to \PLS(\MC)$ be the interpretation $(\Gamma,\bar p,\gamma)$- and $(\Delta,\bar q,\delta)$\=/functors. Then $(\G\circ\F)_{\alpha,\beta}=\G_{\alpha,\beta}\circ\F_{\alpha,\beta}$. Or, if 
codes $\Gamma,\Delta$ are injective, then the restriction of the composition $\G\circ\F$ to subcategory $\LS_{\alpha,\beta}(\MA)$ coincides with the composition of the restricrtion of $\F$ to $\LS_{\alpha,\beta}(\MA)$ and the restriction of $\G$ to $\LS_{\alpha,\beta}(\MB)$.
\end{remark}

\begin{fact}[on composition of injective translation functors] 
Suppose that codes $\Gamma$ and $\Delta$ are injective. Then the composition of injective interpretation $(\Gamma,\bar p,\gamma)$- and $(\Delta,\bar q,\delta)$\=/functors $\F\colon \LS(\MA)\to \LS(\MB)$ and $\G\colon \LS(\MB)\to \LS(\MC)$ is the injective interpretation $(\Gamma\circ\Delta,(\delta(\bar p),\bar q),\delta\circ\gamma)$\=/functor $\G\circ\F\colon \LS(\MA)\to \LS(\MC)$. And the composition of injective translation $(\Gamma,\bar p,\gamma)$- and $(\Delta,\bar q,\delta)$\=/functors $\F\colon \DS(\MA)\to \DS(\MB)$ and $\G\colon \DS(\MB)\to \DS(\MC)$ is the injective translation $(\Gamma\circ\Delta,(\delta(\bar p),\bar q),\delta\circ\gamma)$\=/functor $\G\circ\F\colon \DS(\MA)\to \DS(\MC)$. 
\end{fact}

\begin{proof}
Indeed, if there exist injective interpretation $(\Gamma,\bar p,\gamma)$- and $(\Delta,\bar q,\delta)$\=/functors $\F\colon \LS(\MA)\to \LS(\MB)$ and $\G\colon \LS(\MB)\to \LS(\MC)$, then there exist interpretation $(\Gamma,\bar p,\gamma)$- and $(\Delta,\bar q,\delta)$\=/functors $\F\colon \PLS(\MA)\to \PLS(\MB)$ and $\G\colon \PLS(\MB)\to \PLS(\MC)$. The restriction of the composition $\G\circ\F$ to $\LS(\MA)$ gives the required due to Proposition~\ref{prop:comp1}. And similar arguments with reference to Proposition~\ref{prop:comp2} show that the fact about translation functors is true.
\end{proof}

\begin{remark}[on restrictions of natural isomorphisms]\label{remark:natur_iso}
If $\eta\colon \F_1\to\F_2$ is a natural isomorphism between translation $(\Gamma_1,\bar p_1,\gamma_1)$- and $(\Gamma_2,\bar p_2,\gamma_2)$\=/functors $\F_1,\F_2\colon\PLS_{\kappa,\lambda}(\MA)\to\PLS_{\kappa,\lambda}(\MB)$, then its restriction on subcategory $\PLS_{\alpha,\beta}(\MA)$ is a natural isomorphism between the translation $(\Gamma_1,\bar p_1,\gamma_1)$- and $(\Gamma_2,\bar p_2,\gamma_2)$\=/functors $\F_1,\F_2\colon\PLS_{\alpha,\beta}(\MA)\to\PLS_{\alpha,\beta}(\MB)$. Furthermore, if the interpretation codes $\Gamma_1$ and $\Gamma_2$ are injective, then the restriction of $\eta$ on $\LS_{\kappa,\lambda}(\MA)$ is a natural isomorphisms between the injective translation $(\Gamma_1,\bar p_1,\gamma_1)$- and $(\Gamma_2,\bar p_2,\gamma_2)$\=/functors $\F^\prime_1,\F^\prime_2\colon\LS_{\kappa,\lambda}(\MA)\to\LS_{\kappa,\lambda}(\MB)$; and the restriction of a natural isomorphism $\eta\colon \F^\prime_1\to \F^\prime_2$ on subcategory $\LS_{\alpha,\beta}(\MA)$ is a natural isomorphisms between the injective translation $(\Gamma_1,\bar p_1,\gamma_1)$- and $(\Gamma_2,\bar p_2,\gamma_2)$\=/functors $\F^\prime_1,\F^\prime_2\colon\LS_{\alpha,\beta}(\MA)\to\LS_{\alpha,\beta}(\MB)$.
\end{remark}

It is convenient to put a similar result for the absolute case into a separate statement.

\begin{fact}\label{cor:abs_hom}
Suppose that codes $\Gamma_1$ and $\Gamma_2$ are absolute. Then a strong homotopy $(\theta)\colon (\Gamma_1,\emptyset,\mu_{\Gamma_1})\to(\Gamma_2,\emptyset,\mu_{\Gamma_2})$ is absolute if and only if the restriction of ${\bf N}((\theta))$ on $\PLS_0(\MA)$ is a natural isomorphism between absolute interpretation functors $\F_1,\F_2\colon \PLS_0(\MA)\to \PLS_0(\MB)$. In this case, there exists also a natural isomorphism between absolute interpretation functors $\F_1,\F_2\colon \PLS_{0,\beta}(\MA)\to \PLS_{0,\beta}(\MB)$. 
\end{fact}

\begin{proof}
Indeed, the homotopy $(\theta)$ is absolute if and only if the formula $\theta$ has no parameters from $B$. So, if $\theta$ has no parameters from $B$, then for any object $X/{\sim_X}$ from $\PLS_0(\MA)$ formulas from~\eqref{eq:F} are in the language $L(\MB)$, i.\,e., they has no parameters from $B$. Therefore, the corresponding formula maps are isomorphisms in $\PLS_0(\MB)$ (see Remark~\ref{remark:P}). And, inversely, if the restriction of ${\bf N}((\theta))$ on $\PLS_0(\MA)$ is a natural isomorphism between absolute interpretation functors, then $\theta$, as a formula, which defines ${\bf N}((\theta))_A$, has no   has no parameters.
\end{proof}

\begin{corollary}\label{cor:inj_abs}
Suppose that both codes $\Gamma_1$ and $\Gamma_2$ are absolute and injective. If $\eta\colon \F_1\to\F_2$ is a natural isomorphism between translation $(\Gamma_1,\emptyset,\gamma_1)$- and $(\Gamma_2,\emptyset,\gamma_2)$\=/functors $\F_1,\F_2\colon\PLS_{\kappa,\lambda}(\MA)\to\PLS_{\kappa,\lambda}(\MB)$, then its restriction on subcategory $\LS_{0,\lambda}(\MA)$ is a natural isomorphism between the absolute injective translation $\Gamma_1$- and $\Gamma_2$\=/functors $\F^\prime_1,\F^\prime_2\colon\LS_{0,\lambda}(\MA)\to\LS_{0,\lambda}(\MB)$;  and the restriction of a natural isomorphism $\eta\colon \F^\prime_1\to \F^\prime_2$ on subcategory $\LS_{0,\beta}(\MA)$ is a natural isomorphisms between the absolute injective translation $\Gamma_1$- and $\Gamma_2$\=/functors $\F^\prime_1,\F^\prime_2\colon\LS_{0,\beta}(\MA)\to\LS_{0,\beta}(\MB)$. 
\end{corollary}

The following facts are about extensions of natural isomorphisms.

\begin{lemma}\label{cor:natur_iso1}
Suppose that both codes $\Gamma_1$ and $\Gamma_2$ are injective and $\eta\colon \F_1\to\F_2$ is a natural isomorphism between the injective translation $(\Gamma_1,\bar p_1,\gamma_1)$- and $(\Gamma_1,\bar p_1,\gamma_1)$\=/functors $\F_1,\F_2\colon\DS(\MA)\to\DS(\MB)$. Then there exists and uniquely defined an extension $\tau\colon \F_1^\prime\to \F_2^\prime$ of natural isomorphism $\eta$ to a natural isomorphism between the translation $(\Gamma_1,\bar p_1,\gamma_1)$- and $(\Gamma_1,\bar p_1,\gamma_1)$\=/functors $\F_1^\prime,\F_2^\prime\colon\PDS(\MA)\to\PDS(\MB)$. In particular, if $\eta$ is the identical natural isomorphism, then $\tau$ is the identical natural isomorphism.
\end{lemma}

\begin{proof}
Due to Remark~\ref{remark:DS} we may apply Theorem~\ref{th:natur_iso1}. So we get that there exist an elementary extension $\MA\preceq \nsA$, $\nsA=\langle \nsuA; L(\MA)\rangle$, and strongly homotopic interpretations $(\Gamma_1,\bar p_1,\mu_{\Gamma_1})\colon \nsA \rightsquigarrow \MB$ and $(\Gamma_2,\bar p_2,\mu_{\Gamma_2})\colon \nsA \rightsquigarrow \MB$, such that $\mu_{\Gamma_1} \circ\gamma_1=\mu_{\Gamma_2} \circ\gamma_2=\id_A$, with homotopy isomorphism $\lambda=\eta_A$, defined by formula $\theta$. Let $\F_0\colon \PLS(\MA)\to \PLS(\nsA)$ be the translation $(\Id_{L(\MA)},\emptyset,\id_A)$\=/functor from Lemma~\ref{lemma:iota} and $\tilde\lambda$ be natural isomorphism between the interpretation $(\Gamma_1,\bar p_1,\mu_{\Gamma_1})$- and $(\Gamma_2,\bar p_2,\mu_{\Gamma_2})$\=/functors $\G_1,\G_2\colon \PLS(\nsA)\to \PLS(\MB)$ from Theorem~\ref{th:homotopy}. Take  functions $\gamma^\prime_1\colon \nsuA\to U_{\Gamma_1}(\MB,\bar p_1)$ and $\gamma^\prime_1\colon \nsuA\to U_{\Gamma_1}(\MB,\bar p_1)$, such that $\mu_{\Gamma_1}\circ\gamma_1^\prime=\mu_{\Gamma_2}\circ\gamma_2^\prime=\id_{\nsuA}$ and $\gamma_1(a)=\gamma_1^\prime(a)$, $\gamma_2(a)=\gamma_2^\prime(a)$ for all $a\in A$. Let us consider the restrictions of functors $\G_1$ and $\G_2$ on $\PDS(\nsA)$ as translation $(\Gamma_1,\bar p_1,\gamma^\prime_1)$- and $(\Gamma_2,\bar p_2,\gamma^\prime_2)$\=/functors. Then, by Proposition~\ref{prop:comp2} and~\cite[Lemma~9]{Th_int1}, one has $\F^\prime_1=\G_1\circ\F_0$, $\F^\prime_2=\G_2\circ\F_0$. Furthermore, $\tau=\tilde\lambda\circ\F_0$ is a natural isomorphism between $\F^\prime_1$ and $\F^\prime_2$. Due to Corollary~\ref{cor:homotopy}, the isomorphism $\tilde\lambda_{\nsuA}$ is defined by the formula $\theta$. Since $\F_0(A)=\nsuA$, then the isomorphism $\tau_A$ is defined by $\theta$, i.\,e., $\tau_A=\eta_A$. Therefore, by Lemma~\ref{cor:natur_iso}, the restriction of $\tau$ on $\DS(\MA)$ coincides with $\eta$, and $\tau$ is unique. 
\end{proof}

\begin{lemma}\label{cor:natur_iso2}
Suppose that both codes $\Gamma_1$ and $\Gamma_2$ are injective and $\eta\colon \F_1\to\F_2$ is a natural isomorphism between the injective interpretation $(\Gamma_1,\bar p_1,\gamma_1)$- and $(\Gamma_1,\bar p_1,\gamma_1)$\=/functors $\F_1,\F_2\colon\LS(\MA)\to\LS(\MB)$. Then there exists and uniquely defined an extension $\tau\colon \F_1^\prime\to \F_2^\prime$ of natural isomorphism $\eta$ to a natural isomorphism between the interpretation $(\Gamma_1,\bar p_1,\gamma_1)$- and $(\Gamma_1,\bar p_1,\gamma_1)$\=/functors $\F_1^\prime,\F_2^\prime\colon\PLS(\MA)\to\PLS(\MB)$. In particular, if $\eta$ is the identical natural isomorphism, then $\tau$ is the identical natural isomorphism.
\end{lemma}

\begin{proof}
Again, according to Remark~\ref{remark:DS} we apply Theorem~\ref{th:natur_iso2} and get that there exist strongly homotopic interpretations $(\Gamma_1,\bar p_1,\mu_{\Gamma_1}),(\Gamma_2,\bar p_2,\mu_{\Gamma_2})\colon \MA \rightsquigarrow \MB$ with homotopy isomorphism $\lambda=\eta_A$, defined by a formula $\theta$, and $\mu_{\Gamma_1} \circ\gamma_1=\mu_{\Gamma_2} \circ\gamma_2=\id_A$. Let $\tau={\bf N}((\theta))$ be the natural isomorphism between functors $\F_1^\prime$ and $\F_2^\prime$. Due to Corollary~\ref{cor:homotopy} the isomorphism $\tau_A$ equals to $\eta_A$, therefore, by Lemma~\ref{cor:natur_iso}, $\eta$ is the restriction of $\tau$ to $\LS(\MA)$, and $\tau$ is unique.
\end{proof}

\subsection{Bi-interpretations and translation equivalencies of categories}\label{subsec:bi}

In this subsection, we are going to prove the categorical equivalence of logical geometries of strongly bi-interpretable algebraic structures. 

Remind that an equivalence $\mathcal{A}\sim \mathcal{B}$ between categories $\mathcal{A}$ and $\mathcal{B}$ is called {\em relative to class of functors $\mathbf{K}$} and denoted by $\mathcal{A}\sim_\mathbf{K} \mathcal{B}$, if there exist functors $\F\colon\mathcal{A}\to\mathcal{B}$ and $\G\colon\mathcal{B}\to\mathcal{A}$ from $\mathbf{K}$, such that $\G\circ\F\cong\id_\mathcal{A}$ and $\F\circ\G\cong\id_{\mathcal{B}}$. The same holds for isomorphisms of categories, thus, $\mathcal{A}\simeq_\mathbf{K}\mathcal{B}$ is an isomorphism relative to class $\mathbf{K}$.

\begin{notation}
We denote by $\Int$ ($\Int_\Inj$) the class of all (injective) interpretation functors and by $\Tr$ ($\Tr_\Inj$) the class of all (injective) translation functors.
\end{notation}

According to Remark~\ref{remark:ID}, for any algebraic structure $\MA$ the identity functor $\id_{\PLS(\MA)}$ is in $\Int_\Inj$, as well as $\id_{\LS(\MA)}\in\Int_\Inj$ and $\id_{\PDS(\MA)},\id_{\DS(\MA)}\in\Tr_\Inj$. If $\F\colon \PLS(\MA)\to \PLS(\MB)$ and $\G\colon\PLS(\MB)\to \PLS(\MC)$ are interpretation functors, then $\G\circ\F$ is an interpretation functor; and the same is true for translation functors $\F\colon \PDS(\MA)\to \PDS(\MB)$ and $\G\colon\PDS(\MB)\to \PDS(\MC)$ due to Propositions~\ref{prop:comp1} and~\ref{prop:comp2}.

\begin{definition}
We will refer to an equivalence (an isomorphism) of categories relative to the class $\Int$ ($\Int_\Inj$, $\Tr$, $\Tr_\Inj$) as {\em interpretation (injectively interpretation, translation, injectively translation) equivalence (isomorphism)}.
\end{definition}

Our main goal in this subsection and in the whole paper is to prove the following theorem.

\begin{theorem}[categorical criteria of strong bi-interpretation]\label{th:bi-inter}
Let $\MA=\langle A; L(\MA)\rangle$ and $\MB=\langle B; L(\MB)\rangle$ be algebraic structures. Then the following conditions are equivalent:
\begin{enumerate}[label=(\arabic*)]
    \item\label{item1:invert} $\MA$ and $\MB$ are strongly bi-interpretable in each other;
    \item\label{item2:invert} categories of projective definable sets over $\MA$ and $\MB$ are equivalent relative to the class $\Tr$ of translation functors:
    $$\PDS(\MA)\sim_\Tr\PDS(\MB);$$
        \item\label{item3:invert} categories of projective logical sets over $\MA$ and $\MB$ are equivalent relative to the class $\Int$ of interpretation functors:
    $$\PLS(\MA)\sim_\Int\PLS(\MB).$$
\end{enumerate}
In this case one has also that $\PLS_{\alpha,\beta}(\MA)\sim_\Tr\PLS_{\alpha,\beta}(\MB)$ for any infinite cardinals $\alpha\leqslant\beta$.
\end{theorem}

\begin{proof}
\ref{item1:invert}$\Longrightarrow$\ref{item3:invert}: Assume that $\MA$ and $\MB$ are strongly bi-interpretable, i.\,e., there exist interpretations $(\Gamma,\bar p,\mu_\Gamma)\colon \MA\rightsquigarrow\MB$ and $(\Delta,\bar q,\mu_\Delta)\colon \MB\rightsquigarrow\MA$, such that interpretations $(\Gamma\circ\Delta,(\bar{\bar p},\bar q),\mu_\Gamma\circ\mu_\Delta)$ and $(\Id_{L(\MA)},\emptyset,\id_A)$ are strongly homotopic as well as interpretations $(\Delta\circ\Gamma,(\bar{\bar q},\bar p),\mu_\Delta\circ\mu_\Gamma)$ and $(\Id_{L(\MB)},\emptyset,\id_B)$ are strongly homotopic. By Theorem~\ref{th:inter} there exist interpretation functors $\F=\F_{\Gamma,\bar p,\mu_\Gamma}\colon \PLS(\MA)\to \PLS(\MB)$ and $\G=\G_{\Delta,\bar q,\mu_\Delta}\colon \PLS(\MB)\to \PLS(\MA)$. By Proposition~\ref{th:ABC}, the composition $\G\circ\F$ is the interpretation $(\Gamma\circ\Delta,(\bar{\bar p},\bar q),\mu_\Gamma\circ\mu_\Delta)$\=/functor and the composition $\F\circ\G$ is the interpretation $(\Delta\circ\Gamma,(\bar{\bar q},\bar p),\mu_\Delta\circ\mu_\Gamma)$\=/functor. As we've  noted in Remark~\ref{remark:ID}, the interpretation $(\Id_{L(\MA)},\emptyset,\id_A)$\=/functor is $\id_{\PLS(\MA)}$ and the interpretation $(\Id_{L(\MB)},\emptyset,\id_B)$\=/functor is $\id_{\PLS(\MB)}$. Hence, by Theorem~\ref{th:homotopy}, one has natural isomorphisms $\G\circ\F\cong\id_{\PLS(\MA)}$ and $\F\circ\G\cong\id_{\PLS(\MB)}$. So, the categories $\PLS(\MA)$ and $\PLS(\MB)$ are equivalent relative to the class $\Int$ of interpretation functors. 

\ref{item3:invert}$\Longrightarrow$\ref{item2:invert}: Suppose that there exists a categorical equivalence $\PLS(\MA)\sim_\Int\PLS(\MB)$. It means that there exist interpretation functors $\F\colon \PLS(\MA)\to \PLS(\MB)$, $\G\colon \PLS(\MB)\to\PLS(\MA)$ and natural isomorphisms $\G\circ\F\cong\id_{\PLS(\MA)}$ and $\F\circ\G\cong\id_{\PLS(\MB)}$. Then, according to Remark~\ref{remark:natur_iso}, $(\G\circ\F)_{\alpha,\beta}\cong\id_{\PLS_{\alpha,\beta}(\MA)}$ and $(\F\circ\G)_{\alpha,\beta}\cong\id_{\PLS_{\alpha,\beta}(\MB)}$; and according to Remark~\ref{cor:comp_rest}, $(\G\circ\F)_{\alpha,\beta}=\G_{\alpha,\beta}\circ\F_{\alpha,\beta}$ and $(\F\circ\G)_{\alpha,\beta}=\F_{\alpha,\beta}\circ\G_{\alpha,\beta}$. Therefore, $\PLS_{\alpha,\beta}(\MA)\sim_\Tr\PLS_{\alpha,\beta}(\MB)$.

\ref{item2:invert}$\Longrightarrow$\ref{item1:invert}: Assume now that one has an equivalence $\PDS(\MA)\sim_\Tr\PDS(\MB)$, i.\,e., for some extended codes $(\Gamma,\bar p,\gamma)\colon L(\MA)\cup A\to L(\MB)\cup B$ and $(\Delta,\bar q,\delta)\colon L(\MB)\cup B\to L(\MA)\cup A$ there exist the translation $(\Gamma,\bar p,\gamma)$\=/functor $\F\colon \PDS(\MA)\to \PDS(\MB)$ and the translation $(\Delta,\bar q,\delta)$\=/functor $\G\colon \PLS(\MB)\to\PLS(\MA)$, and natural isomorphisms $\G\circ\F\cong\id_{\PDS(\MA)}$ and $\F\circ\G\cong\id_{\PDS(\MB)}$. By Theorem~\ref{th:functor1}, there are elementary extensions $\MA\preceq\nsA$, $\MB\preceq\nsB$ and interpretations $(\Gamma,\bar p,\mu_\Gamma)\colon \nsA\rightsquigarrow \MB$, $(\Delta,\bar q, \mu_\Delta)\colon \nsB\rightsquigarrow \MA$, such that $\mu_\Gamma\circ\gamma=\id_A$ and $\mu_\Delta\circ\delta=\id_B$; i.\,e., one has
\begin{equation*}
\begin{tikzcd}[row sep=2cm, column sep=0.1cm]
\MA   
& \preceq & \nsA \arrow[r,rightsquigarrow,"{\Gamma,\bar p,\mu_\Gamma}"] &[1.5cm]  \MB  & \preceq &  \nsB \arrow[lllll, rightsquigarrow, "{\Delta,\bar q,\mu_\Delta}", bend left =30].
\end{tikzcd}
\end{equation*}
We will also use the isomorphisms of these interpretations, namely, the $L(\MA)$\=/isomorphism  $\bar\mu_\Gamma\colon \Gamma(\MB,\bar p)\to\nsA$ and the $L(\MB)$\=/isomorphism $\bar\mu_\Delta\colon \Delta(\MA,\bar q)\to\nsB$.

By Proposition~\ref{prop:comp2}, the composition $\G\circ\F$ is the translation $({\Gamma\circ\Delta}, (\delta(\bar p),\bar q),{\delta\circ\gamma})$\=/functor and the composition $\F\circ\G$ is the translation $(\Delta\circ\Gamma,(\gamma(\bar q),\bar p),\gamma\circ\delta)$\=/functor.  Since $\G\circ\F\cong\id_{\PDS(\MA)}$ and $\F\circ\G\cong\id_{\PDS(\MB)}$, then by Corollary~\ref{cor:natur_iso0}, we obtain that there exist a definable $L(\MA)$\=/isomorphism $\bar\mu_{\Gamma\circ\Delta}\colon\Gamma\circ\Delta(\MA,(\delta(\bar p),\bar q))\to \MA$ and a definable $L(\MB)$\=/isomorphism $\bar\mu_{\Delta\circ\Gamma}\colon\Delta\circ\Gamma(\MB,(\gamma(\bar q),\bar p))\to \MB$, such that $\bar\mu_{\Gamma\circ\Delta}(\delta(\gamma(a))/{\sim_{\Gamma\circ\Delta}})=a$ for all $a\in A$ and $\bar\mu_{\Delta\circ\Gamma}(\gamma(\delta(b))/{\sim_{\Delta\circ\Gamma}})=b$ for all $b\in B$. 
   
Let us show that $\MA=\nsA$ and $\MB=\nsB$. At first, since algebraic structure $\Gamma(\MB,\bar p)$ is well-defined and $\MB\preceq\nsB$, then by Lemma~\ref{elem_emb}, the $L(\MA)$\=/structure $\Gamma(\nsB,\bar p)$ is well-defined and there exists an elementary $L(\MA)$\=/embedding $\iota\colon \Gamma(\MB,\bar p)\to\Gamma(\nsB,\bar p)$, such that $\iota(\bar b/{\sim_\Gamma})=\bar b/{\sim_\Gamma}$ for all $\bar b\in U_\Gamma(\MB,\bar p)$. So we have the following chain of interpretations:
$$
\begin{tikzcd}
\Gamma(\nsB,\bar p)\arrow[r,rightsquigarrow,"{\Gamma,\bar p,\pi}"] &\nsB \arrow[r,rightsquigarrow,"{\Delta,\bar q,\mu_\Delta}"] & \MA,
\end{tikzcd}
$$
where $\pi\colon U_{\Gamma}(\nsB,\bar p) \to (\nsB,\bar p)/{\sim_\Gamma}$ is the map, which sends $\bar b$ to $\bar b/{\sim_\Gamma}$. Therefore, by~\cite[Lemma~10]{Th_int1}, there exists an interpretation $(\Gamma\circ\Delta,(\delta(\bar p),\bar q),\pi\circ\mu_\Delta)\colon \Gamma(\nsB,\bar p)\rightsquigarrow\MA$. Let $\tilde\mu_\Delta\colon \Gamma\circ\Delta(\MA,(\delta(\bar p),\bar q))\to \Gamma(\nsB,\bar p)$ be the corresponding $L(\MA)$\=/isomorphism of this interpretation. Note that $\tilde\mu_{\Delta}(\delta(\gamma(a))/{\sim_{\Gamma\circ\Delta}})=\gamma(a)/{\sim_\Gamma}$ for all $a\in A$.

Consider the following chain of $L(\MA)$\=/isomorphisms and one elementary $L(\MA)$\=/embedding:
$$
\begin{tikzcd}
    \nsA \arrow[r,"{\bar\mu_\Gamma^{-1}}"]& \Gamma(\MB,\bar p)\arrow[r,"\iota"]&\Gamma(\nsB,\bar p)\arrow[r,"{\tilde\mu_\Delta^{-1}}"]& \Gamma\circ\Delta(\MA,(\delta(\bar p),\bar q)) \arrow[r,"{\bar \mu_{\Gamma\circ\Delta}}"] &\MA.
\end{tikzcd}
$$
Denote their composition by $\kappa$. Thus, $\kappa\colon \nsA\to\MA$ is elementary $L(\MA)$\=/embedding. At the same time, one has $\MA\preceq\nsA$. Take any element $a\in A$. One has $\kappa(a)=\bar\mu_{\Gamma\circ\Delta}(\tilde\mu_\Delta^{-1}(\iota(\bar\mu_\Gamma^{-1}(a))))=\bar\mu_{\Gamma\circ\Delta}(\tilde\mu_\Delta^{-1}(\iota(\gamma(a)/{\sim_\Gamma})))=\bar\mu_{\Gamma\circ\Delta}(\tilde\mu_\Delta^{-1}(\gamma(a)/{\sim_\Gamma}))=\bar\mu_{\Gamma\circ\Delta}(\delta(\gamma(a))/{\sim_{\Gamma\circ\Delta}})=a$. Therefore, $\kappa$ is isomorphism and $\nsA=\MA$, as required. Similarly, it is verified that $\nsB=\MB$.

Thus, we have two interpretations $(\Gamma,\bar p,\mu_\Gamma)\colon \MA\rightsquigarrow \MB$, $(\Delta,\bar q, \mu_\Delta)\colon \MB\rightsquigarrow \MA$ and definable isomorphisms $\bar\mu_{\Gamma\circ\Delta}\colon\Gamma\circ\Delta(\MA,(\delta(\bar p),\bar q))\to \MA$, $\bar\mu_{\Delta\circ\Gamma}\colon\Delta\circ\Gamma(\MB,(\gamma(\bar q),\bar p))\to \MB$, therefore, algebraic structures $\MA$ and $\MB$ are strongly bi-interpretable.
\end{proof}

We did not burden the formulation of Theorem~\ref{th:bi-inter} with the information that the same triples $(\Gamma,\bar p,\gamma)$ and $(\Delta,\bar q,\delta)$ define a bi-interpretation and functors in equivalences, although this follows from the proof of the theorem, so we note this fact separately.

\begin{corollary}\label{cor:bi0}
Interpretations $(\Gamma,\bar p,\mu_\Gamma)$ and $(\Delta,\bar q,\mu_\Delta)$ give a bi-interpretation between $\MA$ and $\MB$ if and only if translation $(\Gamma,\bar p,\gamma)$- and $(\Delta,\bar q,\delta)$\=/functors give an equivalence $\PDS(\MA)\sim_\Tr\PDS(\MB)$, if and only if interpretation $(\Gamma,\bar p,\gamma)$- and $(\Delta,\bar q,\delta)$\=/functors give an equivalence $\PLS(\MA)\sim_\Int\PLS(\MB)$, where $\mu_\Gamma\circ\gamma=\id_A$, $\mu_\Delta\circ\delta=\id_B$.    
\end{corollary}

\begin{corollary}\label{cor:bi1}
For any algebraic structures $\MA$ and $\MB$ the following conditions are equivalent:
\begin{enumerate}[label=(\arabic*)]
    \item $\MA$ and $\MB$ are strongly injectively bi-interpretable in each other;
    \item categories of definable sets over $\MA$ and $\MB$ are equivalent relative to the class $\Tr_\Inj$ of injective translation functors:    $$\DS(\MA)\sim_{\Tr_\Inj}\DS(\MB);$$
    \item categories of logical sets over $\MA$ and $\MB$ are equivalent relative to the class $\Int_\Inj$ of injective interpretation functors:    $$\LS(\MA)\sim_{\Int_\Inj}\LS(\MB).$$
\end{enumerate}
In this case one has also that $\LS_{\alpha,\beta}(\MA)\sim_{\Tr_\Inj}\LS_{\alpha,\beta}(\MB)$ for any infinite cardinals $\alpha\leqslant\beta$. In more detail, injective interpretations $(\Gamma,\bar p,\mu_\Gamma)$ and $(\Delta,\bar q,\mu_\Delta)$ give a bi-interpretation between $\MA$ and $\MB$ if and only if injective translation $(\Gamma,\bar p,\gamma)$- and $(\Delta,\bar q,\delta)$\=/functors give an equivalence $\DS(\MA)\sim_{\Tr_\Inj}\DS(\MB)$, if and only if injective interpretation $(\Gamma,\bar p,\gamma)$- and $(\Delta,\bar q,\delta)$\=/functors give an equivalence $\LS(\MA)\sim_{\Int_\Inj}\LS(\MB)$, where $\mu_\Gamma\circ\gamma=\id_A$, $\mu_\Delta\circ\delta=\id_B$.
\end{corollary}

\begin{proof}
If $\MA$ and $\MB$ are strongly injectively bi-interpretable in each other, then there exists an equivalence $\PLS(\MA)\sim_{\Int}\PLS(\MB)$; by Remarks~\ref{cor:comp_rest} and~\ref{remark:natur_iso}, restrictions of its functor gives the equivalence $\LS(\MA)\sim_{\Int_\Inj}\LS(\MB)$; and if one has an the equivalence $\LS(\MA)\sim_{\Int_\Inj}\LS(\MB)$, then one has also $\DS(\MA)\sim_{\Tr_\Inj}\DS(\MB)$. Further, due to Fact~\ref{fact:functor} and Lemma~\ref{cor:natur_iso1}, an equivalence $\DS(\MA)\sim_{\Tr_\Inj}\DS(\MB)$ gives an equivalence $\PDS(\MA)\sim_{\Tr_\Inj}\PDS(\MB)$, which delivers an injective strong bi-interpretation between $\MA$ and $\MB$.
\end{proof}

\begin{corollary}\label{cor:bi2}
If algebraic structures $\MA$ and $\MB$ are strongly absolutely bi-interpretable, then for every infinite cardinal $\beta$ one has an equivalence $\PLS_{0,\beta}(\MA)\sim_{\Tr}\PLS_{0,\beta}(\MB)$, in particular,
\begin{gather*}
\PLS_0(\MA)\sim_{\Int}\PLS_0(\MB),\\
\PDS_0(\MA)\sim_\Tr\PDS_0(\MB).
\end{gather*}
\end{corollary}

\begin{proof}
It follows from Fact~\ref{fact:functor}, Remark~\ref{cor:comp_rest} and Fact~\ref{cor:abs_hom}. 
\end{proof}

\begin{corollary}\label{cor:bi3}
If algebraic structures $\MA$ and $\MB$ are strongly absolutely injectively bi-interpretable, then for every infinite cardinal $\beta$ one has an equivalence $\LS_{0,\beta}(\MA)\sim_{\Tr_\Inj}\LS_{0,\beta}(\MB)$, in particular,
\begin{gather*}
\LS_0(\MA)\sim_{\Int_\Inj}\LS_0(\MB),\\
\DS_0(\MA)\sim_{\Tr_\Inj}\DS_0(\MB).
\end{gather*}
\end{corollary}

\begin{proof}
It follows from Fact~\ref{fact:functor}, Remark~\ref{cor:comp_rest} and Corollary~\ref{cor:inj_abs}. 
\end{proof}

In concluding this subsection, let us note the following important fact.

\begin{fact}\label{fact:eq_rel}
The categorical equivalencies $\sim_\Int$, $\sim_{\Int_\Inj}$, $\sim_\Tr$, $\sim_{\Tr_\Inj}$ are equivalencies relations on the classes of all categories $\PLS(\MA)$, $\LS(\MA)$, $\PDS(\MA)$, $\DS(\MA)$.
\end{fact}

\begin{proof}
It is clear that relations $\sim_\Int$, $\sim_{\Int_\Inj}$, $\sim_\Tr$, $\sim_{\Tr_\Inj}$ are reflexive and symmetric. It follows from Propositions~\ref{prop:comp1} and~\ref{prop:comp2} that $\sim_\Int$ and $\sim_\Tr$ are transitive. Since Fact~\ref{fact:functor} and Lemmas~\ref{cor:natur_iso1}, \ref{cor:natur_iso2} allow us to construct extensions of equivalencies $\LS(\MA)\sim_{\Int_\Inj}\LS(\MB)$ and $\DS(\MA)\sim_{\Int_\Inj}\DS(\MB)$ to $\PLS(\MA)\sim_{\Int_\Inj}\PLS(\MB)$ and $\PDS(\MA)\sim_{\Int_\Inj}\PDS(\MB)$, we get that $\sim_{\Int_\Inj}$ and  $\sim_{\Tr_\Inj}$ are transitive as well.
\end{proof}

\subsection{Syntax isomorphisms and translation isomorphisms of categories}\label{subsec:syn}

Strong bi-interpretation can be viewed as a generalization of isomorphism of algebraic structures $\MA$ and $\MB$ in the same signature to the case of algebraic structures $\MA$ and $\MB$ of different signatures. In this subsection, we introduce and study a very special kind of strong bi-interpretation that closely replicates the idea of isomorphism. Let $\MA=\langle A; L(\MA)\rangle$ and $\MB=\langle B; L(\MB)\rangle$ be algebraic structures.

\begin{definition}
    We refer to a strong bi-interpretation $(\Gamma,\bar p,\mu_\Gamma;\Delta,\bar q,\mu_\Delta)$ between $\MA$ and $\MB$ as a {\em syntax isomorphism}, if 
   the composition $(\Gamma\circ\Delta,(\bar{\bar p},\bar q),\mu_\Gamma\circ\mu_\Delta)$ equals to $(\Id_{L(\MA)},\emptyset,\id_A)$ and the composition $(\Delta\circ\Gamma,(\bar{\bar q},\bar p),\mu_\Delta\circ\mu_\Gamma)$ equals to $(\Id_{L(\MB)},\emptyset,\id_B)$ (see Definition~\ref{def:equal}).
    Correspondingly, $\MA$ and $\MB$ are {\em syntactically isomorphic}, if there exists a syntax isomorphism between them. 
\end{definition}

\begin{definition}
   We refer to an interpretation code $\Gamma$ as a {\em syntax code}, if 
    \begin{enumerate}
        \item $\dim_\Gamma=1$;
        \item $U_\Gamma(x,\bar y)=(x=x)$;
        \item $E_\Gamma(x,x^\prime,\bar y)=(x=x^\prime)$, i.\,e., $\Gamma$ is injective.
    \end{enumerate}
\end{definition}

\begin{remark}\label{remark:syntax}
Any identical code $\Id_{L(\MA)}$ is syntax and the composition $\Gamma\circ\Delta$ of syntax codes $\Gamma$ and $\Delta$ is syntax.
\end{remark}

\begin{lemma}\label{le:syn}
For any algebraic structures $\MA$ and $\MB$ the following conditions are equivalent:
\begin{enumerate}[label=(\arabic*)]
    \item $\MA$ and $\MB$ are syntactically isomorphic;
    \item there exists
    \begin{enumerate}
        \item syntax codes $\Gamma: L(\MA)\to L(\MB)$ and $\Delta\colon L(\MB)\to L(\MA)$,
        \item tuples $\bar p\in B^{\,\dim_\param \Gamma}$ and $\bar q\in A^{\dim_\param\Delta}$,
        \item a bijection $h\colon A\to B$,
    \end{enumerate}
     such that $(\Gamma,\bar p,h^{-1};\Delta,\bar q,h)$ is a strong bi-interpretation between $\MA$ and $\MB$.
\end{enumerate}
\end{lemma}

\begin{proof}
Indeed, if $(\Gamma,\bar p,\mu_\Gamma;\Delta,\bar q,\mu_\Delta)$ is a syntax isomorphism, then $U_{\Gamma\circ\Delta}(\MA,(\bar{\bar p},\bar q))/{\sim_{\Gamma\circ\Delta}}=A$, $U_{\Delta\circ\Gamma}(\MB,(\bar{\bar q},\bar p))/{\sim_{\Delta\circ\Gamma}}=B$ and $A=\mu_\Gamma(\mu_\Delta(U_{\Gamma\circ\Delta}(\MA,(\bar{\bar p},\bar q))))$, $B=\mu_\Delta(\mu_\Gamma(U_{\Delta\circ\Gamma}(\MB,(\bar{\bar q},\bar p))))$  (see Corollary~\ref{ex:Gamma_Delta}). Then $\dim\Gamma=\dim\Delta=1$ and  both interpretations $(\Gamma,\bar p,\mu_\Gamma)$ and $(\Delta,\bar q,\mu_\Delta)$ are equal to interpretations $(\Gamma^\prime,\bar p,\mu_\Gamma)$ and $(\Delta^\prime,\bar q,\mu_\Delta)$ with syntax codes $\Gamma^\prime$ and $\Delta^\prime$. Furthermore, $\mu_\Gamma\colon B\to A$ and $\mu_\Delta\colon A\to B$ are bijections, such that $\mu_\Gamma\circ\mu_\Delta=\id_A$. Therefore, the strong bi-interpretation $(\Gamma^\prime,\bar p,\mu_\Delta^{-1};\Delta^\prime, \bar q,\mu_\Delta)$ is the one we are looking for. For the sake of precision, let's note that we are using here Corollary~\ref{cor1:par}, which we will prove a little later.
    
Inversely, if $(\Gamma,\bar p,h^{-1};\Delta,\bar q,h)$ is a strong bi-interpretation with syntax codes, then the composition $(\Gamma\circ\Delta,(\bar{\bar p},\bar q),h^{-1}\circ h)$ equals to $(\Id_{L(\MA)},\emptyset,\id_A)$ and the composition $(\Delta\circ\Gamma,(\bar{\bar q},\bar p),h\circ h^{-1})$ equals to $(\Id_{L(\MB)},\emptyset,\id_B)$, i.\,e., $\MA$ and $\MB$ are syntactically isomorphic.
\end{proof}

Every isomorphism of algebraic structures in the same language is a syntax isomorphism. 

\begin{theorem}\label{th:syn_iso}
For any algebraic structures $\MA=\langle A; L(\MA)\rangle$ and $\MB=\langle B; L(\MB)\rangle$ the following conditions are equivalent:
\begin{enumerate}[label=(\arabic*)]
    \item\label{syn1} $\MA$ and $\MB$ are syntactically isomorphic;
    \item\label{syn2} $\DS(\MA)\simeq_{\Tr_\Inj}\DS(\MB)$;
    \item\label{syn3}$\LS(\MA)\simeq_{\Int_\Inj}\LS(\MB)$;
    \item\label{syn4} $\PDS(\MA)\simeq_{\Tr}\PDS(\MB)$;
    \item\label{syn5} $\PLS(\MA)\simeq_{\Int}\PLS(\MB)$.
\end{enumerate}
\end{theorem}

\begin{proof}
\ref{syn1}$\Longrightarrow$\ref{syn2}: Let $\MA$ and $\MB$ be syntactically isomorphic. Then by Lemma~\ref{le:syn}, there exists a strong bi-interpretation  $(\Gamma,\bar p,\mu_\Gamma;\Delta,\bar q,\mu_\Delta)$, such that  $\mu_\Gamma=\mu_\Delta^{-1}$ and codes $\Gamma,\Delta$ are syntax. According to Corollary~\ref{cor:bi0}, there exists an equivalence $\PLS(\MA)\sim_{\Int_\Inj}\PLS(\MB)$. By Lemma~\ref{le:equal}, it is indeed an isomorphism $\PLS(\MA)\simeq_{\Int}\PLS(\MB)$, so Remarks~\ref{cor:comp_rest} and~\ref{remark:natur_iso} involve that the restriction of the isomorphism $\PLS(\MA)\simeq_{\Int_\Inj}\PLS(\MB)$  gives an isomorphism $\DS(\MA)\simeq_{\Tr_\Inj}\DS(\MB)$. Implication~\ref{syn3}$\Longrightarrow$\ref{syn2} is similar.

\ref{syn2}$\Longrightarrow$\ref{syn4}: If $\DS(\MA)\simeq_{\Tr_\Inj}\DS(\MB)$, then due to Fact~\ref{fact:functor} and Lemma~\ref{cor:natur_iso1} one has $\PDS(\MA)\simeq_{\Tr}\PDS(\MB)$. \ref{syn2}$\Longrightarrow$\ref{syn3}: If $\DS(\MA)\simeq_{\Tr_\Inj}\DS(\MB)$, then $\LS(\MA)\sim_{\Int_\Inj}\LS(\MB)$ due to Corollary~\ref{cor:bi1}; and by Lemma~\ref{cor:natur_iso},  $\LS(\MA)\simeq_{\Int_\Inj}\LS(\MB)$. Implication \ref{syn4}$\Longrightarrow$\ref{syn5} is similar.

\ref{syn5}$\Longrightarrow$\ref{syn1}: Suppose that $\PLS(\MA)\simeq_{\Int}\PLS(\MB)$. Then
there exists interpretation functors $\F\colon \PLS(\MA)\to\PLS(\MB)$ and $\G\colon \PLS(\MB)\to\PLS(\MA)$, besides, $\G\circ\F=\id_{\PLS(\MA)}$ and $\F\circ\G=\id_{\PLS(\MB)}$. Let $(\Gamma,\bar p,\gamma)$ and $(\Delta,\bar q, \delta)$ be the corresponding extended codes. By Corollary~\ref{cor:bi0}, there is strong bi-interpretation $(\Gamma,\bar p,\mu_\Gamma;\Delta,\bar q,\mu_\Delta)$ between $\MA$ and $\MB$, such that  $\mu_\Gamma\circ\gamma=\id_A$ and $\mu_\Delta\circ\delta=\id_B$. And by Lemma~\ref{le:equal}, the composition $(\Gamma\circ\Delta,(\bar{\bar p},\bar q),\mu_\Gamma\circ\mu_\Delta)$ equals to $(\Id_{L(\MA)},\emptyset,\id_A)$ and the composition $(\Delta\circ\Gamma,(\bar{\bar q},\bar p),\mu_\Delta\circ\mu_\Gamma)$ equals to $(\Id_{L(\MB)},\emptyset,\id_B)$, i.\,e., $\MA$ and $\MB$ are syntactically isomorphic.
\end{proof}

Let us formulate and prove an analogue of Fact~\ref{fact:eq_rel} for the case of categorical isomorphisms.

\begin{fact}\label{fact:eq_rel_syn}
The categorical isomorphisms $\simeq_\Int$, $\simeq_{\Int_\Inj}$, $\simeq_\Tr$, $\simeq_{\Tr_\Inj}$ are equivalencies relations on the classes of all categories $\PLS(\MA)$, $\LS(\MA)$, $\PDS(\MA)$, $\DS(\MA)$.
\end{fact}

\begin{proof}
Again, reflexivity and symmetry are obvious. Transitivity of $\simeq_\Int$ follows from Propositions~\ref{prop:comp1}. For all other types of isomorphisms, transitivity follows from Theorem~\ref{th:syn_iso}.
\end{proof}

\begin{example}
Mal'cev correspondence between nilpotent $k$\=/groups $G$ and nilpotent Lie $k$\=/algebras $L$ over a field $k$ of characteristic zero gives an interesting example of syntax isomorphisms for algebraic structures in different languages~\cite{ABM}. Moreover, in this case one has $\LS^\term_0(G)=\LS^\term_0(L)$, $\DS^\term_0(G)=\DS^\term_0(L)$ and $\AS(G)=\AS(L)$.
\end{example}

\subsection{Horizontal transitivity of strong homotopy}\label{subsec:horiz}

The categorical counterpart of interpretations allows us to prove important facts of the theory of interpretations with relative ease. Direct proofs of the corresponding results, without transitioning to categories, seem unnecessarily cumbersome.

For natural isomorphisms, vertical $\circ$ and horizontal $\ast$ compositions are defined in the standard way~\cite[\S\,II.4, 5]{MacLane}. Homotopies of interpretations, like the natural isomorphism of categorical functors, have both vertical and horizontal compositions. Vertical compositions of homotopies were described in~\cite[Lemma~5]{Th_int1}. Here in Subsection~\ref{subsec:big_iso} we will return to them.

{\bf Horizontal transitivity and horizontal compositions.} 
Now it's time to discuss horizontal compositions.

\begin{proposition}[about horizontal composition of homotopy]\label{prop1:par}
Let $(\theta_{\MA,\MB})\colon (\Gamma_1,\bar p_1,\mu_{\Gamma_1}) \to (\Gamma_2,\bar p_2,\mu_{\Gamma_2})$ be strong homotopy of interpretations of $\MA$ in $\MB$ and $(\theta_{\MB,\MC})\colon (\Delta_1,\bar q_1,\mu_{\Delta_1})\to (\Delta_2,\bar q_2,\mu_{\Delta_2})$ be strong homotopy of interpretations of $\MB$ in $\MC$. Then there exists a strong homotopy $(\theta_{\MB,\MC}\ast\theta_{\MA,\MB})\colon (\Gamma_1\circ\Delta_1,(\bar{\bar p}_1,\bar q_1),\mu_{\Gamma_1}\circ\mu_{\Delta_1})\to(\Gamma_2\circ\Delta_2,(\bar{\bar p}_2,\bar q_2),\mu_{\Gamma_2}\circ\mu_{\Delta_2})$ of interpretations of $\MA$ in $\MC$. 
\end{proposition}

\begin{proof}
According to Theorem~\ref{th:homotopy}, the interpretation $(\Gamma_1,\bar p_1,\mu_{\Gamma_1})$- and $(\Gamma_2,\bar p_2,\mu_{\Gamma_2})$\=/functors $\F_1,\F_2\colon\PLS(\MA)\to\PLS(\MB)$ are naturally isomorphic, as well as the interpretation $(\Delta_1,\bar q_1,\mu_{\Delta_1})$- and $(\Delta_2,\bar q_2,\mu_{\Delta_2})$\=/functors $\G_1,\G_2\colon\PLS(\MB)\to\PLS(\MC)$ are naturally isomorphic. Suppose that $\eta\colon \F_1\to\F_2$ and $\tau\colon\G_1\to\G_2$ are the corresponding natural isomorphisms. Then there exists the natural isomorphism $\tau\ast\eta\colon\G_1\circ\F_1\to\G_2\circ\F_2$~\cite{MacLane}. By Proposition~\ref{th:ABC}, $\G_1\circ\F_1$ is the interpretation $(\Gamma_1\circ\Delta_1,(\bar{\bar p}_1,\bar q_1),\mu_{\Gamma_1}\circ\mu_{\Delta_1})$\=/functor and $\G_2\circ\F_2$ is the interpretation $(\Gamma_2\circ\Delta_2,(\bar{\bar p}_2,\bar q_2),\mu_{\Gamma_2}\circ\mu_{\Delta_2})$\=/functor. Due to Corollary~\ref{cor:natur_iso_inter}, the interpretations $(\Gamma_1\circ\Delta_1,(\bar{\bar p}_1,\bar q_1),\mu_{\Gamma_1}\circ\mu_{\Delta_1}),(\Gamma_2\circ\Delta_2,(\bar{\bar p}_2,\bar q_2),\mu_{\Gamma_2}\circ\mu_{\Delta_2})\colon \MA\rightsquigarrow\MC$ are strongly homotopic. So, $(\theta_{\MB,\MC}\ast\theta_{\MA,\MB})={\bf H}(\tau\ast\eta)$.
\end{proof}

\begin{notation}
We will use notation $(\theta_{\MB,\MC}\ast\theta_{\MA,\MB})$ for ${\bf H}({\bf N}((\theta_{\MB,\MC}))\ast{\bf N}((\theta_{\MA,\MB})))$ below. Thus, we mean that $(\theta_{\MB,\MC}\ast\theta_{\MA,\MB})$ is exactly a strong homotopy of the interpretations $(\Gamma_1\circ\Delta_1,(\bar{\bar p}_1,\bar q_1),\mu_{\Gamma_1}\circ\mu_{\Delta_1})$ and $(\Gamma_2\circ\Delta_2,(\bar{\bar p}_2,\bar q_2),\mu_{\Gamma_2}\circ\mu_{\Delta_2})$. We will write also $(\theta_{\MB,\MC})\ast(\theta_{\MA,\MB})=(\theta_{\MB,\MC}\ast\theta_{\MA,\MB})$.
\end{notation}

\begin{corollary}\label{cor:ast_N}
For a strong homotopy $(\theta_{\MA,\MB})$ of interpretations of $\MA$ in $\MB$ and a strong homotopy $(\theta_{\MB,\MC})$ of interpretations of $\MB$ in $\MC$ one has ${\bf \hat N}((\theta_{\MB,\MC})\ast(\theta_{\MA,\MB}))={\bf \hat N}((\theta_{\MB,\MC}))\ast{\bf \hat N}((\theta_{\MA,\MB}))$.
\end{corollary}

Let us formulate the following special cases of Proposition~\ref{prop1:par} separately. They can be proven similarly, but with minor adjustments.

\begin{corollary}\label{cor1:par}
Let $(\Gamma_1,\bar p_1,\mu_{\Gamma_1}),(\Gamma_2,\bar p_2,\mu_{\Gamma_2})\colon \MA\rightsquigarrow\MB$ and $(\Delta_1,\bar q_1,\mu_{\Delta_1}),(\Delta_2,\bar q_2,\mu_{\Delta_2})\colon \MB\rightsquigarrow\MC$ be pairs of equal interpretations. Then $(\Gamma_1\circ\Delta_1,(\bar{\bar p}_1,\bar q_1),\mu_{\Gamma_1}\circ\mu_{\Delta_1}),(\Gamma_2\circ\Delta_2,(\bar{\bar p}_2,\bar q_2),\mu_{\Gamma_2}\circ\mu_{\Delta_2})\colon \MA\rightsquigarrow\MC$ are equal  interpretations.
\end{corollary}

\begin{proof}
In this case $\F_1=\F_2$ and $\G_1=\G_2$, and therefore $\G_1\circ\F_1=\G_2\circ\F_2$. Thus, due to  Lemma~\ref{le:equal}, interpretations $(\Gamma_1\circ\Delta_1,(\bar{\bar p}_1,\bar q_1),\mu_{\Gamma_1}\circ\mu_{\Delta_1}),(\Gamma_2\circ\Delta_2,(\bar{\bar p}_2,\bar q_2),\mu_{\Gamma_2}\circ\mu_{\Delta_2})\colon \MA\rightsquigarrow\MC$ are equal. 
\end{proof}

\begin{corollary}\label{cor2:par}
Suppose that there exist interpretations $(\Gamma,\bar p,\mu_\Gamma)\colon \MA\rightsquigarrow\MB$ and $(\Delta^\prime,\bar q,\mu_\Delta)\colon \MB\rightsquigarrow\MC$. Then for any codes $\Gamma^\prime\approx\Gamma$ and $\Delta^\prime\approx\Delta$ (see~\cite[Definition~3]{Th_int1}) the exist an interpretation $(\Gamma^\prime\circ\Delta^\prime,(\bar{\bar p},\bar q),\mu_\Gamma\circ\mu_\Delta)\colon \MA\rightsquigarrow\MC$, and it equals to the interpretation $(\Gamma\circ\Delta,(\bar{\bar p},\bar q),\mu_\Gamma\circ\mu_\Delta)\colon \MA\rightsquigarrow\MC$.
\end{corollary}

\begin{proof}
As in Example~\ref{ex:approx}, there exist interpretations $(\Gamma^\prime,\bar p,\mu_\Gamma)\colon \MA\rightsquigarrow\MB$ and $(\Delta,\bar q,\mu_\Delta)\colon \MB\rightsquigarrow\MC$, which equal to $(\Gamma,\bar p,\mu_\Gamma)\colon \MA\rightsquigarrow\MB$ and $(\Delta,\bar q,\mu_\Delta)\colon \MB\rightsquigarrow\MC$. Therefore, Corollary~\ref{cor1:par} gives the required. 
\end{proof}

\begin{fact}
Let $(\Gamma,\bar p,\mu_\Gamma)\colon \MA\rightsquigarrow\MB$ and $(\Delta,\bar q,\mu_\Delta)\colon\MB\rightsquigarrow\MC$ be interpretations, which are equal to interpretations $(\Gamma^\prime,\bar p^\prime,\mu_{\Gamma^\prime})\colon \MA\rightsquigarrow\MB$ and $(\Delta^\prime,\bar q^\prime,\mu_{\Delta^\prime})\colon\MB\rightsquigarrow\MC$ with injective (syntax, Diophantine) codes. Then the interpretation $(\Gamma\circ\Delta,(\bar{\bar p},\bar q),\mu_\Gamma\circ\mu_\Delta)\colon \MA\rightsquigarrow\MC$ equals an interpretation with injective (syntax, Diophantine) code.
\end{fact}

\begin{proof}
By Corollary~\ref{cor1:par}, interpretations $(\Gamma\circ\Delta,(\bar{\bar p},\bar q),\mu_\Gamma\circ\mu_\Delta),(\Gamma^\prime\circ\Delta^\prime,(\bar{\bar p}^\prime,\bar q^\prime),\mu_{\Gamma^\prime}\circ\mu_{\Delta^\prime})\colon \MA\rightsquigarrow\MC$ are equal. If $\Gamma^\prime$ and $\Delta^\prime$ are injective, then the interpretation $\Gamma^\prime\circ\Delta^\prime$ equals to $(\Gamma^\prime\circ\Delta^\prime)_{\rm inj}$~\cite[Remark~3] {Th_int1}. If $\Gamma^\prime$ and $\Delta^\prime$ are Diophantine, then $\Gamma^\prime\circ\Delta^\prime$ is Diophantine~\cite[Remark~9]{Th_int1}. If $\Gamma^\prime$ and $\Delta^\prime$ are syntax, then $\dim\Gamma^\prime=\dim\Delta^\prime=\dim \Gamma^\prime\circ\Delta^\prime=1$ and the formula $U_{\Gamma^\prime\circ\Delta^\prime}$ is equivalent to $(x=x)$, and $E_{\Gamma^\prime\circ\Delta^\prime}$ is equivalent to $(x=x^\prime)$, i.\,e., the interpretation $\Gamma^\prime\circ\Delta^\prime$ equals to interpretation with syntax code.
\end{proof}

\begin{lemma}\label{lemma:abs}
Let $(\Gamma_1,\emptyset,\mu_{\Gamma_1}),(\Gamma_2,\emptyset,\mu_{\Gamma_2})\colon \MA\rightsquigarrow\MB$ and $(\Delta_1,\emptyset,\mu_{\Delta_1}),(\Delta_2,\emptyset,\mu_{\Delta_2})\colon \MB\rightsquigarrow\MC$ be pairs of absolutely strongly homotopic interpretations. Then $(\Gamma_1\circ\Delta_1,\emptyset,\mu_{\Gamma_1}\circ\mu_{\Delta_1}),(\Gamma_2\circ\Delta_2,\emptyset,\mu_{\Gamma_2}\circ\mu_{\Delta_2})\colon \MA\rightsquigarrow\MC$ are absolutely strongly homotopic interpretations as well.
\end{lemma}

\begin{proof}
By Proposition~\ref{prop1:par}, there exists a strong homotopy $(\theta)\colon (\Gamma_1\circ\Delta_1,\emptyset,\mu_{\Gamma_1}\circ\mu_{\Delta_1})\to (\Gamma_2\circ\Delta_2,\emptyset,\mu_{\Gamma_2}\circ\mu_{\Delta_2})$. Suppose that $\eta\colon \F_1\to\F_2$ and $\tau\colon\G_1\to\G_2$ are natural isomorphisms between interpretation functors from the proof of Proposition~\ref{prop1:par}. Then $(\theta)={\bf H}(\tau\ast\eta)$. According to Fact~\ref{cor:abs_hom}, the restrictions of these natural isomorphisms to subcategories $\PLS_0(\MA)$ and $\PLS_0(\MB)$ are also natural isomorphisms between absolute interpretation functors $\F_1,\F_2\colon\PLS_0(\MA)\to\PLS_0(\MB)$ and $\G_1,\G_2\colon\PLS_0(\MB)\to\PLS_0(\MC)$. Then the natural isomorphism $\tau\ast\eta\colon\G_1\circ\F_1\to\G_2\circ\F_2$ corresponds the isomorphism $(\tau\ast\eta)_X=\tau_{\F_2(X)}\circ\G_1(\eta_X)$ to object any $X$ from $\PLS(\MA)$. Therefore, the restriction of natural isomorphism $\tau\ast\eta$ to subcategory $\PLS_0(\MA)$ is a natural isomorphism between absolute interpretation functors $\G_1\circ\F_1,\G_2\circ\F_2\colon\PLS_0(\MA)\to\PLS_0(\MC)$. By Lemma~\ref{le:natur_iso}, one has ${\bf N}((\theta))=\tau\ast\eta$. Thus, by Fact~\ref{cor:abs_hom}, we obtain that the homotopy $(\theta)$ is absolute, as required.
\end{proof}

{\bf Associativity of interpretations and homotopies.} 
The question of the associativity of interpretations is also related to the topic of horizontal transitivity. If we have a sequence of interpretations
$$
\MA \stackrel{\Gamma}{\rightsquigarrow}\MB\stackrel{\Delta}{\rightsquigarrow}\MC \stackrel{\Sigma}{\rightsquigarrow}\MD,
$$
then by~\cite[Lemma~10]{Th_int1} we understand that $\MA$ is interpretable in $\MD$. However, there exist two ways from $\MA$ to $\MD$, because we can proceed with different parentheses, $(\Gamma\circ\Delta)\circ\Sigma$ and $\Gamma\circ(\Delta\circ\Sigma)$. It turns out that the resulting interpretations are equal in the sense of Definition~\ref{def:equal}.

\begin{proposition}[about associativity of interpretations with parameters]\label{ABCD}
Let $(\Gamma,\bar p,\mu_\Gamma)\colon \MA\rightsquigarrow\MB$, $(\Delta,\bar q,\mu_\Delta)\colon\MB\rightsquigarrow\MC$ and $(\Sigma,\bar r,\mu_\Sigma)\colon \MC\rightsquigarrow\MD$ be interpretations. Then interpretations $(\Gamma\circ(\Delta\circ\Sigma),({\bar{\bar p}},(\bar{\bar q},\bar r)),\mu_\Gamma\circ(\mu_\Delta\circ\mu_\Sigma))\colon \MA\rightsquigarrow\MD$ and $((\Gamma\circ\Delta)\circ\Sigma,((\bar{\bar{\bar p}},\bar{\bar q}),\bar r),(\mu_\Gamma\circ\mu_\Delta)\circ\mu_\Sigma)\colon \MA\rightsquigarrow\MD$ are equal, besides $\dim_\param\Gamma\circ(\Delta\circ\Sigma)=\dim_\param(\Gamma\circ\Delta)\circ\Sigma$.
\end{proposition}

\begin{proof}
Suppose that  $\F\colon\PLS(\MA)\to\PLS(\MB)$, $\G\colon \PLS(\MB)\to \PLS(\MC)$, and $\mathcal{H}\colon\PLS(\MC)\to \PLS(\MD)$ are interpretation functors, which correspond to the given interpretations. Then by Proposition~\ref{th:ABC}, $\F\circ (\G\circ \mathcal{H})$ is the interpretation functor of the interpretation $\Gamma\circ(\Delta\circ\Sigma)$ and $(\F\circ \G)\circ \mathcal{H}$ is the interpretation functor of the interpretation $(\Gamma\circ\Delta)\circ\Sigma$. Since $\F\circ (\G\circ \mathcal{H})=(\F\circ \G)\circ \mathcal{H}$, then by Lemma~\ref{le:equal} one has the equality of interpretations $\Gamma\circ(\Delta\circ\Sigma)$ and $(\Gamma\circ\Delta)\circ\Sigma$. Additionally, one has $\dim_\param\Gamma\circ(\Delta\circ\Sigma)= \dim_\param\Gamma\cdot\dim \Delta\circ\Sigma+\dim_\param \Delta\circ\Sigma=\dim_\param\Gamma\cdot\dim \Delta\cdot \dim \Sigma+\dim_\param \Delta\cdot \dim\Sigma+\dim_\param\Sigma=(\dim_\param\Gamma\cdot\dim \Delta+\dim_\param \Delta)\cdot \dim \Sigma+\dim_\param\Sigma= \dim_\param\Gamma\circ\Delta\cdot\dim\Sigma+\dim_\param\Sigma= \dim_\param(\Gamma\circ\Delta)\circ\Sigma$.
\end{proof}

\begin{proposition}[about associativity of horizontal compositions of homotopies]\label{ABCD_hom}
Let $(\theta_{\MA,\MB})\colon (\Gamma_1,\bar p_1,\mu_{\Gamma_1})\to (\Gamma_2,\bar p_2,\mu_{\Gamma_2})$ be strong homotopy of interpretations of $\MA$ in $\MB$, $(\theta_{\MB,\MC})\colon (\Delta_1,\bar q_1,\mu_{\Delta_1})\to (\Delta_2,\bar q_2,\mu_{\Delta_2})$ be strong homotopy of interpretations of $\MB$ in $\MC$ and $(\theta_{\MC,\MD})\colon (\Sigma_1,\bar r_1,\mu_{\Sigma_1})\to (\Sigma_2,\bar r_2,\mu_{\Sigma_2})$ be strong homotopy of interpretations of $\MC$ in $\MD$. Then the homotopies $(\theta_{\MC,\MD}\ast(\theta_{\MB,\MC}\ast\theta_{\MA,\MB}))$ and $((\theta_{\MC,\MD}\ast\theta_{\MB,\MC})\ast\theta_{\MA,\MB})$ are equal.
\end{proposition}

\begin{proof}
By Corollary~\ref{cor:ast_N}, one has ${\bf \hat N}((\theta_{\MC,\MD}\ast(\theta_{\MB,\MC}\ast\theta_{\MA,\MB})))={\bf \hat N}((\theta_{\MC,\MD}))\ast {\bf \hat N}((\theta_{\MB,\MC})\ast(\theta_{\MA,\MB}))={\bf \hat N}((\theta_{\MC,\MD}))\ast({\bf \hat N}((\theta_{\MB,\MC}))\ast{\bf \hat N}((\theta_{\MA,\MB})))$ and ${\bf \hat N}((\theta_{\MC,\MD}\ast\theta_{\MB,\MC})\ast\theta_{\MA,\MB})=
({\bf \hat N}((\theta_{\MC,\MD}))\ast{\bf \hat N}((\theta_{\MB,\MC})))\ast{\bf \hat N}((\theta_{\MA,\MB}))$. Further, for natural isomorphisms the equality ${\bf \hat N}((\theta_{\MC,\MD}))\ast({\bf \hat N}((\theta_{\MB,\MC}))\ast{\bf \hat N}((\theta_{\MA,\MB})))=({\bf\hat  N}((\theta_{\MC,\MD}))\ast{\bf \hat N}((\theta_{\MB,\MC})))\ast{\bf \hat N}((\theta_{\MA,\MB}))$ holds~\cite{MacLane}. Therefore, due to Fact~\ref{fact:hat3}, one has $(\theta_{\MC,\MD}\ast(\theta_{\MB,\MC}\ast\theta_{\MA,\MB}))=_{\rm hom}((\theta_{\MC,\MD}\ast\theta_{\MB,\MC})\ast\theta_{\MA,\MB})$.
\end{proof}

{\bf Transitivity of bi-interpretations and invertible interpretations.} 
From the very beginning, it was natural to expect from the strong bi-interpretation that this is an equivalence relation on algebraic structures. But only now can we show why this is indeed the case. First, we will show the transitivity of invertible interpretations.

\begin{proposition}\label{prop2:par}
If interpretations $(\Gamma,\bar p,\mu_\Gamma)\colon\MA\rightsquigarrow\MB$ and $(\Delta,\bar q,\mu_\Delta)\colon\MB\rightsquigarrow\MC$ are (left, right, two-sided) invertible, then the composition $(\Gamma\circ\Delta,(\bar{\bar p},\bar q),\mu_{\Gamma}\circ\mu_{\Delta})\colon \MA\rightsquigarrow\MC$ is (left, right, two-sided) invertible as well.
\end{proposition}

\begin{proof}
We show the case of right invertible interpretations; left invertible and two-sided invertible are similar. Let $(\Gamma^\prime,\bar p^\prime,\mu_{\Gamma^\prime})\colon\MB\rightsquigarrow\MA$ and $(\Delta^\prime,\bar q^\prime,\mu_{\Delta^\prime})\colon\MC\rightsquigarrow\MB$ be interpretations, such that interpretations $(\Gamma\circ\Gamma^\prime,(\bar{\bar p},\bar p^\prime),\mu_{\Gamma}\circ\mu_{\Gamma^\prime})$ and $(\Id_{L(\MA)},\emptyset,\id_A)$ are strongly homotopic, as well as $(\Delta\circ\Delta^\prime,(\bar{\bar q},\bar q^\prime),\mu_{\Delta}\circ\mu_{\Delta^\prime})$ and $(\Id_{L(\MB)},\emptyset,\id_B)$. Further, we will write interpretations just via their codes and denote strong homotopy by $\sim$. Also we will use Proposition~\ref{prop1:par} for transitivity of strong homotopy, Proposition~\ref{ABCD} for associativity of compositions and Corollary~\ref{cor2:par} with~\cite[Lemma~9]{Th_int1} to assert that $\Gamma\sim\Gamma\circ\Id\sim\Id\circ\Gamma$. So, we have $\Gamma\circ\Delta\circ\Delta^\prime\circ\Gamma^\prime\sim\Gamma\circ\Id_{L(\MB)}\circ\Gamma^\prime\sim\Gamma\circ\Gamma^\prime\sim\Id_{L(\MA)}$, as required. 
\end{proof}

\begin{proposition}
Strong bi-interpretability with parameters is transitive. It means, if $\MA$ and $\MB$ are strongly bi-interpretable, and $\MB$ and $\MC$ are strongly bi-interpretable, then $\MA$ and $\MC$ are strongly bi-interpretable. Moreover, strong injective bi-interpretations, strong absolute bi-interpretations, strong absolute injective interpretations, and syntax isomorphisms are transitive as well.
\end{proposition}

\begin{proof}
By Fact~\ref{fact:eq_rel}, relations $\sim_\Int$ and $\sim_{\Int_\Inj}$ are transitive. Therefore, Theorem~\ref{th:bi-inter} and Corollary~\ref{cor:bi1} give the result in the general case and in the case of strong injective bi-interpretations. The result about syntax isomorphisms follows from Theorem~\ref{th:syn_iso} and Fact~\ref{fact:eq_rel_syn}. Finally, for strong absolute bi-interpretations we repeat arguments from the proof of Proposition~\ref{prop2:par}, and using Lemma~\ref{lemma:abs} we obtain the required. 
\end{proof}

\begin{corollary}
Strong bi-interpretability, strong injective bi-interpretability, strong absolute bi-interpretability, strong absolute injective bi-interpretability, and syntactic isomorphism are equivalence relations on the class of all algebraic structures.
\end{corollary}

\subsection{Big interpretation and logical-geometric categories}\label{subsec:big_iso}

The results demonstrated in this article allow us to expand the discussion to big categories. These require separate research, so we will only outline the most obvious conclusions here.

Let us denote by ${\bf AS}$ the category, which objects are all algebraic structures $\MA$, $\MB$, $\MC$, $\MD$, and so on; and morphisms are interpretations $(\Gamma,\bar p,\mu_\Gamma)\colon \MA\rightsquigarrow\MB$ up to equality (see Definition~\ref{def:equal}), i.\,e, classes of equivalence of equal interpretations.

\begin{fact}
${\bf AS}$ is indeed a category.
\end{fact}

\begin{proof}
By Corollary~\ref{cor1:par}, the composition of interpretations is correctly defined in the classes of equivalence of equal interpretations. For any algebraic structure $\MA=\langle A; L(\MA)\rangle$ the identical interpretation $(\Id_{L(\MA)},\emptyset,\id_A)$ is the identical morphism $\id_\MA$. By Proposition~\ref{ABCD}, one has the associativity on morphisms. 
\end{proof}

\begin{remark}
Algebraic structures $\MA$ and $\MB$ are categorically isomorphic in ${\bf AS}$ if and only if they are syntactically isomorphic.
\end{remark}

Denote by ${\bf PLG}$ the category, which objects are categories $\PLS(\MA)$ for all algebraic structures $\MA$ and morphisms are interpretation functors.

\begin{fact}
${\bf PLG}$ is indeed a category.
\end{fact}

\begin{proof}
It follows from Proposition~\ref{prop:comp1} and Remark~\ref{remark:ID}.  
\end{proof}

\begin{remark}
Categories $\PLS(\MA)$ and $\PLS(\MB)$ are isomorphic as objects in ${\bf PLG}$ if and only if they are isomorphic as categories relative to the class of interpretation functors $\Int$, i.\,e, $\PLS(\MA)\simeq_\Int\PLS(\MB)$.
\end{remark}

\begin{theorem}\label{TH}
The categories ${\bf AS}$ and ${\bf PLG}$ are isomorphic.
\end{theorem}

\begin{proof}
One-to-one correspondence on objects of the categories ${\bf AS}$ and ${\bf PLG}$ is trivial, namely, an algebraic structure $\MA$ corresponds to the category $\PLS(\MA)$. Due to Fact~\ref{fact:hat1}, we may assume that ${\bf \hat F}$ and ${\bf \hat I}$ are functors on morphisms in the categories ${\bf AS}$ and ${\bf PLG}$, so one has ${\bf \hat F}\colon {\bf AS}\to {\bf PLG}$, ${\bf\hat  I}\colon {\bf PLG}\to {\bf AS}$. And by Fact~\ref{fact:hat2}, ${\bf \hat I}\circ{\bf \hat F}=\id_{{\bf AS}}$ and  ${\bf \hat F}\circ{\bf \hat I}=\id_{{\bf PLG}}$. Therefore, functors ${\bf \hat F}$ and ${\bf \hat I}$ give an isomorphism of categories ${\bf AS}$ and ${\bf PLG}$.   
\end{proof}

Let ${\bf AS}/{\sim}$ be the category of all algebraic structures and classes of equivalencies of strongly homotopic interpretations as morphisms. 

\begin{fact}
    ${\bf AS}/{\sim}$ is indeed a category.
\end{fact}

\begin{proof}
According to~\cite[Lemma~5]{Th_int1}, strong homotopy is indeed an equivalence relation, so morphisms in ${\bf AS}/{\sim}$ are well-defined. Due to Proposition~\ref{prop1:par}, compositions of morphisms are well-defined as well. 
\end{proof}

\begin{remark}
    Algebraic structures $\MA$ and $\MB$ are categorically isomorphic in ${\bf AS}/{\sim}$ if and only if they are strongly bi-interpretable.
\end{remark}

Denote by ${\bf PLG}/{\sim}$ the category of all categories $\PLS(\MA)$ and classes of naturally isomorphic interpretation functors as morphisms. 

\begin{remark}
    Categories $\PLS(\MA)$ and $\PLS(\MB)$ are isomorphic as objects in ${\bf PLG}/{\sim}$ if and only if they are equivalent as categories relative to the class of interpretation functors $\Int$, i.\,e, $\PLS(\MA)\sim_\Int\PLS(\MB)$.
\end{remark}

\begin{theorem}\label{TH_sim}
    The categories ${\bf AS}/{\sim}$ and ${\bf PLG}/{\sim}$ are isomorphic.
\end{theorem}

\begin{proof}
We continue the argument in the proof of Theorem~\ref{TH}. By Theorem~\ref{th:homotopy}, if two interpretations $(\Gamma_1,\bar p_1,\mu_{\Gamma_1}),(\Gamma_2,\bar p_2,\mu_{\Gamma_2})\colon \MA\rightsquigarrow\MB$ are strongly homotopic, then 
functors ${\bf \hat F}((\Gamma_1,\bar p_1,\mu_{\Gamma_1}))$ and ${\bf \hat F}((\Gamma_2,\bar p_2,\mu_{\Gamma_2}))$ are naturally isomorphic. And by Corollary~\ref{cor:natur_iso_inter}, if interpretation functors ${\bf \hat F}((\Gamma_1,\bar p_1,\mu_{\Gamma_1}))$ and ${\bf \hat F}((\Gamma_2,\bar p_2,\mu_{\Gamma_2}))$ are naturally isomorphic, then interpretations $(\Gamma_1,\bar p_1,\mu_{\Gamma_1}),(\Gamma_2,\bar p_2,\mu_{\Gamma_2})\colon \MA\rightsquigarrow\MB$ are strongly homotopic.
\end{proof}

The strongest connection between categories ${\bf AS}$ and ${\bf PLG}$ is revealed in the language of $2$\=/categories. Denote by $2$\=/${\bf PLG}$ the strict $2$\=/category ${\bf PLG}$ with natural isomorphisms between functors as $2$\=/cells. It is $2$\=/subcategory in $2$\=/{\bf Cat}. To define a similar $2$\=/category $2$\=/${\bf AS}$, we will need several auxiliary considerations and results.

Let $(\Gamma_1,\bar p_1,\mu_{\Gamma_1}),(\Gamma_2,\bar p_2,\mu_{\Gamma_2}), (\Gamma_3,\bar p_3,\mu_{\Gamma_3}), (\Gamma^\prime_1,\bar p^\prime_1,\mu_{\Gamma^\prime_1}), (\Gamma^\prime_2,\bar p^\prime_2,\mu_{\Gamma^\prime_2}),(\Gamma^\prime_3,\bar p^\prime_3,\mu_{\Gamma^\prime_3})\colon \MA\rightsquigarrow\MB$ be interpretations. First of all, we note the stability of homotopies with respect to the replacement of interpretations with equal ones. The equivalence class of the interpretation $(\Gamma,\bar p,\mu_\Gamma)$ with respect to equality $=_{\rm int}$ will be denoted by $[\Gamma,\bar p,\mu_\Gamma]$.

\begin{remark}\label{remark:hom_eq}
If $(\theta)\colon (\Gamma_1,\bar p_1,\mu_{\Gamma_1})\to (\Gamma_2,\bar p_2,\mu_{\Gamma_2})$ is a strong homotopy, $(\Gamma_1,\bar p_1,\mu_{\Gamma_1})=_{\rm int}(\Gamma^\prime_1,\bar p^\prime_1,\mu_{\Gamma^\prime_1})$ and $(\Gamma_2,\bar p_2,\mu_{\Gamma_2})=_{\rm int}(\Gamma^\prime_2,\bar p^\prime_2,\mu_{\Gamma^\prime_2})$, then $(\theta)\colon (\Gamma^\prime_1,\bar p^\prime_1,\mu_{\Gamma^\prime_1})\to (\Gamma^\prime_2,\bar p^\prime_2,\mu_{\Gamma^\prime_2})$ is a strong homotopy as well. Thus, we can speak of a strong homotopy $(\theta)$ between the equivalence classes of interpretations $[\Gamma_1,\bar p_1,\mu_{\Gamma_1}]$ and $[\Gamma_2,\bar p_2,\mu_{\Gamma_2}]$.
\end{remark}

Second, we recall the notion of vertical composition for homotopies and establish its stability with respect to the equality of homotopies.

\begin{notation}
For strong homotopies $(\theta_{1,2})\colon (\Gamma_1,\bar p_1,\mu_{\Gamma_1})\to (\Gamma_2,\bar p_2,\mu_{\Gamma_2})$ and $(\theta_{2,3})\colon (\Gamma_2,\bar p_2,\mu_{\Gamma_2})\to (\Gamma_3,\bar p_3,\mu_{\Gamma_3})$ we will denote by $(\theta_{2,3})\circ(\theta_{1,2})$ the strong homotopy $(\theta_{1,3})\colon(\Gamma_1,\bar p_1,\mu_{\Gamma_1})\to (\Gamma_3,\bar p_3,\mu_{\Gamma_3})$  with $L(\MA)$\=/isomorphism, defined by the formula 
\begin{equation}\label{eq:theta}
       \theta_{1,3}(\bar x_1,\bar x_3)=\exists \,\bar x_2\:(\theta_{1,2}(\bar x_1,\bar x_2)\wedge \theta_{2,3}(\bar x_2,\bar x_3)). 
    \end{equation}
\end{notation}

It is clear that the homotopy isomorphism of $(\theta_{1,3})$ is the composition of the homotopy isomorphisms of $(\theta_{1,2})$ and $(\theta_{2,3})$.

\begin{fact}\label{fact:hat6}
Suppose that $(\theta_{1,2})\colon (\Gamma_1,\bar p_1,\mu_{\Gamma_1})\to (\Gamma_2,\bar p_2,\mu_{\Gamma_2})$,  $(\theta_{2,3})\colon (\Gamma_2,\bar p_2,\mu_{\Gamma_2})\to (\Gamma_3,\bar p_3,\mu_{\Gamma_3})$, $(\theta^\prime_{1,2})\colon (\Gamma^\prime_1,\bar p^\prime_1,\mu_{\Gamma^\prime_1})\to (\Gamma^\prime_2,\bar p^\prime_2,\mu_{\Gamma^\prime_2})$ and   $(\theta^\prime_{2,3})\colon (\Gamma^\prime_2,\bar p^\prime_2,\mu_{\Gamma^\prime_2})\to (\Gamma^\prime_3,\bar p^\prime_3,\mu_{\Gamma^\prime_3})$ are strong homotopies, such that $(\theta_{1,2})=_{\rm hom}(\theta^\prime_{1,2})$ and $(\theta_{2,3})=_{\rm hom}(\theta^\prime_{2,3})$. Then one has $(\theta_{2,3})\circ(\theta_{1,2})=_{\rm hom}(\theta^\prime_{2,3})\circ(\theta^\prime_{1,2})$.
\end{fact}

\begin{proof}
By definition, ${\bf \hat N}((\theta_{2,3})\circ(\theta_{1,2}))$ is a natural isomorphism between interpretation functors ${\bf\hat F}((\Gamma_1,\bar p_1,\mu_{\Gamma_1}))$ and ${\bf\hat F}((\Gamma_3,\bar p_3,\mu_{\Gamma_3}))$; and ${\bf \hat N}((\theta^\prime_{2,3})\circ(\theta^\prime_{1,2}))$ is a natural isomorphism between interpretation functors ${\bf\hat F}((\Gamma^\prime_1,\bar p^\prime_1,\mu_{\Gamma^\prime_1}))$ and ${\bf\hat F}((\Gamma^\prime_3,\bar p^\prime_3,\mu_{\Gamma^\prime_3}))$. Since $(\Gamma_1,\bar p_1,\mu_{\Gamma_1})=_{\rm int}(\Gamma_3,\bar p_3,\mu_{\Gamma_3})$ and $(\Gamma^\prime_1,\bar p^\prime_1,\mu_{\Gamma^\prime_1})=_{\rm int}(\Gamma^\prime_3,\bar p^\prime_3,\mu_{\Gamma^\prime_3})$, then by Fact~\ref{fact:hat1}, ${\bf\hat F}((\Gamma_1,\bar p_1,\mu_{\Gamma_1}))={\bf\hat F}((\Gamma_3,\bar p_3,\mu_{\Gamma_3}))$ and ${\bf\hat F}((\Gamma^\prime_1,\bar p^\prime_1,\mu_{\Gamma^\prime_1}))={\bf\hat F}((\Gamma^\prime_3,\bar p^\prime_3,\mu_{\Gamma^\prime_3}))$. Furthermore, connectors of the homotopies $(\theta_{2,3})\circ(\theta_{1,2})$ and $(\theta^\prime_{2,3})\circ(\theta^\prime_{1,2})$ define one the same $L(\MA)$\=/isomorphism, i.\,e., ${\bf \hat N}((\theta_{2,3})\circ(\theta_{1,2}))_A={\bf \hat N}((\theta^\prime_{2,3})\circ(\theta^\prime_{1,2}))_A$. Hence, by Lemma~\ref{cor:natur_iso}, one has ${\bf \hat N}((\theta_{2,3})\circ(\theta_{1,2}))={\bf \hat N}((\theta^\prime_{2,3})\circ(\theta^\prime_{1,2}))$, and by Fact~\ref{fact:hat3}, $(\theta_{2,3})\circ(\theta_{1,2})=_{\rm hom}(\theta^\prime_{2,3})\circ(\theta^\prime_{1,2})$.
\end{proof}

The equivalence class of the homotopy $(\theta)$ with respect to equality $=_{\rm hom}$ will be denoted by $[\theta]$. Thus, we can talk about strong homotopies $[\theta_{1}]\colon [\Gamma_1,\bar p_1,\mu_{\Gamma_1}]\to [\Gamma_1^\prime,\bar p_1^\prime,\mu_{\Gamma^\prime_1}]$, $[\theta_{2}]\colon [\Gamma_2,\bar p_2,\mu_{\Gamma_2}]\to [\Gamma_2^\prime,\bar p_2^\prime,\mu_{\Gamma^\prime_2}]$ and so on, and correctly define composition $[\theta_2]\circ [\theta_1]$ of them, if $[\Gamma_1^\prime,\bar p_1^\prime,\mu_{\Gamma^\prime_1}]=[\Gamma_2,\bar p_2,\mu_{\Gamma_2}]$.

Third, we establish stability of horizontal composition $\ast$ of homotopies with respect to the equality of homotopies.

\begin{fact}\label{prop3:par}
Suppose that $(\theta_{\MA,\MB}), (\theta^\prime_{\MA,\MB})$ are strong homotopies of interpretations of $\MA$ in $\MB$ and $(\theta_{\MB,\MC}), (\theta^\prime_{\MB,\MC})$ are strong homotopies of interpretations of $\MB$ in $\MC$. If $(\theta_{\MA,\MB})=_{\rm hom} (\theta^\prime_{\MA,\MB})$ and $(\theta_{\MB,\MC})=_{\rm hom} (\theta^\prime_{\MB,\MC})$, then $(\theta_{\MB,\MC}\ast\theta_{\MA,\MB})=_{\rm hom}(\theta^\prime_{\MB,\MC}\ast\theta^\prime_{\MA,\MB})$. 
\end{fact}

\begin{proof}
According to Fact~\ref{fact:hat3}, one has ${\bf \hat N}((\theta_{\MA,\MB}))={\bf \hat N}((\theta^\prime_{\MA,\MB}))$ and ${\bf \hat N}((\theta_{\MB,\MC}))= {\bf \hat N}((\theta^\prime_{\MB,\MC}))$. Therefore,  ${\bf \hat N}((\theta_{\MB,\MC}))\ast{\bf \hat N}((\theta_{\MA,\MB}))={\bf \hat N}((\theta^\prime_{\MB,\MC}))\ast{\bf \hat N}((\theta^\prime_{\MA,\MB}))$ and, due to Corollary~\ref{cor:ast_N}, ${\bf \hat N}((\theta_{\MB,\MC}\ast\theta_{\MA,\MB}))={\bf \hat N}((\theta^\prime_{\MB,\MC}\ast\theta^\prime_{\MA,\MB}))$. Hence, again by Fact~\ref{fact:hat3}, $(\theta_{\MB,\MC}\ast\theta_{\MA,\MB})=_{\rm hom}(\theta^\prime_{\MB,\MC}\ast\theta^\prime_{\MA,\MB})$.
\end{proof}

Thus, the horizontal composition $\ast$ is correctly defined on equivalence classes $[\theta]$ of strong homotopies of interpretations.

Fourth, we need to fix the connections between vertical compositions of natural isomorphisms and homotopies of interpretations. 

\begin{fact}\label{fact:hat5}
For any strong homotopies $(\theta_{1,2})\colon (\Gamma_1,\bar p_1,\mu_{\Gamma_1})\to (\Gamma_2,\bar p_2,\mu_{\Gamma_2})$ and $(\theta_{2,3})\colon (\Gamma_2,\bar p_2,\mu_{\Gamma_2})\to (\Gamma_3,\bar p_3,\mu_{\Gamma_3})$ one has ${\bf \hat N}((\theta_{2,3})\circ(\theta_{1,2}))={\bf \hat N}((\theta_{2,3}))\circ{\bf \hat N}((\theta_{1,2}))$. And for any natural isomorphisms $\eta_{1,2}\colon \F_1\to\F_2$, $\eta_{2,3}\colon \F_2\to\F_3$ between interpretation functors $\F_1,\F_2,\F_3\colon \PLS(\MA)\to \PLS(\MB)$ one has ${\bf H}(\eta_{2,3}\circ\eta_{1,2})={\bf  H}(\eta_{2,3})\circ{\bf  H}(\eta_{1,2})$.
\end{fact}

\begin{proof}
Since both categorical isomorphisms ${\bf \hat N}((\theta_{2,3})\circ(\theta_{1,2}))_A$ and $({\bf \hat N}((\theta_{2,3}))\circ{\bf \hat N}((\theta_{1,2})))_A$ are defined by the same formula $\theta_{1,3}$~\eqref{eq:theta}, and both ${\bf \hat N}((\theta_{2,3})\circ(\theta_{1,2}))$ and ${\bf \hat N}((\theta_{2,3}))\circ{\bf \hat N}((\theta_{1,2}))$ are natural isomorphisms between the interpretation functors ${\bf \hat F}((\Gamma_1,\bar p_1,\mu_{\Gamma_1}))$ and ${\bf \hat F}((\Gamma_3,\bar p_3,\mu_{\Gamma_3}))$, then by Lemma~\ref{cor:natur_iso}, one has ${\bf \hat N}((\theta_{2,3})\circ(\theta_{1,2}))={\bf \hat N}((\theta_{2,3}))\circ{\bf \hat N}((\theta_{1,2}))$. 

Let $(\Gamma_i,\bar p_i,\mu_{\Gamma_i})={\bf I}(\F_i)$, $i=1,2,3$, and $(\theta_{1,2})={\bf H}(\eta_{1,2})$, $(\theta_{2,3})={\bf H}(\eta_{2,3})$. Then both ${\bf H}(\eta_{2,3}\circ\eta_{1,2})$ and ${\bf  H}(\eta_{2,3})\circ{\bf  H}(\eta_{1,2})$ are strong homotopies between interpretations $(\Gamma_1,\bar p_1,\mu_{\Gamma_1})$ and $(\Gamma_3,\bar p_3,\mu_{\Gamma_3})$ with the same connector $\theta_{1,3}$~\eqref{eq:theta}, i.\,e., ${\bf H}(\eta_{2,3}\circ\eta_{1,2})={\bf  H}(\eta_{2,3})\circ{\bf  H}(\eta_{1,2})$.  
\end{proof}

Fifth, natural isomorphisms satisfy the interchange law~\cite[Theorem~1, \S II.5]{MacLane}; we show that its analog holds for homotopies of interpretations.

\begin{proposition}[interchange law]\label{prop4:par}
Suppose that $(\theta_{\MA,\MB}^{1,2})\colon (\Gamma_1,\bar p_1,\mu_{\Gamma_1})\to (\Gamma_2,\bar p_2,\mu_{\Gamma_2})$ and $(\theta_{\MA,\MB}^{2,3})\colon (\Gamma_2,\bar p_2,\mu_{\Gamma_2})\to (\Gamma_3,\bar p_3,\mu_{\Gamma_3})$ are strong homotopies of interpretations of $\MA$ in $\MB$, and $(\theta_{\MB,\MC}^{1,2})\colon (\Delta_1,\bar q_1,\mu_{\Delta_1})\to (\Delta_2,\bar q_2,\mu_{\Delta_2})$ and $(\theta_{\MB,\MC}^{2,3})\colon (\Delta_2,\bar q_2,\mu_{\Delta_2})\to (\Delta_3,\bar q_3,\mu_{\Delta_3})$ are strong homotopies of interpretations of $\MB$ in $\MC$. Then one has 
$$
(\theta_{\MB,\MC}^{2,3}\ast\theta_{\MA,\MB}^{2,3})\circ(\theta_{\MB,\MC}^{1,2}\ast\theta_{\MA,\MB}^{1,2})=_{\rm hom}((\theta_{\MB,\MC}^{2,3})\circ(\theta_{\MB,\MC}^{1,2}))\ast((\theta_{\MA,\MB}^{2,3})\circ(\theta_{\MA,\MB}^{1,2})).
$$
\end{proposition}

\begin{proof}
Indeed, by Corollary~\ref{cor:ast_N} and Fact~\ref{fact:hat5}, one has ${\bf \hat N}((\theta_{\MB,\MC}^{2,3}\ast\theta_{\MA,\MB}^{2,3})\circ(\theta_{\MB,\MC}^{1,2}\ast\theta_{\MA,\MB}^{1,2}))={\bf \hat N}((\theta_{\MB,\MC}^{2,3}\ast\theta_{\MA,\MB}^{2,3}))\circ{\bf \hat N}((\theta_{\MB,\MC}^{1,2}\ast\theta_{\MA,\MB}^{1,2}))=({\bf \hat N}((\theta_{\MB,\MC}^{2,3}))\ast{\bf \hat N}((\theta_{\MA,\MB}^{2,3})))\circ({\bf \hat N}((\theta_{\MB,\MC}^{1,2}))\ast{\bf \hat N}((\theta_{\MA,\MB}^{1,2})))=({\bf \hat N}((\theta_{\MB,\MC}^{2,3}))\circ{\bf \hat N}((\theta_{\MB,\MC}^{1,2})))\ast({\bf \hat N}((\theta_{\MA,\MB}^{2,3}))\circ{\bf \hat N}((\theta_{\MA,\MB}^{1,2})))={\bf \hat N}(((\theta_{\MB,\MC}^{2,3})\circ(\theta_{\MB,\MC}^{1,2}))\ast((\theta_{\MA,\MB}^{2,3})\circ(\theta_{\MA,\MB}^{1,2})))$. Therefore, by Fact~\ref{fact:hat3}, we obtain the required.  
\end{proof}

Let us denote by $2$\=/${\bf AS}$ the strict $2$\=/category ${\bf AS}$ with strong homotopies up to equality $=_{\rm hom}$ as $2$\=/cells and $\circ,\ast$ as vertical and horizontal compositions. 

\begin{fact}
    $2$\=/${\bf AS}$ is a strict $2$\=/category.
\end{fact}

\begin{proof}
By Remark~\ref{remark:hom_eq}, $2$\=/cells are correctly defined in $2$\=/${\bf AS}$. Vertical and horizontal compositions $\circ,\ast$ on $2$\=/cells are correctly defined due to Facts~\ref{fact:hat6} and~\ref{prop3:par}. It is easy to see that for every $0$\=/cells $\MA$ and $\MB$ the set ${\rm Hom}(\MA,\MB)$ of $1$\=/cells together with the vertical compositions is a category. Further, $0$\=/cells as objects and $2$\=/cells as morphisms together with the horizontal compositions form a category due to Proposition~\ref{ABCD_hom}. And by Proposition~\ref{prop4:par} one has the interchange law.
\end{proof}

\begin{theorem}\label{th:2cat}
The $2$\=/categories $2$\=/${\bf AS}$ and $2$\=/${\bf PLG}$ are $2$\=/isomorphic. 
\end{theorem}

\begin{proof}
We continue the argument of the proofs of Theorems~\ref{TH} and~\ref{TH_sim}. For any algebraic structures $\MA$ and $\MB$ the categories ${\rm Hom}(\MA,\MB)$ and ${\rm Hom}(\PLS(\MA),\PLS(\MB))$ of $1$\=/cells together with the vertical compositions $\circ$ are isomorphic due to Facts~\ref{fact:hat4}, \ref{fact:hat5} and Lemma~\ref{le:equal}. The functors ${\bf \hat N}$ and ${\bf \hat H}$, providing mappings on morphisms in these categories, are consistent with respect to horizontal composition $\ast$, by Corollary~\ref{cor:ast_N}.
\end{proof}

\end{document}